\documentclass[12pt,a4paper,reqno]{amsart}

\usepackage{bbm}

\usepackage{amsmath,amssymb,amsthm,amscd,mathrsfs}

\usepackage[english]{babel}
\usepackage[mathscr]{eucal}

\usepackage[pdftex]{color,graphicx}

\usepackage[pdftex,bookmarks=true,colorlinks=true, pdfstartview=FitV, linkcolor=black,  citecolor=black, urlcolor=black,pdfstartview=FitH]{hyperref}
 \usepackage{stmaryrd}
\usepackage{color}
\usepackage[applemac]{inputenc} 
\usepackage{fancyhdr}
\usepackage{epsfig}
\usepackage{cite}
\usepackage{graphicx}
\usepackage{latexsym}
\usepackage{bm}
\usepackage{mathtools}				
\usepackage{esdiff}					
\usepackage{epic}
\usepackage{multicol}
\usepackage{supertabular} 
\usepackage{listings}
\usepackage{verbatim}
\usepackage{turnstile} 

\usepackage{verbatim}

\usepackage{tikz}
\usetikzlibrary{arrows.meta}				
\usetikzlibrary{decorations.markings}		         
\usepackage{tikz-dimline}					

\usepackage{xcolor}
\usetikzlibrary{arrows.meta,calc,decorations.pathreplacing,positioning}

\definecolor{trajblue}{RGB}{35,92,155}
\definecolor{distorange}{RGB}{214,126,48}
\definecolor{regiongray}{RGB}{105,105,105}
\definecolor{stripblue}{RGB}{222,235,247}
\definecolor{cutgreen}{RGB}{61,139,94}
\definecolor{accentred}{RGB}{172,54,54}
\definecolor{shellpurple}{RGB}{112,78,154}

\usepackage{dsfont}

\font\sc=rsfs10 at 12pt

\numberwithin{equation}{section}

\newcommand{\x}{\xi}

\newcommand{\w}{{\text{w}}}

\newcommand{\C}{{\mathbb C}}
\newcommand{\Dl}{{\mathbb D}}
\newcommand{\Pl}{{\mathbb P}}
\newcommand{\R}{{\mathbb R}}

\newcommand{\N}{{\mathbb N}}

\newcommand{\mbf}[1]{\protect\mbox{\boldmath$#1$\unboldmath}}
\newcommand{\lmbf}[1]{\protect\mbox{{\scriptsize\boldmath$#1$\unboldmath}}}

\newcommand{\bb}{\mbf{b}}

\newcommand{\db}{\mbf{d}}

\newcommand{\fb}{\mbf{f}}
\newcommand{\gb}{\mbf{g}}
\newcommand{\hb}{\mbf{h}}

\newcommand{\pb}{\mbf{p}}

\newcommand{\rb}{\mbf{r}}

\newcommand{\ub}{\mbf{u}}

\newcommand{\wb}{\mbf{w}}
\newcommand{\xb}{\mbf{x}}

\newcommand{\Ab}{\mbf{A}}

\newcommand{\Bb}{\mbf{B}}

\newcommand{\Cb}{\mbf{C}}
\newcommand{\Db}{\mbf{D}}
\newcommand{\Eb}{\mbf{E}}
\newcommand{\Fb}{\mbf{F}}
\newcommand{\Gb}{\mbf{G}}

\newcommand{\Ib}{\mbf{I}}

\newcommand{\Kb}{\mbf{K}}
\newcommand{\Lb}{\mbf{L}}

\newcommand{\Qb}{\mbf{Q}}
\newcommand{\Rb}{\mbf{R}}

\newcommand{\Tb}{\mbf{T}}
\newcommand{\Ub}{\mbf{U}}
\newcommand{\Vb}{\mbf{V}}

\newcommand{\DF}{\mathfrak D}

\newcommand{\Ac}{{\mathcal A}}
\newcommand{\Bc}{{\mathcal B}}

\newcommand{\Dc}{{\mathcal D}}

\newcommand{\Hc}{{\mathcal H}}

\newcommand{\Kc}{{\mathcal K}}

\newcommand{\MC}{{\mathcal M}}

\newcommand{\Oc}{{\mathcal O}}
\newcommand{\Pc}{{\mathcal P}}

\newcommand{\Rc}{{\mathcal R}}

\newcommand{\Tc}{{\mathcal T}}
\newcommand{\Uc}{{\mathcal U}}

\newcommand{\Ssf}{{\mathsf S}}

\newcommand{\supp}{\hbox{{\rm supp}}\,}

\DeclareMathOperator{\im}{{\rm Im}\,}
\DeclareMathOperator{\re}{{\rm Re}\,}
\DeclareMathOperator{\rank}{rank}

\newcommand{\Res}{\operatorname{Res\,}}

\newcommand{\Dom}{\operatorname{Dom\,}}

\newcommand\Item[1][]{%
	\ifx\relax#1\relax  \item \else \item[#1] \fi
	\abovedisplayskip=0pt\abovedisplayshortskip=0pt~\vspace*{-\baselineskip}}

\newcommand{\tS}{\sc\mbox{S}\hspace{1.0pt}} 

\newcommand{\Fr}{\sc\mbox{F}\hspace{1.0pt}} 

\newcommand{\kuma}[1]{\langle {#1} \rangle}

\newcommand{\sm}{\operatorname{spec}\,}
\newcommand{\smd}{\operatorname{spec}_{\operatorname{d}}\,}
\newcommand{\sme}{\operatorname{spec}_{\operatorname{ess}}\,}

\newtheorem{theorem}{Theorem}[section]
\newtheorem{proposition}[theorem]{Proposition}
\newtheorem{lemma}[theorem]{Lemma}
\newtheorem{corollary}[theorem]{Corollary}

\theoremstyle{definition}
\newtheorem{definition}[theorem]{Definition}
\newtheorem{assumption}[theorem]{Assumption}

\theoremstyle{remark}
\newtheorem{remark}[theorem]{Remark}

\begin{document}

	\tikzset{->-/.style={decoration=
				{
				markings,
				mark=at position #1 with
				{\arrow[line width = 0.6mm]{stealth}}
				},
				postaction={decorate}
				}}


\title{Resonance expansions and local-energy decay estimates for Dirac operators} 

\author{Zhuo Chen}

\author{Michael Melgaard} 
\address[Z.~Chen,  M.~Melgaard]{Department of Mathematics \\ School of Mathematical  and Physical Sciences \\
                                     University of Sussex      \\
                                     BN1 9QH Brighton, United Kingdom} 
\email{m.melgaard@sussex.ac.uk (corresponding author)}

\date{8 September 2026}

\begin{abstract}
We study resonance expansions and localised long-time dynamics for three-dimensional semiclassical Dirac operators. 
For smooth Hermitian matrix-valued perturbations which are analytic outside a compact set, we first obtain exponential 
cut-off resolvent bounds away from the resonance set and polynomial bounds in thin resonance-free rectangles adjacent 
to the real axis. Combining these estimates with a local $\Oc(h^{-3})$ upper bound on the number of resonances, 
a pigeonhole selection of admissible contours, almost-analytic functional calculus, and a Cauchy--Green deformation, 
we derive resonance expansions for spectrally localised propagators. Positive-energy resonances contribute for large 
positive times, whereas negative-energy resonances contribute for large negative times. We then specialise to 
mass-type perturbations $\mbf{\beta}V(x)$ in a nontrapping energy window. Using the logarithmic resonance-free 
region established in our earlier work together with a quantitative FBI-transform/escape-function argument, we obtain a 
polynomial bound for the continued cut-off resolvent and deduce rapid local-energy decay. We subsequently refine this 
polynomial estimate by a finite-time propagation argument in the undistorted region. Introducing the maximal positive-sheet connection time $T_\chi$ associated with the spatial cutoff,  we prove a geometric continued-resolvent estimate
and,  for every fixed $M$ below the resonance-free depth constant, derive the explicit local-energy decay threshold $T_N>T_\chi+(N+1)/M$. The same analysis gives a rigorous microlocal obstruction to a uniform 
$[h\log(1/h)]^{-1}$ bound for the full distorted resolvent and identifies a finite-flight-time amplification mechanism for the sandwiched continued resolvent at logarithmic depth. 

\textit{Keywords}: Semiclassical Dirac operator; scattering resonances; resonance expansions; cut-off resolvent estimates; local-energy decay; complex distortion; nontrapping dynamics; FBI transform; finite-time propagation \\
\textit{2020 Mathematics Subject Classification.} Primary: 35P25. Secondary: 35Q41, 81Q20, 81U24, 47A40, 35B40 \\ 
\end{abstract}

\maketitle

\section{Introduction}
\label{ZCMM2025b:SecIntro}
Scattering resonances provide a natural link between the spectral theory of non-self-adjoint deformations and the long-time behaviour of quantum evolution.  In semiclassical scattering, resonance expansions express a spectrally localised propagator as a sum of contributions associated with poles of a meromorphically continued resolvent, together with a controlled remainder.  For Schr\"odinger operators, such expansions were developed by Burq and Zworski~\cite{BurqZworski2001} and, in the presence of potential barriers, by Nakamura, Stefanov, and Zworski~\cite{NSZwor2003}.  Related expansions occur for scattered waves~\cite{TangZworski2000}, Stark Hamiltonians~\cite{Kameoka2021}, and, more recently, discrete-time unitary models such as quantum walks~\cite{KHHMES2024}.

For Dirac operators the corresponding theory is less complete.  The operator is first-order and matrix-valued, is not semibounded, and has two continuous spectral branches separated by the thresholds $\pm mc^2$.  Parisse~\cite{Parisse1992a,Parisse1992b} adapted semiclassical resonance methods to Dirac operators and studied shape resonances.  Khochman~\cite{Khochman2007}, using Hunziker's complex distortion method~\cite{Hunziker1986}, established a spectral-shift expansion and local upper bounds on the number of resonances.  Kungsman and Melgaard~\cite{JKMM20214c} subsequently proved lower bounds of order $h^{-3}$ for scalar electric potentials near shifted extremal energies.  These results motivate the study of resonance distributions and propagator expansions for Dirac operators under non-globally analytic perturbations.

We consider the three-dimensional semiclassical Dirac operator
\begin{equation*}
  \Dl(h)=-ich\,\mbf{\alpha}\cdot\nabla+\mbf{\beta} mc^2+\Vb(x),
\end{equation*}
where $\mbf{\alpha}=(\mbf{\alpha}_1,\mbf{\alpha}_2,\mbf{\alpha}_3)$ and $\mbf{\beta}$ are the Dirac matrices, $m>0$ is the particle mass, 
$c$ is the speed of light, and $h>0$ is the semiclassical parameter.  The associated evolution equation is
\begin{equation}
  ih\,\partial_t\psi(t,x)=\Dl(h)\psi(t,x),\qquad  \psi(0,x)=\psi(x) \in \Dom(\Dl(h))=H^1(\R^3;\C^4),
\label{zcmm2025b:res_exp_Dirac_eqs}
\end{equation}
and its solution is $\psi(t)=e^{-it\Dl(h)/h}\psi$.  If $E$ is an eigenvalue, the corresponding stationary component is $e^{-itE/h}\psi$.  A resonance $z$, by contrast, contributes a factor $e^{-itz/h}$.  Consequently, poles on the positive-energy continued sheet in the lower half-plane decay for $t>0$, whereas the reflected negative-energy poles in the upper half-plane decay for $t<0$.  This asymmetry in time is one of the characteristic features of the Dirac resonance expansions proved below.

Our first set of main results, Theorems~\ref{ZCMM2025b:MainThm2-pos} and~\ref{ZCMM2025b:MainThm2-neg}, concerns a smooth Hermitian $4\times4$ matrix-valued potential which is long-range and analytic outside a compact set.  The proof has three principal ingredients.  First, the local resonance-counting estimate of Theorem~\ref{ZCMM2025b:Thm2dot6Khochman} gives at most $Ch^{-3}$ resonances in a fixed admissible energy domain.  Second, Theorem~\ref{ZCMM2025b:Thm-resol-est-away-critset} and Proposition~\ref{ZCMM2025b:Lemma-SMP} provide, respectively, an exponential cut-off resolvent estimate away from the poles and a polynomial estimate in thin resonance-free rectangles adjacent to the real axis.  Third, Lemma~\ref{ZCMM2025b:pigeonhole-admissible-contour} uses the $\Oc(h^{-3})$ counting bound to select the vertical and horizontal pieces of the contour at polynomial distance from every resonance.  Almost-analytic functional calculus and a Cauchy--Green deformation then yield the propagator expansions.  For every fixed $L>M+7$, the positive-energy expansion holds uniformly for $t\geq h^{-L}$ and the negative-energy expansion uniformly for $t\leq-h^{-L}$.

Our second group of results treats the special mass-type perturbation $\Vb(x)=\mbf{\beta} V(x)$.  For such potentials our earlier work~\cite{ZCMM2025a} established a resonance-free strip of depth $Ch\log(1/h)$ around a nontrapping energy interval.  In Section~\ref{zcmm2025b-sec-masstype-resexp} we first extract a quantitative consequence of the same FBI-transform/escape-function mechanism: on every fixed smaller logarithmic strip the continued cut-off resolvent satisfies a polynomial bound in $h^{-1}$.  A contour deformation in this pole-free region gives the local-energy decay statement of Theorem~\ref{ZCMM2025b:MainThm1-masstype}: for each prescribed $N$ the spectrally localised propagator is $\Oc(h^N)$ after a time threshold $T_N$, with the reflected statement at negative energies.

Section~\ref{ZCMM2026new:SecFiniteTime} then gives a geometric refinement of this mass-type estimate.  
For the full distorted resolvent, a uniform $C/[h\log(1/h)]$ bound is ruled out at generic exposing transverse crossings: an exposing complex distortion forces a sign change of the imaginary part of the actual positive 
branch along outgoing trajectories, producing an $\Oc(h)$ quasimode and hence an $h^{-1}$ 
lower bound for the resolvent norm.  For the sandwiched continued resolvent we do not claim such a 
general lower bound; instead, the local flow-box and free Dirac calculations identify the additional 
finite-flight-time amplification factor $e^{sT/h}$ at $z=E-is$.  This motivates a finite-time propagation construction and explains the polynomial loss in the geometric upper bound proved below.

To quantify this geometry, we introduce the maximal positive-sheet connection time $T_\chi$ between the microlocal interaction regions determined by the cutoff.  Nontrapping implies $T_\chi<\infty$, and a uniform no-return argument controls the outgoing resolvent tail after any fixed time $T>T_\chi$.  The resulting Theorem~\ref{ZCMM2026new:TheoremGeometricResolvent} gives, for every $\varepsilon>0$,
\[
 \|\Rb^{\rm cont}_{\chi,\psi,+}(E-is;h)\|_{L^2\to L^2}
 \le C_\varepsilon\frac{\exp\!\big(s(T_\chi+\varepsilon)/h\big)-1}{s},
 \qquad 0<s\le M h\log(1/h).
\]
At the bottom of the logarithmic strip this yields the explicit bound 
$[M\gamma(h)]^{-1}h^{-M(T_\chi+\varepsilon)}$, $\gamma(h)=h\log(1/h)$, 
replacing the unspecified exponent coming from the FBI-space norm comparison.  Inserting 
the geometric bound into the same Cauchy--Green deformation leads to Corollary~\ref{ZCMM2026new:CorGeometricDecay}: up to an arbitrarily small fixed margin, one may take
\[
 T_N>T_\chi+\frac{N+1}{M}.
\]
Thus Sections~3--4 describe dynamics in terms of resonant contributions, Section~5 gives rapid localised decay in a nontrapping resonance-free window by the FBI/escape-function method, and Section~6 refines the quantitative decay mechanism in terms of a classical flight-time invariant.

Mass-type Dirac operators also arise in models with a spatially varying mass gap, including effective descriptions of Dirac materials and related condensed-matter systems; see, for example,~\cite{Bastard1988,Bagchi2022,Kiessling2025}.  In the present paper, however, the role of the mass-type assumption is mathematical: it is precisely the structure used in~\cite{ZCMM2025a} to obtain the sign needed for the logarithmic resonance-free region.

\subsection*{Organization of the paper} 
Section~\ref{ZCMM2025b:SecPrelim} develops the spectral and complex-distortion framework and records the local $\Oc(h^{-3})$ resonance-counting bound.  Section~\ref{ZCMM2025b:SecResolvEst} proves cut-off resolvent estimates near the real axis: first an exponential bound away from the resonance set and then polynomial estimates in thin resonance-free rectangles by a three-lines/semiclassical maximum-principle argument.  Section~\ref{ZCMM2025b:SecRE2} uses these estimates, the pigeonhole construction of admissible contours, almost-analytic functional calculus, and Cauchy--Green deformation to prove the positive- and negative-energy resonance expansions of Theorems~\ref{ZCMM2025b:MainThm2-pos} and~\ref{ZCMM2025b:MainThm2-neg}.  Section~\ref{zcmm2025b-sec-masstype-resexp} specialises to mass-type potentials, derives a polynomial continued-resolvent bound in a logarithmic resonance-free strip from the FBI/escape-function construction of~\cite{ZCMM2025a}, and proves the localised propagation/local-energy decay statement of Theorem~\ref{ZCMM2025b:MainThm1-masstype}.  Section~\ref{ZCMM2026new:SecFiniteTime} refines this quantitative analysis by replacing the unspecified FBI-space polynomial loss in the final resolvent bound with a finite-time propagation estimate in the undistorted region; the FBI estimate is retained only as a coarse $h$-temperedness input for the outgoing tail. It first identifies the obstruction to a pure $[h\log(1/h)]^{-1}$ resolvent bound, then introduces the positive-sheet connection time $T_\chi$, proves a uniform outgoing no-return estimate, and derives the geometric cut-off resolvent bound of Theorem~\ref{ZCMM2026new:TheoremGeometricResolvent}.  The section concludes by inserting this bound into the Cauchy--Green deformation and obtaining the explicit decay threshold of Corollary~\ref{ZCMM2026new:CorGeometricDecay}.  Appendix~\ref{ZCMM2025b:appendix-inhomogeneous-FBI} gives the detailed derivation of the inhomogeneous FBI estimate used in Section~\ref{zcmm2025b-sec-masstype-resexp}.

\section{Preliminaries}
\label{ZCMM2025b:SecPrelim}
We first review some fundamental facts about the Dirac operator, and then we define resonances 
for the Dirac operator with non-globally analytic potentials, following the approach developed by Hunziker \cite{Hunziker1986}. 

The free, semiclassical Dirac operator, describing the motion of a relativistic electron or positron without external forces, is the unique self-adjoint extension of the symmetric operator defined on 
$C_{0}^{\infty} (\R^{3}; \C^{4})$ in the Hilbert space $\Hc = L^{2}(\R^3;\C^{4})$ by
\begin{equation*}
\Dl_0(h) = c\sum_{j=1}^3 \mbf{\alpha}_j D_j +\mbf{\beta} mc^2
\end{equation*}
where $D_j:=\frac{h}{i}\partial_{x_j}=-ih\partial_{x_j}$, $m>0$ is the mass of the particle, $c$ is the speed of light, the parameter $h>0$ is the semiclassical parameter (playing the role of the reduced Planck constant), and the $4\times 4$ Hermitian matrices 
$\mbf{\alpha}_1 ,\mbf{\alpha}_2, \mbf{\alpha}_3 $ and $\mbf{\beta}$ are given by
\begin{equation}
\mbf{\alpha}_i=\left(
	\begin{matrix}
		0 & \mbf{\sigma}_i\\
		\mbf{\sigma}_i & 0
	\end{matrix}\right),   \qquad \mbf{\beta}=\left(\begin{matrix}
		\Ib_2 & 0\\
		0 & -\Ib_2
	\end{matrix}\right), 
\end{equation}
where $\Ib_N$ is the $N\times N$ identity matrix, and $\mbf{\sigma}_i$ are the $2\times 2$ Pauli matrices defined by
\begin{equation*}
	\mbf{\sigma}_1=\left(
	\begin{matrix}
		0 & 1\\
			1 & 0
		\end{matrix}\right), \,\,
		\mbf{\sigma}_2=\left(
		\begin{matrix}
			0 & -i\\
			i & 0
		\end{matrix}\right),\,\,
		\mbf{\sigma}_3=\left(
	\begin{matrix}
			1 & 0\\
			0 & -1
		\end{matrix}\right).
\end{equation*}
These matrices satisfy the following algebraic relations:
\begin{equation}
\label{zcmm2025d:Dirac-matrix}
\begin{cases*}
	\mbf{\alpha}_i  \mbf{\alpha}_j + \mbf{\alpha}_j  \mbf{\alpha}_i =2 \delta_{ij} \Ib_4\\
	\mbf{\alpha}_i \mbf{\beta} + \mbf{\beta}  \mbf{\alpha}_i=0\\
	\mbf{\beta}^2=\Ib_4
\end{cases*}
\end{equation}
where $ \delta_{ij}$ is the Kronecker symbol. We introduce 
\begin{equation*}
	\mbf{\mathcal{W}}= \ub(\x) \Fr, \quad \ub(\x)= \frac{(mc^2+\lambda(\x))\cdot\Ib_4 +c\mbf{\beta }\mbf{\alpha} \cdot \x }{\sqrt{2\lambda(\x)(mc^2+\lambda(\x)) } },
\end{equation*}
where $ \lambda(\x)=\sqrt{c^2\x^2 +m^2c^4}$, and $\Fr$ is  the  Fourier transform and $\Fr^{-1}$ denotes its inverse.  Then, under the Foldy-Wouthuysen transformation 
$\Ub_{FW}= \Fr^{-1}\mbf{\mathcal{W}}$, the free Dirac operator can be diagonalised as 
\begin{equation*}
	\Ub_{FW} \Dl_0(h) \Ub_{FW}^{-1} =\left(
		\begin{matrix}
			\sqrt{-c^2h^2\Delta +m^2c^4}\cdot \Ib_2 &  0 \\
			0 & -\sqrt{-c^2h^2\Delta +m^2c^4}\cdot \Ib_2 
		\end{matrix}
		\right).
\end{equation*}
From this diagonalization, it follows directly that the spectrum of the free Dirac operator coincides with that of $\Ub_{FW} \Dl_0(h) \Ub_{FW}^{-1}$.  It is straightforward to 
deduce that the spectrum of the free Dirac operator equals 
$\sm(\Dl_0(h))= \sme(\Dl_0(h))= (-\infty, -mc^2] \cup [mc^2,+\infty)$.

\subsection*{Perturbed Dirac operator}
To describe the interaction of a particle with external fields, we perturb
$\Dl_{0}(h)$ by a potential $\Vb \in C^\infty(\R^3; \MC_4(\C))$, viewed as a multiplication operator on $\Hc$; 
$\MC_4(\C)$ denotes the algebra of $4\times4$ matrices with complex entries. In this paper, the potential $\Vb$ is assumed to satisfy the following conditions.

\begin{assumption}\label{ZCMM2025b:MainAssumpV}
The potential $\Vb \in C^\infty(\R^3; \MC_4(\C))$ is a smooth, Hermitian, $4\times4$ 
matrix-valued function and satisfies:
\begin{enumerate}
\item[(H1)] There exist constants $0<\varepsilon<1$ and $R_0>0$ such that
$\Vb(x)$ extends analytically to the sector
\[
W_{\varepsilon}
:=\bigl\{ z\in\C^3 : |\im z|\le \varepsilon |\re z|,\ |\re z|>R_0 \bigr\}.
\]

\item[(H2)] There exists $\delta>0$ such that, for every multiindex $\alpha\in\N^3$ and every $0<\varepsilon'<\varepsilon$, the analytic continuation satisfies
\[
\|\partial_z^\alpha \Vb(z)\|_{\Bc(\C^4)}
\le C_{\alpha,\varepsilon'}\langle \re z\rangle^{-\delta-|\alpha|},
\qquad z\in W_{\varepsilon'},
\]
where $\langle y\rangle=(1+|y|^2)^{1/2}$. In particular, on the real axis,
\[
\|\partial_x^\alpha \Vb(x)\|_{\Bc(\C^4)}
\le C_\alpha\langle x\rangle^{-\delta-|\alpha|}.
\]
\end{enumerate}
\end{assumption}

Under Assumption~\ref{ZCMM2025b:MainAssumpV}, the potential $\Vb$ is
relatively $\Dl_0(h)$-compact, and the operator
\[
\Dl(h) := \Dl_0(h)+\Vb
\]
is self-adjoint on
$\Dom(\Dl(h))=\Dom(\Dl_0(h))=H^1(\R^3;\C^4)$.
By Weyl's theorem, we have
\[
\sme(\Dl(h))
=\sme(\Dl_0(h))
=(-\infty,-mc^2]\cup[mc^2,\infty).
\]
\subsection*{Weyl Quantization}
In the following, we introduce the Weyl quantization for a matrix-valued symbol. 
Let $\Ab(x,\xi)\in \tS(T^*\R^3;\MC_4(\C)) \simeq \tS(\R^6;\MC_4(\C))$ ($\tS$ being the Schwartz space of rapidly decreasing functions) and let $\ub\in\tS(\R^3;\C^4)$. Its Weyl quantization is defined by
\begin{equation*}
(\Ab^\w\ub)(x)=(2\pi h)^{-3}\int_{\R^3_y\times\R^3_\xi}
 e^{i(x-y)\cdot\xi/h}\Ab\!\left(\frac{x+y}{2},\xi\right)\ub(y)\,dy\,d\xi .
\end{equation*}
The same definition extends to symbols in $C^\infty(T^*\R^3;\MC_4(\C))$ satisfying the usual semiclassical symbol estimates. In particular, the perturbed Dirac operator $\Dl(h)$ has symbol
\begin{equation*}
\db(x,\xi)=c\,\mbf{\alpha}\cdot\xi+\mbf{\beta}mc^2+\Vb(x),
\end{equation*}
which belongs to the symbol class $\Ssf(\kuma{\x})$; see \cite{ZworskiBook2012} for definitions. 

 \subsection*{Complex distortion and resonances}
In the spirit of the Aguilar--Balslev--Combes theory of complex scaling
\cite{Aguilar1971,BalslevCombes1971}, we summarize the spectral deformation
theory for the Dirac operator, following Hunziker \cite{Hunziker1986} and
Khochman \cite{Khochman2007}.

Fix $r_\epsilon>0$ so small that 
\[
0<r_\epsilon<\min\Big\{\frac{\epsilon}{\sqrt{1+\epsilon^2}},\frac{1}{2M_{g}}\Big\},
\qquad \Theta_\epsilon:=\{\theta\in\C:|\theta|<r_\epsilon\} 
\]
and define the distortion map $\mbf{\phi}_\theta(x)=x+\theta\,\gb(x)$,
where we choose a vector field $\gb\in C^\infty(\R^3,\R^3)$ such that 
\[
\gb(x)=
\begin{cases}
0, & |x|\le R_0,\\
x, & |x|\ge 2R_0,
\end{cases}
\qquad M_{g}:=\sup_{x\in\R^3}\|D\gb(x)\|<\infty 
\]
and $R_{0}> 0$ was introduced in Assumption~\ref{ZCMM2025b:MainAssumpV}.  
For real $\theta$ with $|\theta|$ small enough, $\mbf{\phi}_\theta$ is a global diffeomorphism of 
$\R^3$. Then, for $\theta\in \R$, we define the unitary operator 
$\Ub_\theta:\mathcal \tS(\R^3,\C^4)\to\mathcal \tS(\R^3,\C^4)$ by
\begin{equation}
\Ub_\theta \fb(x)
= J_{\lmbf{\phi}_\theta}(x)^{1/2}\,\fb(\mbf{\phi}_\theta(x)),
\qquad
\Ub_\theta^{-1}\fb(x)
= J_{\lmbf{\phi}_\theta^{-1}}(x)^{1/2}\,\fb(\mbf{\phi}_\theta^{-1}(x)),
\label{zcmm2025b:eqs_unitarytrans}
\end{equation}
where $J_{\lmbf{\phi}_\theta}(x)=\det D\mbf{\phi}_\theta(x)$ and $D\mbf{\phi}_\theta(x)$ is the Jacobian matrix of 
$\mbf{\phi}_\theta$. Following Hunziker \cite{Hunziker1986}, Khochman \cite{Khochman2007}, and Chen--Melgaard \cite{ZCMM2025a}, let $\mathcal A$ be the space of $\C^4$-valued entire functions $f=(f_j)_{j=1}^4$ such that, for every $0<\varepsilon'<\varepsilon$ and every $N\in\N$,
\[
\sup_{z\in W_{\varepsilon'}}\langle z\rangle^N |f_j(z)|<\infty,
\qquad j=1,\dots,4,
\]
and let $\Bc \subset L^2(\R^3;\C^4)$ be the set of restrictions to $\R^3$ of functions in $\Ac$. The set $\Bc$ is dense in $L^2(\R^3;\C^4)$; in particular, the standard analytic-vector construction contains a dense family of Hermite-type vectors. For 
$f\in \Bc$, the map $\theta\mapsto\Ub_\theta f$, initially defined for real $\theta$, extends holomorphically to $\Theta_\epsilon$, and $\Ub_\theta\Bc$ is dense for 
$\theta\in\Theta_\epsilon$. Consequently, the coefficient formula obtained below for real $\theta$ extends analytically to complex $\theta$ and defines a holomorphic family of type~(A) \cite{Kato1976}, with common domain $H^1(\R^3;\C^4)$. Throughout the sequel, $\Dl_{0,\theta}(h)$ and 
$\Dl_\theta(h)$ for complex $\theta$ refer to these holomorphic continuations. Identities involving $\Ub_\theta$ for non-real $\theta$ are understood first on analytic vectors and then by analytic continuation; they are not unitary-conjugation identities on $L^2$. The distorted Dirac operators are defined by
\[
\Dl_{0,\theta}(h):=\Ub_\theta \Dl_0(h)\Ub_\theta^{-1}, \qquad \Dl_\theta(h)=\Dl_{0,\theta}(h)+\Vb(\mbf{\phi}_\theta(x)).
\]
We set $\Ab_\theta(x):=(D\mbf{\phi}_\theta(x))^{-1}$, and by the definition of the distorted operator, the chain rule gives the conjugation formula for  
$D_j= \frac{h}{i}\partial_{x_j} :$
\[
\Ub_\theta D_j \Ub_\theta^{-1} = \sum_{k=1}^3 (\Ab_{\theta}(x))_{kj}D_k
-\frac{h}{2i}\sum_{k=1}^3 (\Ab_{\theta}(x))_{kj}\partial_{x_k}\log J_{\lmbf{\phi}_\theta}(x).
\]
Therefore,
\[
\Ub_\theta(\mbf{\alpha}\cdot D)\Ub_\theta^{-1}
= \mbf{\alpha}\cdot \Ab_{\theta}(x)^T D
-\frac{h}{2i}\,\mbf{\alpha}\cdot \Ab_{\theta}(x)^T\nabla\log J_{\lmbf{\phi}_\theta}(x).
\]
Since
\[
D\mbf{\phi}_\theta(x)=
\begin{cases}
(1+\theta)\Ib_3, & |x|\ge 2R_0,\\
\Ib_3, & |x|\le R_0,
\end{cases}
\]
we may write $\Ab_{\theta}(x)^T = (1+\theta)^{-1} \Ib_3 + \Bb_\theta(x)$. 
The coefficient $\Bb_\theta(x)$ is smooth, analytic in $\theta$, and compactly 
supported in $\{|x|\le2R_0\}$. Indeed, $\Ab_\theta(x)^T=(1+\theta)^{-1}\Ib_3$ for $|x|\ge2R_0$, 
whereas for $|x|\le R_0$ one has $\Ab_\theta(x)^T=\Ib_3$, so $\Bb_\theta$ need not vanish 
in the inner ball.
 
Thus,
\[
\Dl_{0,\theta}(h)
= \frac{c}{1+\theta}\,\mbf{\alpha}\cdot D
+ \mbf{\beta} mc^2
+ \sum_{j=1}^3 \mbf{\alpha}_j \Rb_{1,j}(x,D,\theta),
\]
where
\[
\Rb_{1,j}(x,D,\theta)
= c\sum_{k=1}^3 (\Bb_{\theta}(x))_{kj}D_k
-\frac{ch}{2i}\sum_{k=1}^3 (\Ab_{\theta}(x))_{kj}\partial_{x_k}\log J_{\lmbf{\phi}_\theta}(x).
\]

Then, we rewrite the distorted free Dirac operator as
\begin{equation}
\label{RED_Sec2-eqs-free-distored-ops-final}
\Dl_{0,\theta}(h)
= \frac{c}{1+\theta}\,\boldsymbol{\alpha}\cdot D
+ \boldsymbol{\beta} mc^2
+ \Rb_2(x,D,\theta,h),
\end{equation}
where $\Rb_2(x,D,\theta,h)= \sum_{j=1}^3 \mbf{\alpha}_j \Rb_{1,j}(x,D,\theta)$ is a matrix-valued semiclassical differential operator of order $1$, compactly supported in $x\in\R^3$, and analytic in $\theta$ for $|\theta|$ small
enough. 

More precisely, the principal symbol of $\Rb_2$ is given by
\begin{equation}
\label{RED_Sec2-eqs-R2-principal-symbol}
\rb_{2}^{(1)}(x,\xi,\theta) = c\,\boldsymbol{\alpha}\cdot \Big( \Ab_{\theta}(x)^T - \frac{1}{1+\theta} \Ib_3 \Big)\xi \in \Ssf(\kuma{\x}), 
\end{equation}
where the symbol class $\Ssf(\kuma{\x})$ is defined in \cite{ZworskiBook2012}. 
Next, we give an upper bound for $\rb_{2}^{(1)}(x,\xi,\theta)$.

\begin{lemma}
Let $\Fb_\theta(x):=D\mbf{\phi}_\theta(x)=\Ib_3+\theta D \gb(x)$, $\Ab_{\theta}(x)=\Fb_\theta(x)^{-1}$, and $\Bb_\theta(x) =\Ab_{\theta}(x)^T-(1+\theta)^{-1} \Ib_3$. 
There exist $\theta_0>0$ and constants $C_{\alpha}, C_{\alpha\beta}>0$ such that for all 
$|\theta|\le\theta_0$ and all multiindices $\alpha,\beta\in\N^3$,
\begin{equation}
\label{zcmm2025b:estimate-r2}
\sup_{x\in\R^3}\|\partial_x^\alpha\Bb_{\theta}(x)\|\le C_{\alpha} |\theta|, 
\qquad \|\partial_x^\alpha\partial_\xi^\beta\rb_{2}^{(1)}(x,\xi,\theta)\|
\le C_{\alpha\beta}|\theta|\langle\xi\rangle^{1-|\beta|}.
\end{equation}
\label{ZCMM2025b:lem:bound-r2}
\end{lemma}

\begin{proof}
First, we show $\| \Ab_{\theta}\| \leq C $ for some constant $C>0$.  The Neumann series gives 
$\Fb_\theta(x)^{-1}=(\Ib_3+\theta D\gb(x))^{-1}$ because $|\theta|\,\|D\gb(x)\|\le \frac12$. Hence,
\begin{equation}\label{eq:inv-bound}
\sup_{x\in\R^3}\|\Ab_{\theta}(x)\|=\sup_{x\in \R^3}\| \Fb_\theta(x)^{-1}\|
\le \frac{1}{1-\sup_{x\in \R^3}\|\theta D\gb(x) \|}\le 2.
\end{equation}
Using the resolvent identity, we obtain
\[
\Ab_{\theta}(x)-\frac{1}{1+\theta} \Ib_3
=\Fb_\theta(x)^{-1}-((1+\theta)\Ib_3)^{-1}
=\Fb_\theta(x)^{-1}\Big((1+\theta)\Ib_3-\Fb_\theta(x)\Big)\frac{1}{1+\theta}.
\]
But $(1+\theta)\Ib_3-\Fb_\theta(x)=\theta(\Ib_3-D\gb(x))$, so by \eqref{eq:inv-bound}, we have
\[
\sup_{x\in \R^3}\Big\|\Ab_{\theta}(x)-\frac{1}{1+\theta}\Ib_3\Big\|
\le \sup_{x\in \R^3}\|\Ab_{\theta}(x)\|\,\frac{|\theta|}{|1+\theta|}\,\sup_{x\in \R^3}\|\Ib_3-D\gb(x)\|
\le C\,|\theta|
\]
for $|\theta|\le \theta_0$. Taking
transpose does not change the operator norm, hence the same estimate holds for
$\Bb_{\theta}(x)=\big(\Ab_{\theta}(x)-\frac{1}{1+\theta}\Ib_3\big)^T$. 
Differentiating the identity $\Ab_\theta(\Ib_3+\theta D\gb)=\Ib_3$ gives, for $|\alpha|=1$,
\[
\partial_x^\alpha\Ab_\theta=-\Ab_\theta\,\theta\,\partial_x^\alpha(D\gb)\,\Ab_\theta,
\]
and hence $\|\partial_x^\alpha\Ab_\theta\|=\Oc_\alpha(|\theta|)$. Repeated 
differentiation and induction give $\partial_x^\alpha\Ab_\theta=\Oc_\alpha(|\theta|)$ 
for every $|\alpha|\ge1$. Since $(1+\theta)^{-1}\Ib_3$ is independent of $x$, the same 
estimates hold for $\partial_x^\alpha\Bb_\theta$. The estimates for $\rb_2^{(1)}$ 
then follow from \eqref{RED_Sec2-eqs-R2-principal-symbol}.
\end{proof}

In the chosen semiclassical quantization, the lower-order symbol satisfies $\rb_2^{(0)}(x,\xi,\theta;h)=h\,\Oc(|\theta|)$ in $\Ssf(1)$, with all symbol seminorms, and is supported in $\{|x|\le2R_0\}$. We therefore write
\[
\rb_2(x,\xi,\theta;h)=\rb_2^{(1)}(x,\xi,\theta)+\rb_2^{(0)}(x,\xi,\theta;h).
\]
Then, by Lemma~\ref{ZCMM2025b:lem:bound-r2}, $\rb_2\in\Ssf(\langle\xi\rangle)$ and, for all multiindices $\alpha,\beta\in\N^3$,
\begin{equation}
\label{RED_Sec2-eqs-R2-symbol-estimate}
\big|\partial_x^\alpha\partial_\xi^\beta\rb_2(x,\xi,\theta;h)\big|
\le C_{\alpha\beta}|\theta|\langle\xi\rangle^{1-|\beta|},
\end{equation}
uniformly for $h\in(0,h_0]$ and $|\theta|$ sufficiently small. As a consequence, the full symbol of the distorted Dirac operator $\Dl_\theta(h)$ can be written as
\begin{equation}
\label{ZCMM2025b:RED-Sec2-symbol-distoredops-1}
\db_\theta(x,\xi;h) 
= \frac{c}{1+\theta}\,\mbf{\alpha}\cdot\xi
+ \boldsymbol{\beta}mc^2
+ \Vb(x+\theta\gb(x))
+ \rb_2(x,\xi,\theta;h),
\end{equation}
with $\rb_2\in \Ssf(\kuma{\x})$ compactly supported in $x$, and analytic in $\theta$ for $|\theta|$ small enough. Next, to locate the essential spectrum of the distorted free Dirac operator $\Dl_{0,\theta}(h)$, we follow Khochman \cite{Khochman2007} by introducing the constant-coefficient auxiliary operator
\[
\Pl_\theta:=\frac{c}{1+\theta}\sum_{j=1}^3 \boldsymbol{\alpha}_j D_j+\boldsymbol{\beta}mc^2,
\qquad
\pb_\theta(\xi)=\frac{c}{1+\theta}\,\boldsymbol{\alpha}\cdot\xi+\boldsymbol{\beta}mc^2 .
\]
Since $(\boldsymbol{\alpha}\cdot\xi)^2=|\xi|^2\Ib_4$ and $\boldsymbol{\beta}^2=\Ib_4$,
the eigenvalues of $\pb_\theta(\xi)$ are
\[
\pm \sqrt{m^2c^4+\frac{c^2|\xi|^2}{(1+\theta)^2}}
=\pm \sqrt{m^2c^4+\frac{c^2\lambda}{(1+\theta)^2}},\qquad \lambda:=|\xi|^2\in[0,\infty).
\]
Hence
\[
\sm(\Pl_\theta)=\Gamma_\theta^+\cup\Gamma_\theta^-,
\qquad
\Gamma_\theta^\pm:=\Big\{\pm \sqrt{m^2c^4+\frac{c^2\lambda}{(1+\theta)^2}}:\lambda\in[0,\infty)\Big\}.
\]
For $\theta\in\Theta_\epsilon$, the principal branch of the square root is holomorphic on
$\C\setminus(-\infty,0]$, and the argument of
$m^2c^4+\frac{c^2\lambda}{(1+\theta)^2}$ stays in $(-\pi/2,\pi/2)$ for $\lambda\ge0$. 
Furthermore, we have
\[
\Dl_{0,\theta}(h)=\Pl_\theta+\Rb_2(x,D,\theta,h),
\]
where $\Rb_2(x,D,\theta,h)$ is a first-order semiclassical differential
operator with coefficients compactly supported in $\{|x|\leq 2R_0\}$. In particular,
$\Dl_{0,\theta}(h)=\Pl_\theta$, $|x|\geq 2R_0$. Moreover, for $|\theta|$ sufficiently 
small, the principal symbol of $\Dl_{0,\theta}(h)$ is uniformly elliptic for $|\xi|$ 
sufficiently large. Thus $\Dl_{0,\theta}(h)$ is a first-order elliptic operator
which agrees, outside a compact set in $x$, with the constant-coefficient operator $\Pl_\theta$.
Let $ z\notin\Gamma_\theta^+\cup\Gamma_\theta^-$. Then $\Pl_\theta-z \mbf{I} \, : \, H^1(\R^3;\C^4)\to L^2(\R^3;\C^4)$ is invertible. 
Since $\Dl_{0,\theta}(h)-z\mbf{I}$ is elliptic at high frequency and agrees with $\Pl_\theta-z \mbf{I}$ for $|x|\geq 2R_0$, one may combine the
constant-coefficient inverse $(\Pl_\theta-z \mbf{I})^{-1}$ near spatial
infinity with a local elliptic parametrix on a compact set. This gives
an operator $\Qb_z \, : \, L^2(\R^3;\C^4)\longrightarrow H^1(\R^3;\C^4)$
such that
\[
\Qb_z(\Dl_{0,\theta}(h)-z\mbf{I} )=\mbf{I}+\Kb_1(z),\qquad (\Dl_{0,\theta}(h)-z \mbf{I})\Qb_z=\mbf{I}+\Kb_2(z),
\]
where $\Kb_1(z)$ and $\Kb_2(z)$ are compact operators. Hence
$\Dl_{0,\theta}(h)-z\mbf{I} \, : \,  H^1(\R^3;\C^4) \longrightarrow L^2(\R^3;\C^4)$ is 
Fredholm. To determine its index, consider the family
\[
\Dl_{s,\theta}(h):=\Pl_\theta+s\Rb_2(x,D,\theta,h), \qquad 0\leq s\leq1.
\]
For every $s\in[0,1]$, the operator $\Dl_{s,\theta}(h)-z$ has the
same elliptic principal behaviour at high frequency and agrees with 
$\Pl_\theta-z$ outside the fixed compact set $\{|x|\leq2R_0\}$. The preceding 
parametrix construction therefore shows that $\Dl_{s,\theta}(h)-z$ is 
Fredholm for every $s\in[0,1]$. Its Fredholm index is consequently constant in $s$. Since
$\Dl_{s=0,\theta}(h)-z \mbf{I}=\Pl_\theta-z \mbf{I}$ is invertible, its index is zero, and therefore
\[
\operatorname{ind}\big(\Dl_{0,\theta}(h)-z\mbf{I}\big)=0.
\]
Conversely, if $z\in\Gamma_\theta^+\cup\Gamma_\theta^-$,  one constructs a Weyl sequence by taking a generalized eigenfunction of the constant-coefficient operator $\Pl_\theta$ 
at energy $z$ and localizing and translating it farther and farther into the region
where $\Dl_{0,\theta}(h)=\Pl_\theta$. Hence $z$ belongs to the essential spectrum of 
$\Dl_{0,\theta}(h)$. We conclude that
\[
\sme(\Dl_{0,\theta}(h))   =\Gamma_\theta^+\cup\Gamma_\theta^-.
\]
Finally, whenever $z\in\rho(\Dl_{0,\theta}(h))\cap\rho(\Pl_\theta)$,  the resolvent identity gives
\[
(\Dl_{0,\theta}(h)-z \mbf{I})^{-1}-(\Pl_\theta-z \mbf{I})^{-1} = -(\Dl_{0,\theta}(h)-z\mbf{I})^{-1} \Rb_2(x,D,\theta,h) (\Pl_\theta-z \mbf{I})^{-1}.
\]
The operator on the right is compact on $L^2(\R^3;\C^4)$: the two resolvents provide one derivative of regularity on either side, while $\Rb_2$ is a first-order operator whose coefficients are compactly supported. Equivalently, the resulting composition is a compactly
supported pseudodifferential operator of negative order, modulo a
smoothing operator. Thus the resolvent difference is compact.

We now pass from the distorted free operator to
\[
 \Dl_\theta(h)=\Dl_{0,\theta}(h)+\Vb(\mbf{\phi}_\theta(x)).
\]
By Assumption~\ref{ZCMM2025b:MainAssumpV}, uniformly for $\theta$ in compact subsets of $\Theta_\epsilon$,
\[
 \|\Vb(\mbf{\phi}_\theta(x))\|_{\Bc(\C^4)}\longrightarrow0,  \qquad |x|\to\infty.
\]
For $z\in\rho(\Dl_{0,\theta}(h))$, elliptic regularity gives
\[
 (\Dl_{0,\theta}(h)-z \mbf{I})^{-1}:L^2(\R^3;\C^4)\longrightarrow H^1(\R^3;\C^4)
\]
boundedly.  Multiplication by $\Vb(\mbf{\phi}_\theta(x))$ is compact from $H^1$ to $L^2$: after insertion of a spatial cutoff, 
compactness follows from Rellich's theorem on a fixed ball, while the norm of the tail tends to zero as the cutoff radius 
tends to infinity.  Hence
\[
 \Vb(\mbf{\phi}_\theta)(\Dl_{0,\theta}(h)-z \mbf{I})^{-1} \in\Bc_\infty(L^2(\R^3;\C^4)).
\]
It follows that $\Vb(\phi_\theta)$ is relatively compact with respect to $\Dl_{0,\theta}(h)$ and therefore
\[
 \sme(\Dl_\theta(h))=\sme(\Dl_{0,\theta}(h))  =\Gamma_\theta^+\cup\Gamma_\theta^-.
\]
Moreover, whenever $z\in\rho(\Dl_{0,\theta}(h))$,
\[
 \Dl_\theta(h)-z \mbf{I} =\Big(I+\Vb(\mbf{\phi}_\theta)(\Dl_{0,\theta}(h)-z \mbf{I})^{-1}\Big)   (\Dl_{0,\theta}(h)-z \mbf{I}).
\]
Thus, more generally, $\Dl_\theta(h)-z \mbf{I} :H^1\to L^2$ is Fredholm of index zero for every $z\notin\Gamma_\theta^+\cup\Gamma_\theta^-$.  The family is analytic in $z$ and is invertible at points of the physical resolvent set (for example, at suitable points with sufficiently large imaginary part).  The analytic Fredholm theorem therefore implies that $(\Dl_\theta(h)-z \mbf{I})^{-1}$ is meromorphic on each connected component of 
$\C\setminus(\Gamma_\theta^+\cup\Gamma_\theta^-)$ used below, with finite-rank poles; see, 
for example,~\cite{ReedSimon1980}. 

Fix $\theta_1\in\Theta_\epsilon\cap\C^+$ and define the following sectors
\[
S_{\theta_1}^{\pm}:=\bigcup_{\theta\in D_{\epsilon,\theta_1}^+}\Gamma_\theta^\pm,
\qquad
D_{\epsilon,\theta_1}^+:=\Big\{\theta\in\Theta_\epsilon\cap\C^+:
|\arg(1+\theta)|<|\arg(1+\theta_1)|,\ |1+\theta|>|1+\theta_1|\Big\}.
\]
The resonances of $\Dl(h)$ are defined as the discrete eigenvalues of $\Dl_\theta(h)$ in $S_{\theta_1}^{\pm}$,
\[
\Res(\Dl(h)):=\smd(\Dl_\theta(h))\cap\big(S_{\theta_1}^+\cup S_{\theta_1}^-\big),
\]
and this set is independent of $\theta\in D_{\epsilon,\theta_1}^+$.
For $\fb,\gb \in\Bc$, the scalar resolvent matrix element
\[
F_{\lmbf{f},\mbf{g}}(z)=\langle \mbf{f},(z \mbf{I} -\Dl(h))^{-1} \mbf{g}\rangle.
\]
is initially defined on the physical resolvent set. For real $\theta$, unitarity gives the corresponding identity with the distorted operator. Since $\theta\mapsto\Dl_\theta(h)$ and $\theta\mapsto\Ub_\theta \mbf{f}$ are holomorphic on analytic vectors, the identity continues analytically in $\theta$. Analytic Fredholm theory then gives a meromorphic continuation of $F_{f,g}$ into the sector exposed by the distortion, and its poles coincide, with algebraic multiplicity, with the discrete eigenvalues of $\Dl_\theta(h)$ there. By uniqueness of meromorphic continuation, these poles are independent of the admissible distortion parameter $\theta$ and of the particular distortion field within the Hunziker class; see \cite{Hunziker1986,Khochman2007,ZCMM2025a}. This justifies the definition of resonances through the discrete spectrum of $\Dl_\theta(h)$. We set
\[
\Res_\pm(\Dl(h)):=\smd(\Dl_\theta(h))\cap S_{\theta_1}^\pm .
\]
The algebraic multiplicity of a resonance $z\in\Res(\Dl(h))$ is given by the rank of the Riesz projection
\[
m(z):=\rank \Big(\frac{1}{2\pi i}\oint_{\gamma(z)}(\zeta\cdot \Ib-\Dl_\theta(h))^{-1}\,d\zeta\Big),
\]
where $\gamma(z)$ is a small positively oriented contour around $z$ containing no other point of $\sm(\Dl_\theta(h))$.

We limit our study of resonances to a domain that satisfies the following assumption.

\begin{assumption}
For each sign $\pm$, let $\Omega_\pm$ be an open, simply connected and relatively compact subset of
\[
 \{z\in\C:\ \pm\re z>mc^2\}.
\]
We assume that
\[
 \Omega_\pm\cap\{z\in\C:\ \pm\im z>0\}\neq\emptyset,
\]
and that there exists an admissible distortion parameter
$\theta_\pm\in\Theta_\epsilon\cap\C^+$ such that
\[
 \overline{\Omega_\pm}\cap
 \bigl(\Gamma_{\theta_\pm}^+\cup\Gamma_{\theta_\pm}^-\bigr)=\emptyset.
\]
In particular, one may take open rectangles
\[
 \Omega_+=(a_+,b_+)+i(\alpha_+,\beta_+),
 \qquad mc^2<a_+<b_+,\quad \alpha_+<0<\beta_+,
\]
and
\[
 \Omega_-=(a_-,b_-)+i(\alpha_-,\beta_-),
 \qquad a_-<b_-<-mc^2,\quad \alpha_-<0<\beta_-,
\]
provided their closures are disjoint from the distorted essential spectrum for the chosen distortion parameters.
\label{ZCMM2025b-Dirac-domain-assumption}
\end{assumption}

We briefly review Khochman's proof of the upper bound for the number of resonances of the Dirac operator, as presented in \cite[Theorem~2.6]{Khochman2007}, because several parts of that argument are used below.

\begin{theorem}
Assume that $\Dl(h)$ satisfies Assumption~\ref{ZCMM2025b:MainAssumpV} and 
let $\Omega_\pm$ satisfy Assumption~\ref{ZCMM2025b-Dirac-domain-assumption}. Then there exists a constant $C_{\Omega_\pm}>0$, independent of $h$, such that
\begin{equation}
\label{ZCMM2025b_Number_of_resonance}
 \#\bigl(\Res(\Dl(h))\cap\Omega_\pm\bigr)\leq C_{\Omega_\pm}h^{-3},
\end{equation}
where the resonances are counted with algebraic multiplicity.
\label{ZCMM2025b:Thm2dot6Khochman}
\end{theorem}

\begin{proof}
We follow the construction in \cite{Khochman2007}.  To avoid a boundary issue in the zero-counting argument, first choose an admissible open domain $\widetilde\Omega_\pm$ such that
\[
 \overline{\Omega_\pm}\Subset\widetilde\Omega_\pm
 \Subset\{z\in\C:\ \pm\re z>mc^2\},
\]
with $\overline{\widetilde\Omega_\pm}$ still disjoint from the distorted essential spectrum.  All estimates below are made uniformly for $z\in\overline{\widetilde\Omega_\pm}$.

Let $\hb_\theta(x,\xi)$ denote the semiclassical principal symbol of the distorted operator 
$\Dl_\theta(h)$.  Choose $\psi\in C_0^\infty(\R^3;\R)$ with $0\leq\psi\leq1$, $\psi=1$ on $\{|x|\leq1\}$, and define the real, nonnegative function
\[
 f(x,\xi)=C_0\psi\!\left(\frac{x}{\alpha_0}\right)
              \psi\!\left(\frac{\xi}{\beta_0}\right),
\]
where $C_0,\alpha_0,\beta_0>0$ are chosen sufficiently large.  Set
\[
 \widetilde{\Tb}=\operatorname{Op}_h^w(-if) \otimes \mbf{I}_4,
 \qquad  \widetilde{\Dl}_\theta(h)=\Dl_\theta(h)+\widetilde{\Tb}.
\]
The principal symbol of $\widetilde\Dl_\theta(h)-z \mbf{I}$ is therefore
\[
 \hb_\theta(x,\xi)-i f(x,\xi) \mbf{I}_4-z \mbf{I}_4.
\]
Khochman's symbol construction gives, for every $\mu,\nu\in\N^3$,
\begin{equation}
\label{ZCMM24Dirac_symbols_eqs}
 \left\|\partial_x^\mu\partial_\xi^\nu
 \bigl( \hb_\theta(x,\xi)-if(x,\xi)I_4-zI_4\bigr)^{-1}\right\|
 \leq C_{\mu\nu}\langle\xi\rangle^{-1-|\nu|},
 \qquad z\in\overline{\widetilde{\Omega}}_\pm.
\end{equation}
Using \eqref{ZCMM24Dirac_symbols_eqs} and semiclassical composition, we obtain a parametrix for
$\widetilde{\Dl}_\theta(h)-z\mbf{I}$. More precisely, the operator is Fredholm of index zero and the parametrix estimate excludes a nontrivial kernel for $h$ sufficiently small. For $N\in\N$ set
\[
 \Dc_N(h):=\Dom\bigl(\Dl(h)^N\bigr),\qquad \Dc_0(h):=\Hc,
\]
with these spaces equipped with their graph norms. Then, for every $N\in\N$,
\begin{equation}
\label{ZCMM24Dirac_tilde_resolvent}
 (\widetilde{\Dl}_\theta(h) - z \mbf{I} )^{-1}  =\Oc_N(1)\, : \, \Dc_N(h) \longrightarrow \Dc_{N+1}(h),
 \qquad z\in\overline{\widetilde{\Omega}}_\pm.
\end{equation}
We now introduce the finite-rank modification used by Khochman.  Let $\chi\in C_0^\infty(\R)$ be chosen so that
\[
 \chi(|x|^2+|\xi|^2)=1\qquad\hbox{on }\supp f,
\]
and put
\[
 \Tb_h=\chi(-h^2\Delta+|x|^2)\widetilde{\Tb}.
\]
The spectral cutoff $\chi(-h^2\Delta+|x|^2)$ of the harmonic oscillator has rank $\Oc(h^{-3})$.  
Choose $\chi_1\in C_0^\infty(\R^3)$ so that the resulting operator is spatially localised and define
\begin{equation}
\label{ZCMM24Dirac_eqs_hatDtheta}
 \widehat{\Dl}_\theta(h)  :=\Dl_\theta(h)+\Kb_h, \qquad  \Kb_h:=\chi_1 \Tb_h\chi_1.
\end{equation}
Then
\[
 \operatorname{rank} \Kb_h=\Oc(h^{-3}),  \qquad
\Kb_h=\Oc(1)\, : \, \Dc_{N}(h) \longrightarrow \Dc_{M}(h) 
\]
for all $M,N\in\N$, and the functional-calculus comparison with 
\eqref{ZCMM24Dirac_tilde_resolvent} gives
\begin{equation}
\label{ZCMM24Dirac_hat_resolvent}
 (z-\widehat{\Dl}_\theta(h))^{-1}=  \Oc_N(1)\, : \, \Dc_{N}(h) \longrightarrow \Dc_{N+1}(h)
\end{equation}
uniformly on $\overline{\widetilde\Omega_\pm}$. For $z\in\widetilde\Omega_\pm$, set
\[
\Bb(z):=\Kb_h(z \mbf{I} - \widehat{\Dl}_\theta(h))^{-1}.
\]
Then $\Bb(z)$ is holomorphic and finite rank, with
\[
 \operatorname{rank} \Bb(z)=\Oc(h^{-3}),
 \qquad \| \Bb(z)\|=\Oc(1),
 \qquad \|\Bb(z)\|_{\mathrm{tr}}= \Oc(h^{-3}).
\]
Since $\Kb_h=\widehat{\Dl}_\theta(h)-\Dl_\theta(h)$,
\begin{equation}
\label{ZCMM24Dirac_eqs_rep_Dtheta}
 (\mbf{I}+\Bb(z))(z \mbf{I}  -\widehat{\Dl}_\theta(h))=z \mbf{I}-\Dl_\theta(h).
\end{equation}
Thus the resonances in $\widetilde\Omega_\pm$, counted with algebraic multiplicity, coincide with the zeros of the holomorphic Fredholm determinant
\[
 \Db(z):=\det(\mbf{I}+\Bb(z)).
\]
Indeed, integrating the logarithmic derivative of $\Db$ around a small contour about a zero gives the trace of the corresponding Riesz projection, hence its rank. 

The determinant is controlled by the trace norm:
\begin{equation}
\label{ZCMM24Dirac_D_upper}
 | \Db(z)|\leq\exp\bigl(\| \Bb(z)\|_{\mathrm{tr}}\bigr)
 \leq \exp(Ch^{-3}),  \qquad z\in\widetilde\Omega_\pm.
\end{equation}
Next choose $\rho>0$ and a point
\[
 z_0\in\widetilde\Omega_\pm\cap\{z:\ \pm\im z\geq\rho\},
\]
on the physical resolvent side.  The distorted resolvent satisfies
\[
 (z_0-\Dl_\theta(h))^{-1}=\Oc(1).
\]
By \eqref{ZCMM24Dirac_eqs_rep_Dtheta},
\[
 (\mbf{I}+\Bb(z_0))^{-1}
 =(z_0 \mbf{I}-\widehat{\Dl}_\theta(h))(z_0 \mbf{I} -\Dl_\theta(h))^{-1}
 =\Oc(1).
\]
Moreover,
\[
 ( \mbf{I}+ \mbf{B}(z_0))^{-1}= \mbf{I}-\mbf{B}(z_0)(\mbf{I}+\mbf{B}(z_0))^{-1},
\]
and the second term has trace norm $\Oc(h^{-3})$.  Hence
\[
 \left|\det\bigl(( \mbf{I}+\Bb(z_0))^{-1}\bigr)\right|  \leq \exp(Ch^{-3}),
\]
which is equivalent to
\begin{equation}
\label{ZCMM24Dirac_D_lower}
 | \Db(z_0)|\geq \exp(-Ch^{-3}).
\end{equation}

Finally, apply Jensen's formula (see, e.g.,  \cite[Appendix D.1.2]{DyatlovZworski2019}), or 
equivalently the standard holomorphic zero-counting estimate on the nested domains 
$\Omega_\pm\Subset\widetilde\Omega_\pm$, to the function $\Db$.  Combining \eqref{ZCMM24Dirac_D_upper} and \eqref{ZCMM24Dirac_D_lower} gives
\[
 \sum_{\substack{z\in\Omega_\pm\\ \lmbf{D}(z)=0}} \operatorname{ord}_z\Db \leq Ch^{-3}, 
\]
where $\operatorname{ord}_z \Db$ equals the algebraic multiplicity of the corresponding resonance. This proves \eqref{ZCMM2025b_Number_of_resonance}.
\end{proof}

\section{Resolvent estimate of Dirac operator near real axis}
\label{ZCMM2025b:SecResolvEst}
In this section we establish resolvent estimates near the real axis.  First, using the resonance-counting 
bound of Theorem~\ref{ZCMM2025b:Thm2dot6Khochman}, we obtain an exponential estimate for the 
cut-off resolvent away from the critical set.  We then apply a semiclassical maximum principle to derive 
the refined estimate of Proposition~\ref{ZCMM2025b:Lemma-SMP}, which controls the cut-off resolvent 
in a thin neighbourhood of the real axis.  We also prove a reflected version of the three-lines theorem 
in order to treat negative-energy resonances in the upper half-plane.

We begin with the exponential estimate away from the critical set.  A related resolvent estimate for Dirac operators with compactly supported potentials was proved by Kungsman--Melgaard~\cite{JKMM2014a} using a Grushin problem.  Here we give a different argument that applies to the more general potentials considered in this paper.

We now fix the resolvent convention. Put
\[
\Rb(z;h):=(z \mbf{I} -\Dl(h))^{-1}, \qquad \im z\ne0.
\]

\begin{theorem}
Suppose that $\Vb$ satisfies Assumption~\ref{ZCMM2025b:MainAssumpV}, and let
\[
  \Omega_{0,\pm}\Subset \Omega_{1,\pm}\Subset \Omega_\pm
\]
be fixed domains contained in an admissible domain $\Omega_\pm$ satisfying Assumption~\ref{ZCMM2025b-Dirac-domain-assumption}. Let $\chi\in C_0^\infty(\R^3;\R)$ satisfy
\[
   \supp\chi\subset\{x\in\R^3:|x|\le R_0\},
\]
so that the distortion is the identity on $\supp\chi$. Then there exist constants $A>0$, $h_0>0$, and $S_0>0$, depending on the fixed domains, such that for every $0<h<h_0$ and every $0<S(h)<S_0$,
\begin{equation}
\label{ZCMM2025b:eqs-Theorem-resolvent-estimate-1}
 \|\chi\Rb(z;h)\chi\|_{\Bc(L^2(\R^3;\C^4))} \leq A\exp\!\left(Ah^{-3}\log\frac1{S(h)}\right)
\end{equation}
for all
\[
 z\in \Omega_{0,\pm}\setminus
 \bigcup_{m\in\Res(\Dl(h))\cap\Omega_{1,\pm}}D(m,S(h)).
\]
Here $D(m,S(h))$ denotes the disk centred at $m$ with radius $S(h)$ and 
$\bigcup_{m\in\Res(\Dl(h))\cap\Omega_{1,\pm}}D(m,S(h))$ is referred to as the critical set. 
\label{ZCMM2025b:Thm-resol-est-away-critset}
\end{theorem}

\begin{proof}
We use the factorization constructed in the preceding section.  Uniformly for $z\in\Omega_{1,\pm}$,
\begin{equation}
\label{ZCMM2025b:factorization-resolvent}
 (z \mbf{I}-\Dl_\theta(h))^{-1}
   =( z \mbf{I} -\widehat{\Dl}_\theta(h))^{-1}(\mbf{I}+\Bb(z))^{-1},
\end{equation}
where $(z \mbf{I}-\widehat{\Dl}_\theta(h))^{-1}=\Oc(1)$ and $\Bb(z)$ is a holomorphic finite-rank, hence trace-class, family satisfying
\begin{equation}
\label{ZCMM2025b:B-trace-bound}
 \operatorname{rank}\Bb(z)=\Oc(h^{-3}),
 \qquad
 \|\Bb(z)\|_{\mathrm{tr}}\le Ch^{-3}.
\end{equation}
Set
\[
 F(z):=\det(\mbf{I}+\Bb(z)).
\]
By the determinant bounds obtained in the proof of Theorem~\ref{ZCMM2025b:Thm2dot6Khochman},
\begin{equation}
\label{ZCMM2025b:F-upper}
 \sup_{z\in\Omega_{1,\pm}}\log|F(z)|\le Ch^{-3}.
\end{equation}
Choose a fixed point
\[
 z_0\in\Omega_{1,\pm}\cap\{z\in\C_\pm:|\im z|\ge c_0\}
\]
on the physical resolvent side.  The self-adjoint resolvent estimate, together with the factorization above, 
gives
\begin{equation}
\label{ZCMM2025b:F-lower-z0}
 |F(z_0)|\ge \exp(-Ch^{-3}).
\end{equation}
Consequently,
\begin{equation}
\label{ZCMM2025b:F-growth-difference}
 \sup_{z\in\Omega_{1,\pm}}\log|F(z)|-\log|F(z_0)|\le Ch^{-3}.
\end{equation}

We now apply the lower-modulus estimate for holomorphic functions on nested domains, in the form of \cite[Equation~(D.1.13)]{DyatlovZworski2019}.  The zeros of $F$ in $\Omega_{1,\pm}$, with their orders, coincide with the resonances of $\Dl(h)$ there, with their algebraic multiplicities. Thus, for 
$z\in\Omega_{0,\pm}$ outside the disks $D(m,S(h))$ centred at these zeros,
\begin{equation}
\label{ZCMM2025b:eqs_Lower_bound_Derterminant}
 \log|F(z)|
 \ge -C\log\frac1{S(h)}
 \left(
   \sup_{\zeta\in\Omega_{1,\pm}}\log|F(\zeta)|-\log|F(z_0)|
 \right)
 \ge -Ch^{-3}\log\frac1{S(h)}.
\end{equation}
Hence
\begin{equation}
\label{ZCMM2025b:eqs_inverse_B(Z)}
 |F(z)|^{-1}  \leq C\exp\!\left(Ch^{-3}\log\frac1{S(h)}\right).
\end{equation}

For a trace-class operator the standard Fredholm determinant inequality gives
\begin{equation}
\label{ZCMM2025b:eqs_B(Z)}
 \|(\mbf{I}+\Bb(z))^{-1}\|  \leq |\det( \mbf{I}+\Bb(z))|^{-1}\det(\mbf{I}+|\Bb(z)|),  \qquad |\Bb|=(\Bb^*\Bb)^{1/2},
\end{equation}
and
\[
\det(\mbf{I}+|\Bb(z)|)  \le \exp(\| \Bb(z)\|_{\mathrm{tr}})  \le \exp(Ch^{-3}).
\]
Combining this with \eqref{ZCMM2025b:factorization-resolvent} and \eqref{ZCMM2025b:eqs_inverse_B(Z)} yields the same bound for the distorted resolvent.

Finally, on the physical resolvent set the distortion is the identity on $\supp\chi$, so
\[
 \chi(z\mbf{I}-\Dl_\theta(h))^{-1}\chi= \chi\Rb(z;h)\chi.
\]
Both sides are meromorphic in the admissible domain, hence this identity continues meromorphically. This proves \eqref{ZCMM2025b:eqs-Theorem-resolvent-estimate-1}.
\end{proof} 

We also need a ``reflective'' version of the three line theorem from complex analysis 
(see, e.g., \cite[Lemma D.1]{DyatlovZworski2019}) to handle the negative energy resonances.

\begin{lemma}
Suppose that $f$ is holomorphic in a neighbourhood of
\[
 \Omega=[-2R,2R]+i[-\delta_-,\delta_+],
 \qquad 0<\delta_-<\delta_+<1,
\]
and that $M_-,M_+,M>0$ satisfy
\begin{align*}
 |f(x-i\delta_-)|&\le M_-, & |x|&\le2R,\\
 |f(x+i\delta_+)|&\le M_+, & |x|&\le2R,\\
 |f(z)|&\le M, & z&\in\Omega.
\end{align*}
Assume
\begin{equation}
\label{ZCMM2025b:TLTCChange}
 R^2\delta_+^{-2}
 \ge \log\!\left(\frac{M}{\min(M_-,M_+)}\right).
\end{equation}
Then, for $x\in[-R,R]$,
\begin{equation}
\label{ZCMM2025b_eqs_TLTCChange_real_line}
 |f(x)|\le eM_-^\vartheta M_+^{1-\vartheta},
 \qquad
 \vartheta:=\frac{\delta_+}{\delta_-+\delta_+}.
\end{equation}
Moreover, writing $m_\pm=\log M_\pm$, for $|x|\le R$ and $-\delta_-\le y\le\delta_+$,
\begin{equation}
\label{ZCMM2025b:reflected-three-line-strip}
 \log|f(x+iy)|
 \le 1+
 \frac{(\delta_-+y)m_+ +(\delta_+-y)m_-}{\delta_-+\delta_+}
 +\delta_+^{-2}y^2.
\end{equation}
\label{ZCMM2025b:negative-Three-line-theorem}
\end{lemma}

\begin{proof}
Define
\[
 g(z):=\overline{f(\overline z)}.
\]
Then $g$ is holomorphic in a neighbourhood of
\[
 [-2R,2R]+i[-\delta_+,\delta_-].
\]
For $g$, the lower boundary has depth $\delta_+$ and bound $M_+$, while the upper boundary has depth $\delta_-$ and bound $M_-$. Since $\delta_-<\delta_+$, Dyatlov--Zworski \cite[Lemma~D.1]{DyatlovZworski2019} applies, and its condition is precisely \eqref{ZCMM2025b:TLTCChange}.  Its conclusion gives
\[
 |g(x)|\le eM_-^{\delta_+/(\delta_-+\delta_+)}
               M_+^{\delta_-/(\delta_-+\delta_+)},
 \qquad |x|\le R.
\]
Because $g(x)=\overline{f(x)}$ for real $x$, this is \eqref{ZCMM2025b_eqs_TLTCChange_real_line}.  Formula \eqref{ZCMM2025b:reflected-three-line-strip} follows in the same way from the quantitative estimate in the proof of that lemma, namely \cite[Equation~(D.1.5)]{DyatlovZworski2019}, after translating the horizontal variable to an arbitrary $x\in[-R,R]$.
\end{proof}

In the  following,  we apply the three line theorem (see, e.g., \cite[Lemma D.1]{DyatlovZworski2019}) (or semiclassical maximum principle, see, e.g.,  \cite[Lemma 7.7]{DyatlovZworski2019}) to obtain cut-off resolvent estimates on domains close to the real axis 
that contain no resonances. 

Let $k_1,k_2$ be fixed positive integers, independent of $h$, with
\begin{equation}
\label{ZCMM2025b:k-condition}
 k_2>3,\qquad k_1\le k_2-3.
\end{equation}
For $h$ sufficiently small, define the larger, vertically symmetric resonance-free rectangles
\begin{align}
\Xi_+(h)&=[E-5h^{k_1},E+5h^{k_1}]
     +i[-4h^{k_2+1},4h^{k_2+1}],
     \qquad E-5h^{k_1}>mc^2,\label{ZCMM2025b:314}\\
\Xi_-(h)&=[E-5h^{k_1},E+5h^{k_1}]
     +i[-4h^{k_2+1},4h^{k_2+1}],
     \qquad E+5h^{k_1}<-mc^2.\label{ZCMM2025b:315}
\end{align}
We assume that these rectangles are contained in fixed admissible domains to which Theorem~\ref{ZCMM2025b:Thm-resol-est-away-critset} applies.

Then, for the Dirac operator, we have the following result,  which is our analogue of 
Lemma 2.3 in \cite{BurqZworski2001} valid for the Schr\"odinger operator.

\begin{proposition}
Suppose that $\Dl(h)$ has no resonances and no real poles in $\Xi_\pm(h)$. 
Then $\chi\Rb(z;h)\chi$ is holomorphic in $\Xi_\pm(h)$, and there exists $C>0$, independent of $h$, such that
\begin{equation}
\label{ZCMM2025b_eqs__SMP_resolvent_estimate}
 \|\chi\Rb(z;h)\chi\|_{\Bc(L^2(\R^3;\C^4))}
 \le Ch^{-k_2-5}
\end{equation}
for $z\in\widetilde\Xi_\pm(h)$, where
\begin{align*}
\widetilde\Xi_+(h)
 &= [E-h^{k_1},E+h^{k_1}]
    +i[-2h^{k_2+5},2h^{k_2+1}],\\
\widetilde\Xi_-(h)
 &= [E-h^{k_1},E+h^{k_1}]
    +i[-2h^{k_2+1},2h^{k_2+5}].
\end{align*}
\label{ZCMM2025b:Lemma-SMP}
\end{proposition}

\begin{proof}
We treat the two energy branches separately.  In both cases the smaller three-line rectangle is separated from the boundary of the larger resonance-free rectangle by a distance bounded below by $c h^{k_2+1}$.  Since $k_1\le k_2-3$, its horizontal margin is much larger.  We may therefore apply Theorem~\ref{ZCMM2025b:Thm-resol-est-away-critset} with
\begin{equation}
\label{ZCMM2025b:S-choice-SMP}
 S(h)=c_1h^{k_2+1}
\end{equation}
for $c_1>0$ sufficiently small.  The absence of resonances in $\Xi_\pm(h)$ then implies the uniform bound
\begin{equation}
\label{ZCMM2025b:global-exp-SMP}
 \|\chi\Rb(z;h)\chi\| \leq M(h):=A\exp\!\left(Ah^{-3}\log\frac1h\right)
\end{equation}
throughout each of the three-line rectangles used below.

\medskip
\noindent
\emph{Positive energies.}
Put
\[
 R=h^{k_1},\qquad
 \delta_-=2h^{k_2+1},\qquad
 \delta_+=2h^{k_2+5},
\]
and consider
\[
 \mathcal S_+(h):=[-2R,2R]+i[-\delta_-,\delta_+].
\]
For unit vectors $\fb,\gb\in L^2(\R^3;\C^4)$ define
\[
 G_+(w;\fb,\gb)
 :=\langle\chi\Rb(E+w;h)\chi\fb,\gb\rangle.
\]
The function $G_+$ is holomorphic on $\mathcal S_+(h)$ and, by \eqref{ZCMM2025b:global-exp-SMP},
\[
 |G_+(w;\fb,\gb)|\le M(h),\qquad w\in\mathcal S_+(h).
\]
On the lower boundary $\im w=-\delta_-$ we retain the exponential estimate
\[
 M_-:=M(h).
\]
On the upper boundary, which lies on the physical resolvent side, self-adjointness gives
\[
 |G_+(w;\fb,\gb)|
 \le \|\chi\|_\infty^2\|((E+w) - \Dl(h))^{-1}\|  \leq \frac{\|\chi\|_\infty^2}{\delta_+} \leq Ch^{-k_2-5}
 =:M_+.
\]
The hypothesis of Dyatlov--Zworski \cite[Lemma~D.1]{DyatlovZworski2019} is satisfied for all sufficiently small $h$, because
\begin{equation}
\label{ZCMM2025b:positive-three-line-condition}
 R^2\delta_-^{-2}
 =\frac14h^{2(k_1-k_2-1)}
 \ge C h^{-3}\log\frac1h
 \ge \log\!\left(\frac{M(h)}{\min(M_-,M_+)}\right).
\end{equation}
Indeed, \eqref{ZCMM2025b:k-condition} gives $2(k_1-k_2-1)\le-8$.

Let
\[
 \theta:=\frac{\delta_-}{\delta_-+\delta_+}.
\]
Then
\[
 1-\theta=\frac{\delta_+}{\delta_-+\delta_+}\le Ch^4.
\]
The three-line theorem therefore gives, for $|\re w|\le R$,
\[
 |G_+(w;\fb,\gb)|_{\im w=0}
 \le eM_+^\theta M_-^{1-\theta}
 \le Ch^{-k_2-5},
\]
since
\[
 M_-^{1-\theta}
 \le \exp\!\left(Ch^4h^{-3}\log\frac1h\right)
 =\exp\!\left(Ch\log\frac1h\right)\le C.
\]
More is obtained from \cite[Equation~(D.1.5)]{DyatlovZworski2019}.  After translating the horizontal variable, for $|\re w|\le R$ and $|\im w|\le\delta_+$, the coefficient of $\log M_-$ is bounded by
\[
 \frac{\delta_+-\im w}{\delta_-+\delta_+}\le Ch^4,
\]
and
\[
 \delta_-^{-2}(\im w)^2\le C h^8.
\]
Consequently,
\begin{equation}
\label{ZCMM2025b:positive-thin-strip}
 \|\chi\Rb(E+w;h)\chi\|\le Ch^{-k_2-5},
 \qquad |\re w|\le R,\quad |\im w|\le\delta_+.
\end{equation}
For the remainder of the upper, physical part of $\widetilde\Xi_+(h)$, namely $\delta_+\le\im w\le2h^{k_2+1}$, self-adjointness gives directly
\[
 \|\chi\Rb(E+w;h)\chi\|
 \le \frac{\|\chi\|_\infty^2}{\im w}
 \le Ch^{-k_2-5}.
\]
Together with \eqref{ZCMM2025b:positive-thin-strip}, this proves \eqref{ZCMM2025b_eqs__SMP_resolvent_estimate} on $\widetilde\Xi_+(h)$.

\medskip
\noindent
\emph{Negative energies.}
Now put
\[
 R=h^{k_1},\qquad
 \delta_-=2h^{k_2+5},\qquad
 \delta_+=2h^{k_2+1},
\]
and consider
\[
 \mathcal S_-(h):=[-2R,2R]+i[-\delta_-,\delta_+].
\]
For unit vectors $\fb,\gb$ set
\[
 G_-(w;\fb,\gb)
 :=\langle\chi\Rb(E+w;h)\chi\fb,\gb\rangle.
\]
Again $|G_-|\le M(h)$ throughout $\mathcal S_-(h)$.  On the lower, physical boundary,
\[
 M_-:=Ch^{-k_2-5},
\]
while on the upper, resonance-side boundary we use
\[
 M_+:=M(h).
\]
The reflected three-line condition \eqref{ZCMM2025b:TLTCChange} holds because
\begin{equation}
\label{ZCMM2025b:negative-three-line-condition}
 R^2\delta_+^{-2}
 =\frac14h^{2(k_1-k_2-1)}
 \ge C h^{-3}\log\frac1h
 \ge \log\!\left(\frac{M(h)}{\min(M_-,M_+)}\right)
\end{equation}
for sufficiently small $h$.

With
\[
 \vartheta:=\frac{\delta_+}{\delta_-+\delta_+},
 \qquad
 1-\vartheta\le Ch^4,
\]
Lemma~\ref{ZCMM2025b:negative-Three-line-theorem} gives on the real segment
\[
 |G_-(w;\fb,\gb)|
 \le eM_-^\vartheta M_+^{1-\vartheta}
 \le Ch^{-k_2-5}.
\]
Using the strip estimate \eqref{ZCMM2025b:reflected-three-line-strip}, for $|\re w|\le R$ and $|\im w|\le\delta_-$ the coefficient of the exponentially large quantity $\log M_+$ is at most
\[
 \frac{\delta_-+\im w}{\delta_-+\delta_+}\le Ch^4,
\]
and $\delta_+^{-2}(\im w)^2\le Ch^8$. Hence
\begin{equation}
\label{ZCMM2025b:negative-thin-strip}
 \|\chi\Rb(E+w;h)\chi\|\le Ch^{-k_2-5},
 \qquad |\re w|\le R,\quad |\im w|\le\delta_-.
\end{equation}
Finally, on the remaining lower, physical part of $\widetilde\Xi_-(h)$, where
$-2h^{k_2+1}\le\im w\le-\delta_-$, the self-adjoint estimate gives
\[
 \|\chi\Rb(E+w;h)\chi\|
 \le \frac{\|\chi\|_\infty^2}{|\im w|}
 \le Ch^{-k_2-5}.
\]
This proves \eqref{ZCMM2025b_eqs__SMP_resolvent_estimate} on $\widetilde\Xi_-(h)$ and completes the proof.
\end{proof}

\section{Resonance Expansions for the Dirac Propagator}
\label{ZCMM2025b:SecRE2}
In this section we establish resonance expansions for the propagator of the time-dependent Dirac equation.  The positive-energy expansion, Theorem~\ref{ZCMM2025b:MainThm2-pos}, combines Proposition~\ref{ZCMM2025b:Lemma-SMP} with semiclassical functional calculus, an almost-analytic extension, and a Helffer--Sj\"ostrand/Cauchy--Green deformation; see, for example,~\cite{HelfferSjostrand1986,DimissiSjostrand1999}.  After deforming the integral representation, we separate the resonance residues from contour and area terms and estimate the latter by the resolvent bounds of Section~\ref{ZCMM2025b:SecResolvEst}.

The negative-energy expansion, Theorem~\ref{ZCMM2025b:MainThm2-neg}, is the reflected analogue.  A key distinction is that $e^{-itz/h}$ for the negative-energy resonances considered here decays in the relevant continued half-plane when $t<0$, rather than when $t>0$.  Consequently the negative-energy resonance expansion is formulated for large negative times, with the contour orientation and the physical/continued resolvent branches reflected accordingly.

The upper bound on the number of resonances in any fixed admissible subset $\Omega_\pm$ satisfying Assumption~\ref{ZCMM2025b-Dirac-domain-assumption} is of order 
$\Oc(h^{-3})$. This counting estimate permits us to choose the non-real boundary of the contour away from all resonances. The precise selection is the following.

\begin{lemma}
\label{ZCMM2025b:pigeonhole-admissible-contour}
Let $\delta>0$ and let $[a-2\delta,b+2\delta]$ be contained either in $(mc^2,\infty)$ or in $(-\infty,-mc^2)$. Assume that, in a fixed complex neighbourhood $\Uc$ of this interval, the resonance counting estimate
\[
 \#\left(\Res(\Dl(h))\cap \Uc \right) \leq C h^{-3}
\]
holds. Fix integers $k,M$ such that
\begin{equation}
\label{ZCMM2025b:admissible-exponent-condition}
 M>3,\qquad k\ge3,\qquad k\le M-3.
\end{equation}
Then, for all sufficiently small $h$, one can choose
\[
 c(h)\in(\delta,2\delta),\qquad d(h)\in[h^M,2h^M],
\]
so that the two vertical lines $\re z=a-c(h)$ and $\re z=b+c(h)$ are at distance at least $c_v h^k$ from all resonances in $\Uc$, while the horizontal line $\im z=-d(h)$ in the positive-energy case, respectively $\im z=d(h)$ in the negative-energy case, is at distance at least $c_h h^{M+3}$ from all resonances in $\mathcal U$. Here $c_v,c_h>0$ are independent of $h$.
\end{lemma}

\begin{proof}
Write $N(h)=\#(\Res(\Dl(h))\cap \Uc)$, so that $N(h)\le Ch^{-3}$. For the vertical sides, each resonance 
$z_j$ gives at most two forbidden values of the parameter $c$, namely $a-\re z_j$ and $\re z_j-b$. 
Remove from $(\delta,2\delta)$ intervals of radius
\[
 \rho_h:=\frac{\delta}{8(2N(h)+1)}
\]
centered at these forbidden values whenever they lie in $(\delta,2\delta)$. The total removed length is strictly smaller than $\delta$, hence some $c(h)\in(\delta,2\delta)$ remains. Since $N(h)\le Ch^{-3}$,
\[
 \rho_h\ge c h^3\ge c h^k,
\]
for $h\in(0,1)$ and $k\ge3$. Thus both vertical sides are separated from every resonance by at least $c_vh^k$.

For the horizontal side, write $s_j=-\im z_j$ in the positive-energy case and $s_j=\im z_j$ in the negative-energy case. Remove from $[h^M,2h^M]$ intervals of radius
\[
\sigma_h:=\frac{h^M}{4(N(h)+1)}
\]
centered at those $s_j$ which belong to $[h^M,2h^M]$. Again the removed intervals cannot cover 
the whole interval. Hence there exists $d(h)\in[h^M,2h^M]$ with
\[
| d(h)-s_j|\ge \sigma_h\ge c h^{M+3}
\]
for every resonance under consideration. This proves the claim.
\end{proof}

We use the contour produced by Lemma~\ref{ZCMM2025b:pigeonhole-admissible-contour} in the following definition.

\begin{definition}[Admissible contour domains]
\label{ZCMM2025b:Assumption-The-admissable-condition}
Let $\delta>0$ and $a<b$. In the positive-energy case assume
\[
 a-2\delta>mc^2,
\]
and in the negative-energy case assume
\[
 b+2\delta<-mc^2.
\]
Fix $k,M$ satisfying \eqref{ZCMM2025b:admissible-exponent-condition}. Define
\begin{align*}
 \Omega_+(h)&=(a-c(h),b+c(h))-i(0,d(h)),\\
 \Omega_-(h)&=(a-c(h),b+c(h))+i(0,d(h)). 
\end{align*}
Suppose that, for all sufficiently small $h$, there exist fixed (admissible) open sets 
$\Omega_\pm^{\mathrm{fix}}$ satisfying 
Assumption~\ref{ZCMM2025b-Dirac-domain-assumption} such that
\begin{equation}
\overline{\Omega}_\pm(h)\Subset\Omega_\pm^{\mathrm{fix}}\Subset \Uc, 
\label{ZCMM2025b:Assumption-The-admissable-condition-inclusion}
\end{equation} 
where $\Uc$ is the fixed neighbourhood used in 
Lemma~\ref{ZCMM2025b:pigeonhole-admissible-contour}. 
We call $\Omega_\pm(h)$ 
(and we denote their closures by $\overline{\Omega}_{\pm}(h)$) admissible contour domains 
if the inclusion \eqref{ZCMM2025b:Assumption-The-admissable-condition-inclusion} holds, if 
\[
 h^M\le d(h)\le2h^M,
\]
and if, writing $V_\pm(h)$ for the union of the two vertical boundary segments and $H_\pm(h)$ for the non-real horizontal boundary segment of $\overline{\Omega}_\pm(h)$, one has
\begin{align}
 \operatorname{dist}\!\left(\Res(\Dl(h))\cap \Uc,V_\pm(h)\right)
 &\ge c_vh^k,\label{ZCMM2025b:vertical-gap}\\
 \operatorname{dist}\!\left(\Res(\Dl(h))\cap \Uc, H_\pm(h)\right)
 &\ge c_hh^{M+3}.\label{ZCMM2025b:horizontal-gap}
\end{align} 
We moreover choose the fixed admissible sets $\Omega_\pm^{\mathrm{fix}}$ with 
sufficient room around the contour family so that, for all sufficiently small $h$, the 
$\Oc(h^k)$ side collars introduced in the Cauchy--Green deformation are also contained in $\Omega_\pm^{\mathrm{fix}}$. We also assume that there are no real poles of the cut-off resolvent in $[a-2\delta,b+2\delta]$.  If real eigenvalues are present, their ordinary spectral-projection terms must be added separately to the expansion below.
\end{definition}

\subsection*{Resonance expansion for positive energy}
 In this subsection, we establish the resonance expansion of the Dirac propagator with a long-range, matrix-valued potential in the positive-energy region.

\begin{theorem}
Suppose that $\Dl(h)$ satisfies Assumption~\ref{ZCMM2025b:MainAssumpV}, 
let $\Omega_+(h)$ be an admissible contour domain in the sense of Definition~\ref{ZCMM2025b:Assumption-The-admissable-condition}, and let 
$\psi\in C_c^\infty(\R;\R)$ satisfy
\[
 \operatorname{supp}\psi=[a,b]\Subset(mc^2,\infty).
\]
Let $\chi\in C_0^\infty(\R^3;\R)$ with $\operatorname{supp}\chi\subset\{x\in\R^3:|x|\le R_0\}$. Then, for every fixed
\begin{equation}
\label{ZCMM2025b:positive-time-exponent}
 L>M+7,
\end{equation}
we have, uniformly for $t\ge h^{-L}$,
\begin{equation}
\label{ZCMM2025b:MainThm2-pos-eq1}
\chi e^{-it\Dl(h)/h}\chi\psi(\Dl(h))
 =\sum_{z\in\Res(\Dl(h))\cap\Omega_+(h)}
 \chi\Res_{w=z}\!\left(e^{-itw/h}\Rb(w;h)\right)\chi\psi(\Dl(h))
 +\Oc(h^\infty)
\end{equation}
in $\mathcal B(L^2(\R^3;\C^4))$.
Here $\Rb(w;h)$ denotes the meromorphic continuation of the resolvent associated with the positive-energy sheet. If the meromorphic resolvent has a pole of order greater than one at a resonance $z$, then the corresponding residue term in \eqref{ZCMM2025b:MainThm2-pos-eq1} contains polynomial factors in $t/h$ arising from the higher-order principal part of the resolvent.
\label{ZCMM2025b:MainThm2-pos}
\end{theorem}

\begin{proof}
Choose $0<\eta< \min\{c_v/10,1/4\}$ and let $\psi_h\in C_0^\infty(\R;\R)$ satisfy
\begin{align}
 \psi_h&=1\quad\text{on }[a-c(h),b+c(h)],\label{ZCMM2025b:psi-h-one}\\
 \operatorname{supp}\psi_h
 &\subset[a-c(h)-\eta h^k,b+c(h)+\eta h^k],\label{ZCMM2025b:psi-h-support}\\
 |\partial_x^j\psi_h(x)|&\le C_jh^{-kj},\qquad j\ge0.\label{ZCMM2025b:psi-h-derivatives}
\end{align}
Since $c(h)>\delta$ and $\operatorname{supp}\psi=[a,b]$, the supports of $\psi$ and $1-\psi_h$ are separated by a fixed positive spectral distance. Since $\psi_h\equiv1$ on a fixed neighbourhood of $\supp \psi$, the two spectral regions are separated by a fixed gap, and
the semiclassical functional calculus therefore gives
\begin{equation}
\label{ZCMM2025b:spectral-cutoff-separation}
 (1-\psi_h)(\Dl(h))\,\chi\,\psi(\Dl(h)) =\Oc(h^\infty)  \quad\text{in }\Bc(L^2(\R^3;\C^4)).
\end{equation}
To make the uniformity in the $h$-dependent cutoff $\psi_h$ explicit, choose a fixed $\psi_0\in C_0^\infty(\R)$ with $\psi_0=1$ near $[a,b]$ and $\supp\psi_0\subset(a-\delta,b+\delta)$.  Then $\psi_h=1$ on $\supp\psi_0$ for every sufficiently small $h$, so $(1-\psi_h)(\Dl)\psi_0(\Dl)=0$ exactly by the spectral calculus.  Hence
\[
 (1-\psi_h)(\Dl)\chi\psi(\Dl)
 =(1-\psi_h)(\Dl)(1-\psi_0(\Dl))\chi\psi(\Dl).
\]
The first factor is uniformly bounded (use the $j=0$ bound in \eqref{ZCMM2025b:psi-h-derivatives}), while the last two fixed spectral localizations are separated by a fixed gap; standard semiclassical pseudolocality/functional calculus therefore makes the latter composition $\Oc(h^\infty)$.  This proves \eqref{ZCMM2025b:spectral-cutoff-separation} without using any derivatives of $\psi_h$. Consequently, using the unitarity of $e^{-it\Dl(h)/h}$,
\begin{equation}
\label{ZCMM2025b:eqs-propagator-formula-2}
 \chi e^{-it\Dl(h)/h}\chi\psi(\Dl(h))
 =\chi e^{-it\Dl(h)/h}\psi_h(\Dl(h))\chi\psi(\Dl(h))
 +\Oc(h^\infty),
\end{equation}
uniformly in $t\in\R$. 
Fix $\varepsilon_0>0$, independent of $h$, and choose $\kappa\in C_0^\infty(\R;[0,1])$ such that
\[
 \kappa(y)=1\quad\text{for }|y|\le\varepsilon_0,
 \qquad
 \kappa(y)=0\quad\text{for }|y|\ge2\varepsilon_0.
\]
We construct an almost-analytic Taylor extension of $\psi_h$ in the imaginary direction and multiply it by the \textit{fixed} cutoff $\kappa( \im z)$. Since $d(h)\asymp h^M\to0$, for all sufficiently small $h$ one has $2d(h)<\varepsilon_0$. Hence $\kappa\equiv1$ on the whole local deformation region $| \im z|\le2d(h)$, and no $\bar\partial$-term coming from the vertical cutoff occurs there. The derivative of $\kappa$ is supported at the fixed distance 
$\varepsilon_0\leq | \im z|\le2\varepsilon_0$, which is disjoint from the Cauchy--Green region used below. Consequently, the almost-analytic extension can be chosen so that
\begin{align}
 \widetilde\psi_h|_{\R}&=\psi_h,\label{ZCMM2025b:aa-real}\\
 \widetilde\psi_h(z)&=1,
 \quad z\in[a-c(h),b+c(h)]+i[-d(h),d(h)],\label{ZCMM2025b:aa-one}\\
 \operatorname{supp}\bar\partial\widetilde\psi_h & \cap\{|\im z|\le2d(h)\}
 \subset\mathcal T(h),   \label{ZCMM2025b:aa-support}
\end{align}
where
\begin{align}
 \mathcal T(h):={} &\Bigl(\{ | \re z-(a-c(h))| \leq 2\eta h^k\}
 \cup\{| \re z-(b+c(h))|\leq 2\eta h^k \}\Bigr)\nonumber\\
                            & \hspace{35mm} \cap \{ | \im z|\leq 2d(h) \}.
\label{ZCMM2025b:transition-region}
\end{align}
The extension may in addition be chosen with the following real-support control in the 
local strip:
\begin{equation}
\label{ZCMM2025b:aa-real-support}
 \supp \widetilde\psi_h\cap\{|\im z|\le2d(h)\}  \subset [a-c(h)-2\eta h^k,\,b+c(h)+2\eta h^k]+i[-2d(h),2d(h)].
\end{equation}
This follows by inserting, in the standard almost-analytic construction, a real cutoff which 
equals one on $\supp \psi_h$ and vanishes outside the slightly larger interval appearing in 
\eqref{ZCMM2025b:aa-real-support}. Moreover, for every prescribed $N\in\N$ the extension may be taken of sufficiently high almost-analytic order so that
\begin{equation}
\label{ZCMM2025b:property-2almostanalext}
 |\bar\partial\widetilde\psi_h(x+iy)|  \leq C_Nh^{-k}\left(\frac{|y|}{h^k}\right)^N,
 \qquad |y|\le2d(h),
\end{equation}
with $C_N$ independent of $h$. Indeed, the derivative bounds 
\eqref{ZCMM2025b:psi-h-derivatives} give the usual remainder $\Oc(|y|^N h^{-k(N+1)})$. 
Any additional $\bar\partial$-support produced by $\kappa'$ lies outside 
$| \im z|\leq2d(h)$ and therefore does not enter the deformation below.

Recall the resolvent convention, $\Rb(z;h):=(z-\Dl(h))^{-1}$, $\im z\ne0$. For positive energies, $\Rb_+(z;h)$ denotes the outgoing cut-off resolvent, initially defined for $\im z>0$ and meromorphically continued through the positive continuous spectrum into the lower half-plane; $\Rb_-(z;h)$ denotes the physical lower-half-plane branch. With this convention Stone's formula reads
\[
 e^{-it\Dl(h)/h}\psi_h(\Dl(h))  =\frac1{2\pi i}\int_{\R}e^{-it\lambda/h}  \bigl(\Rb_-(\lambda;h)-\Rb_+(\lambda;h)\bigr)\psi_h(\lambda)\,d\lambda,
\]
in the cut-off sense. We now state explicitly both the geometric deformation domain 
and the sign convention in the Cauchy--Green argument. Set
\[
 q_h(z):=e^{-itz/h}\widetilde{\psi}_h(z),
\]
and define the closed deformation rectangle
\begin{equation}
\label{ZCMM2025b:CG-rectangle}
 \DF_+(h)  :=[a-c(h)-3\eta h^k,\,b+c(h)+3\eta h^k]+i[-d(h),0].
\end{equation}
By \eqref{ZCMM2025b:aa-real-support}, $q_h$ vanishes on the two vertical boundary 
segments of $\DF_+(h)$, and hence these sides contribute nothing to the 
Cauchy--Green formula. The portion of $\DF_+(h)$ below the real axis is the 
union of the central contour domain $\Omega_+(h)$ and two side collars. Each side 
collar has horizontal width $3\eta h^k$. Since $\eta<c_v/10$, every point in either collar 
lies at distance strictly less than $c_vh^k$ from the corresponding vertical boundary 
$V_+(h)$. By \eqref{ZCMM2025b:vertical-gap}, the side collars therefore contain no 
resonances. By the choice of the fixed admissible neighbourhood in 
Definition~\ref{ZCMM2025b:Assumption-The-admissable-condition}, after decreasing $h_0$ 
if necessary we also have
\[
 \DF_+(h)\Subset\Omega_+^{\mathrm{fix}}\Subset\Uc.
\]
It follows that the poles of the continued outgoing resolvent enclosed by $\DF_+(h)$ 
are exactly the resonances in $\Omega_+(h)$; enlarging the deformation to include the 
side collars introduces no additional residue terms. Moreover, although the Cauchy--Green 
area integral is formally taken over $\DF_+(h)$, \eqref{ZCMM2025b:aa-support} 
reduces it to $\Tc_{+}(h)$ because $\bar\partial\widetilde{\psi}_h$ vanishes elsewhere in 
the local strip.

We orient the lower horizontal boundary of $\DF_+(h)$ from left to right. For 
the physical branch $\Rb_-$, which is holomorphic in the lower half-plane, the 
Cauchy--Green formula gives
\[
 \int_{\R} q_h(\lambda)\Rb_-(\lambda;h)\,d\lambda
 =\int_{\im z=-d(h)}q_h(z)\Rb_-(z;h)\,dz  -2i\iint q_h^{\,\bar\partial}(z)\Rb_-(z;h)\,L(dz),
\]
where
\[
 q_h^{\,\bar\partial}(z)=e^{-itz/h}\bar\partial\widetilde\psi_h(z).
\]
For the outgoing branch $\Rb_+$, meromorphically continued into the lower half-plane, 
the same formula applied after excising small disks around the enclosed poles yields
\[
\begin{aligned}
 \int_{\R} q_h(\lambda)\Rb_+(\lambda;h)\,d\lambda
 &=\int_{\im z=-d(h)}q_h(z)\Rb_+(z;h)\,dz
   -2i\iint q_h^{\,\bar\partial}(z)\Rb_+(z;h)\,L(dz)\\
 &\quad
   -2\pi i\sum_{z\in\Res(\Dl(h))\cap\Omega_+(h)}
   \Res_{w=z} \bigl(q_h(w)\Rb_+(w;h)\bigr).
\end{aligned}
\]
Subtracting the two identities and dividing by $2\pi i$ therefore gives the exact sign pattern
\[
\begin{aligned}
 &\frac1{2\pi i}\int_{\R}q_h(\lambda)
 \bigl(\Rb_-(\lambda;h)-\Rb_+(\lambda;h)\bigr)\,d\lambda\\
 &\qquad=\frac1{2\pi i}\int_{\im z=-d(h)}q_h(z)
 \bigl(\Rb_-(z;h)-\Rb_+(z;h)\bigr)\,dz\\
 &\qquad\quad
 -\frac1\pi\iint q_h^{\,\bar\partial}(z)
 \bigl(\Rb_-(z;h)-\Rb_+(z;h)\bigr)\,L(dz)\\
 &\qquad\quad
 +\sum_{z\in\Res(\Dl(h))\cap\Omega_+(h)}
 \Res_{w=z} \bigl(q_h(w)\Rb_+(w;h)\bigr).
\end{aligned}
\]
Thus, with the lower contour oriented from left to right, the lower-contour term has a positive sign, the $\bar\partial$-area term has a negative sign, and the resonance contribution has a positive sign. Since $q_h$ vanishes outside the real projection of $\DF_+(h)$, the horizontal contour integrals may equivalently be written as integrals over the full lines $\R$ and $\{\im z=-d(h)\}$, as done below.

Applying this identity to the cut-off propagator and using \eqref{ZCMM2025b:aa-one}, we obtain
\begin{equation}
\label{ZCMM2025b:res-exp-positiveEnergy-formula}
 \chi e^{-it\Dl(h)/h}\psi_h(\Dl(h))\chi
 =\Ab_+(h)+\Bb_+(h)+\Cb_+(h),
\end{equation}
where, with the lower horizontal contour oriented from left to right,
\begin{align}
 \Ab_+(h) &:=\frac1{2\pi i}\int_{\im z=-d(h)}e^{-itz/h}
 \chi\bigl(\Rb_-(z;h)-\Rb_+(z;h)\bigr)\chi\widetilde\psi_h(z)\,dz,
 \label{ZCMM2025b:Aplus}\\
 \Bb_+(h) &:=-\frac1\pi\iint_{\mathcal T_+(h)}e^{-itz/h}
 \chi\bigl(\Rb_-(z;h)-\Rb_+(z;h)\bigr)\chi
 \bar\partial\widetilde\psi_h(z)\,L(dz),
 \label{ZCMM2025b:Bplus}\\
 \Cb_+(h) &:=\sum_{z\in\Res(\Dl(h))\cap\Omega_+(h)}
 \chi\Res_{w=z}\!\left(e^{-itw/h}\Rb(w;h)\right)\chi.
 \label{ZCMM2025b:Cplus}
\end{align}
Here
\[
 \Tc_+(h):=\Tc(h)\cap\{-d(h)\le\im z\le0\}.
\]
The almost-analytic cutoff is identically one in a full neighbourhood of every pole enclosed by the contour, so $q_h(w)=e^{-itw/h}$ there and no factor $\widetilde\psi_h$ appears in \eqref{ZCMM2025b:Cplus}. The assumption on real poles in Definition~\ref{ZCMM2025b:Assumption-The-admissable-condition} ensures that the deformation does not cross a pole on the real axis.

\begin{figure}[htbp]
\centering
\begin{tikzpicture}[x=1.20cm,y=1.0cm,>=Latex,font=\small]
  \def\xOleft{-3.00}   
  \def\xOright{3.00}   
  \def\xPleft{-3.40}   
  \def\xPright{3.40}   
  \def\xSleft{-3.80}   
  \def\xSright{3.80}   
  \def\xDleft{-4.20}   
  \def\xDright{4.20}   
  \def\xa{-1.25}
  \def\xb{1.25}
  \def\ybot{-2.10}

  \fill[gray!5] (-4.82,-2.55) rectangle (4.82,1.92);
  \draw[densely dashed,gray!70,line width=0.85pt,rounded corners=4pt]
       (-4.82,-2.55) rectangle (4.82,1.92);
  \node[gray!65,anchor=north east,fill=white,inner sep=1.2pt]
       at (4.70,1.82) {$\Omega_+^{\mathrm{fix}}$};

  \fill[blue!5] (\xDleft,\ybot) rectangle (\xDright,0);
  \draw[blue!70!black,line width=1.25pt]
       (\xDleft,\ybot) rectangle (\xDright,0);

  \fill[cyan!15] (\xDleft,\ybot) rectangle (\xOleft,0);
  \fill[cyan!15] (\xOright,\ybot) rectangle (\xDright,0);

  \fill[orange!20] (\xOleft,\ybot) rectangle (\xOright,0);
  \draw[orange!85!black,line width=1.25pt]
       (\xOleft,\ybot) rectangle (\xOright,0);

  \fill[green!30] (\xSleft,\ybot) rectangle (-2.20,0);
  \fill[green!30] (2.20,\ybot) rectangle (\xSright,0);
  \draw[green!55!black,line width=0.85pt] (\xSleft,\ybot) rectangle (-2.20,0);
  \draw[green!55!black,line width=0.85pt] (2.20,\ybot) rectangle (\xSright,0);

  \draw[dashed,blue!75,line width=1.0pt] (\xOleft,\ybot-0.18) -- (\xOleft,0.30);
  \draw[dashed,blue!75,line width=1.0pt] (\xOright,\ybot-0.18) -- (\xOright,0.30);
  \draw[red!80!black,line width=1.8pt] (\xOleft,\ybot) -- (\xOright,\ybot);
  \draw[red!55!black,line width=1.0pt] (\xSleft,\ybot) -- (\xOleft,\ybot);
  \draw[red!55!black,line width=1.0pt] (\xOright,\ybot) -- (\xSright,\ybot);

  \draw[->,line width=0.9pt] (-5.15,0) -- (5.25,0) node[right] {$\re z$};
  \draw[->,line width=0.9pt] (0,-2.95) -- (0,2.45) node[above] {$\im z$}; 

  \fill (\xa,0) circle (1.1pt);
  \fill (\xb,0) circle (1.1pt);
  \node[below=3pt] at (\xa,0) {$a$};
  \node[below=3pt] at (\xb,0) {$b$};
  \node[below=3pt,xshift=-5pt,fill=white,inner sep=1pt] at (\xOleft,0) {$a-c(h)$};
  \node[below=3pt,xshift=5pt,fill=white,inner sep=1pt] at (\xOright,0) {$b+c(h)$};

  \draw[line width=2.0pt,black] (\xa,0.25) -- (\xb,0.25);
  \node[above=2pt,fill=white,inner sep=1pt] at (0,0.25) {$\supp\psi=[a,b]$};

  \draw[blue!65!black,line width=1.35pt] (\xPleft,0.90) -- (\xPright,0.90); 
  \draw[blue!65!black,line width=0.7pt] (\xPleft,0.79) -- (\xPleft,1.01);
  \draw[blue!65!black,line width=0.7pt] (\xPright,0.79) -- (\xPright,1.01);
  \node[blue!65!black,above=2pt,fill=white,inner sep=1pt] at (0,0.90)
       {$\supp\psi_h\subset[a-c(h)-\eta h^k,\,b+c(h)+\eta h^k]$};

  \draw[purple!70!black,line width=1.35pt] (\xSleft,1.47) -- (\xSright,1.47);
  \draw[purple!70!black,line width=0.7pt] (\xSleft,1.36) -- (\xSleft,1.58);
  \draw[purple!70!black,line width=0.7pt] (\xSright,1.36) -- (\xSright,1.58);
  \node[purple!70!black,above=2pt,fill=white,inner sep=1pt] at (0,1.47)
       {local real support of $\widetilde{\psi}_h$}; 

  \node[blue!70!black,fill=white,inner sep=1.2pt] at (0.65,-0.30) {$\DF_+(h)$};
  \node[orange!85!black,fill=white,inner sep=1.2pt] at (1.60,-0.85) {$\Omega_+(h)$};
  \node[red!80!black,below=5pt,fill=white,inner sep=1.1pt]
       at (0,\ybot) {$H_+(h)$};

  \draw[<->,green!50!black,line width=0.75pt] (\xSleft,-1.05) -- (-2.20,-1.05);
  \node[green!45!black,fill=white,inner sep=1.1pt] at (-3.00,-1.05) {$4\eta h^k$};
  \draw[<->,green!50!black,line width=0.75pt] (2.20,-1.05) -- (\xSright,-1.05);
  \node[green!45!black,fill=white,inner sep=1.1pt] at (3.00,-1.05) {$4\eta h^k$};

  \node[cyan!50!black,rotate=90,fill=white,inner sep=1.2pt]
       at (-4.02,-1.04) {resonance-free side collar};
  \node[cyan!50!black,rotate=90,fill=white,inner sep=1.2pt]
       at (4.02,-1.04) {resonance-free side collar};

  \node[blue!70!black,anchor=north east,fill=white,inner sep=1.1pt]
       at (\xDleft,-2.34) {$q_h=0$};
  \node[blue!70!black,anchor=north west,fill=white,inner sep=1.1pt]
       at (\xDright,-2.34) {$q_h=0$};
  \node[blue!70!black,align=center,fill=white,inner sep=1.1pt]
       at (0,-3.2) {vertical sides of $\DF_+(h)$ at \\ $a-c(h)-3\eta h^k$ and $b+c(h)+3\eta h^k$};
  \draw[<->,black!70,line width=0.75pt] (-0.45,0) -- (-0.45,\ybot);
  \node[black!70,left=3pt,fill=white,inner sep=1pt] at (-0.45,-1.00) {$d(h)$};
  \node[black!70,left=2pt,fill=white,inner sep=1pt] at (2.1,\ybot) {$\im z=-d(h)$}; 

  \foreach \xx/\yy in {-1.60/-1.62,0.60/-1.52,1.55/-1.72}{
    \draw[red!70!black,line width=0.9pt] (\xx-0.07,\yy-0.07) -- (\xx+0.07,\yy+0.07);
    \draw[red!70!black,line width=0.9pt] (\xx-0.07,\yy+0.07) -- (\xx+0.07,\yy-0.07);
  }
  \node[red!70!black,fill=white,inner sep=1.2pt]
       at (0.99,-1.62) {{\tiny poles contributing to $\Cb_+(h)$}};

  \node[green!45!black,above=2pt,fill=white,inner sep=1.1pt]
       at (-3.00,0.05) {$\Tc_+(h)$};
  \node[green!45!black,above=2pt,fill=white,inner sep=1.1pt]
       at (3.00,0.05) {$\Tc_+(h)$};
\end{tikzpicture}
\caption{Geometry of the positive-energy Cauchy--Green deformation in the proof of Theorem~\ref{ZCMM2025b:MainThm2-pos}. The picture is schematic and not to scale in $h$. The orange rectangle is the admissible contour domain $\Omega_+(h)$ and the larger dark-blue rectangle is the actual deformation rectangle $\DF_{+}(h)$ from \eqref{ZCMM2025b:CG-rectangle}. The light-blue parts of $\DF_{+}(h) \setminus\overline{\Omega}_{+}(h)$ are resonance-free side collars. The two green strips represent $\Tc_+(h)=\Tc(h)\cap\{-d(h)\leq\im z\leq0\}$, the only part of the local $\bar\partial$-support contributing to the Cauchy--Green area term $\Bb_+(h)$; each has full horizontal width $4\eta h^k$. The red central lower segment is $H_+(h)$. The integral $\Ab_+(h)$ is taken on the line $\im z=-d(h)$, but the factor $q_h=e^{-itz/h}\widetilde{\psi}_h$ restricts its effective real support to the interval shown above the axis. The vertical sides of $\DF_+(h)$ are placed one further $\eta h^k$ beyond that local real support, so $q_h$ vanishes there and no vertical contour terms remain. The crosses inside $\Omega_+(h)$ schematically denote resonances whose residues form $\Cb_+(h)$. The dashed gray rectangle represents the fixed admissible neighbourhood $\Omega_+^{\mathrm{fix}}$ containing the entire deformation.}
\label{fig:myfig}
\end{figure}
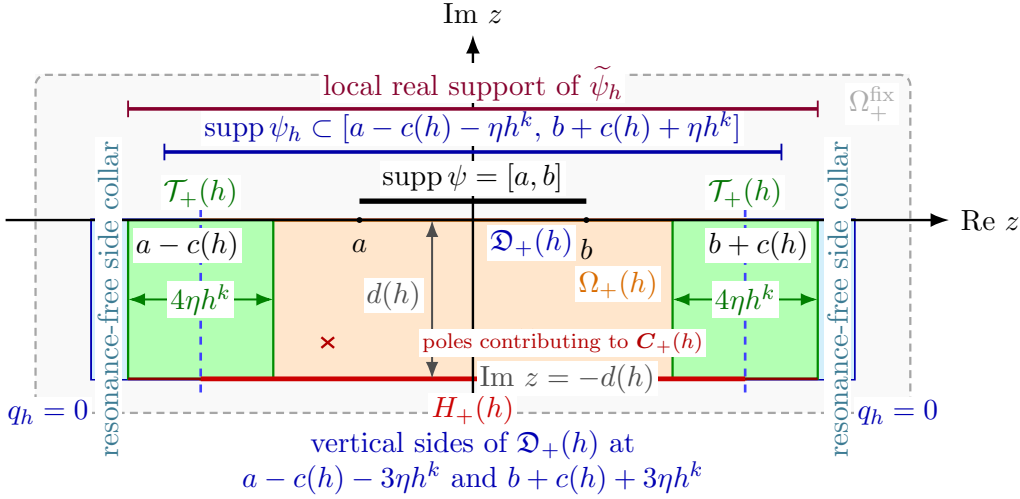

\noindent
The geometry in Figure~\ref{fig:myfig} is organised around the three terms
$\Ab_+(h)$, $\Bb_+(h)$, and $\Cb_+(h)$ in \eqref{ZCMM2025b:res-exp-positiveEnergy-formula}.  The \emph{orange} rectangle is the admissible contour domain
\[
  \Omega_+(h)=(a-c(h),b+c(h))-i(0,d(h)),
\]
whose vertical sides were selected by the pigeonhole argument and satisfy the
separation estimate \eqref{ZCMM2025b:vertical-gap}.  The \emph{dark-blue}
rectangle is the larger Cauchy--Green deformation rectangle
\[
  \DF_+(h)=[a-c(h)-3\eta h^k,b+c(h)+3\eta h^k]+i[-d(h),0].
\]
The portions of $\DF_+(h)$ lying immediately to the left and right of
$\Omega_+(h)$ are coloured \emph{light blue}.  They are the additional side
collars required by the support of the almost-analytic cutoff.  Since
$\eta<c_v/10$, these collars remain inside the $c_vh^k$ resonance-free
neighbourhood of the vertical sides of $\Omega_+(h)$.  Consequently enlarging
$\Omega_+(h)$ to $\DF_+(h)$ introduces no additional poles, and the residue
sum is still precisely over $\Res(\Dl(h))\cap\Omega_+(h)$.

The two \emph{green} strips are the sets on which $\bar\partial\widetilde{\psi}_h$ 
may be nonzero in the downward deformation, namely $\Tc_+(h)$.  In particular, 
their full horizontal width is $4\eta h^k$ rather than $2\eta h^k$.  These strips are 
the only parts of $\DF_+(h)$ contributing to the area integral $\Bb_+(h)$.  The 
horizontal support bars above the real axis display the successive support 
enlargements used in the proof: $\supp \psi=[a,b]$, then the $\Oc(h^k)$ support 
of $\psi_h$, and finally the local real support of $\widetilde\psi_h$.  The latter ends 
at $a-c(h)-2\eta h^k$ and $b+c(h)+2\eta h^k$, whereas the vertical sides of $\DF_+(h)$ 
are placed one further $\eta h^k$ away, at $a-c(h)-3\eta h^k$ and $b+c(h)+3\eta h^k$.  
This is the geometric reason that $q_h=e^{-itz/h}\widetilde{\psi}_h$ vanishes on the vertical 
sides and hence no vertical contour terms appear in the Cauchy--Green identity.

Finally, the \emph{red} horizontal segment is $H_+(h)$, the non-real
horizontal boundary of $\Omega_+(h)$.  The integral defining $\Ab_+(h)$ is
written over the full line $\im z=-d(h)$, but its integrand is compactly
supported in the real direction by $q_h$; thus only the portion lying under
the local real support of $\widetilde{\psi}_h$ contributes.  The red crosses are
schematic resonances inside $\Omega_+(h)$ and correspond to the residue term
$\Cb_+(h)$.  The dashed gray rectangle indicates the fixed admissible set
$\Omega_+^{\mathrm{fix}}$ containing $\DF_+(h)$; by
Definition~\ref{ZCMM2025b:Assumption-The-admissable-condition}, this fixed set
is itself compactly contained in the neighbourhood $\Uc$ on which the resonance
counting and resolvent estimates are applied.  The negative-energy geometry
used later is obtained by reflecting this picture across the real axis and
interchanging the physical and continued resolvent branches.

Next, we will show that the operator norms of $\Ab_{+}(h)$ and $\Bb_{+}(h)$ are negligible for large time.  
\newline
\noindent
Estimate of $\Ab_{+}(h)$. On the lower contour $\im z=-d(h)$, 
Theorem~\ref{ZCMM2025b:Thm-resol-est-away-critset}, together with the horizontal 
separation \eqref{ZCMM2025b:horizontal-gap}, gives the exponential bound for the 
continued branch $\Rb_+(z;h)$ with $S(h)$ a fixed multiple of $h^{M+3}$. The 
physical branch satisfies the self-adjoint estimate
\[
 \|\chi\Rb_-(z;h)\chi\|\le C d(h)^{-1}\leq C h^{-M},
\]
which is dominated, for small $h$, by the same exponential bound. Since 
$|e^{-itz/h}|=e^{-td(h)/h}$ on this contour and $\widetilde{\psi}_h$ has uniformly bounded real support, we obtain
\begin{equation}
\label{ZCMM2025b:EstimateOfA}
 \|\Ab_+(h)\|  \leq C\exp\!\left(-\frac{t d(h)}h+Ah^{-3}\log\frac1h\right).
\end{equation}
Since $d(h)\ge h^M$, if $t\ge h^{-L}$ and $L>M+7$, then
\[
 \frac{td(h)}h\ge h^{M-1-L}\gg h^{-3}\log(1/h),
\]
and hence \eqref{ZCMM2025b:EstimateOfA} is $\Oc(h^\infty)$, uniformly for $t\ge h^{-L}$.
\newline
\noindent
Estimate of $\Bb_+(h)$. Put
\[
 M_1:=M+5.
\]
By \eqref{ZCMM2025b:vertical-gap}, the two side collars in $\Tc_{+}(h)$ are resonance-free. After translating the rectangles and changing the harmless fixed numerical constants in Proposition~\ref{ZCMM2025b:Lemma-SMP}, that proposition applies there with $k_1=k$ and $k_2=M$.  For $z=x-is\in\Tc_+(h)$ with $0<s\le h^{M_1}$, Proposition~\ref{ZCMM2025b:Lemma-SMP} gives
\begin{equation}
\label{ZCMM2025b:side-polynomial-resolvent}
 \|\chi\Rb_+(z;h)\chi\|\le Ch^{-M_1},
\end{equation}
whereas self-adjointness of $\Dl(h)$ gives the physical-branch estimate
\begin{equation}
\label{ZCMM2025b:physical-lower-branch}
 \|\chi\Rb_-(x-is;h)\chi\|  \leq \|\chi\|_\infty^2 s^{-1}.
\end{equation}
We therefore estimate the two branches separately rather than imposing a uniform polynomial bound on their difference up to the real axis.

Split $\Bb_+(h)=\Bb_{+,1}(h)+\Bb_{+,2}(h)$ according to
\[
 -h^{M_1}\leq \im z\le0,  \qquad  -d(h)\leq \im z\leq -h^{M_1},
\]
respectively. On the first region, $|e^{-itz/h}|\le1$ because $t>0$ and $\im z\le0$, and the total horizontal width of the two collars is $\Oc(h^k)$. Put $s=-\im z$. For any $N\ge1$, \eqref{ZCMM2025b:property-2almostanalext}, \eqref{ZCMM2025b:side-polynomial-resolvent}, and \eqref{ZCMM2025b:physical-lower-branch} yield
\begin{align}
 \|\Bb_{+,1}(h)\|
 &\le C_N h^k\int_0^{h^{M_1}}
 \left(h^{-M_1}+s^{-1}\right)
 h^{-k}\left(\frac{s}{h^k}\right)^N ds\nonumber\\
 &\le C_N\left[
 h^{-M_1-kN}\int_0^{h^{M_1}}s^N\,ds
 +h^{-kN}\int_0^{h^{M_1}}s^{N-1}\,ds
 \right]\nonumber\\
 &\le C_N h^{N(M_1-k)}.
\label{ZCMM2025b:Bplus1-estimate}
\end{align}
The restriction $N\geq1$ is exactly what makes the physical $s^{-1}$ singularity integrable. 
Since $M_1=M+5>k$, given any $P>0$ we may choose the almost-analytic order $N$ so 
large that $N(M_1-k)\geq P$. Hence
\[
 \|\Bb_{+,1}(h)\|=\Oc(h^\infty).
\]
The negative-energy proof is treated by the reflected argument, with $s=\im z>0$ and the roles of the physical and continued branches interchanged.

On the deeper part $h^{M_1}\leq s=-\im z\leq d(h)$, the vertical separation \eqref{ZCMM2025b:vertical-gap} and the choice $\eta<c_v/10$ leave a distance comparable to $h^k$ between the side collars and every resonance. Hence Theorem~\ref{ZCMM2025b:Thm-resol-est-away-critset} applies to the continued branch with $S(h)=ch^k$ and gives
\[
 \|\chi\Rb_+(z;h)\chi\|  \leq C\exp\!\left(Ah^{-3}\log\frac1h\right).
\]
For the physical branch, self-adjointness gives $\|\chi\Rb_-(z;h)\chi\| \leq Cs^{-1}\leq Ch^{-M_1}$, which is again dominated by the same exponential bound for small $h$. Thus
\[
 \|\chi(\Rb_-(z;h)-\Rb_+(z;h))\chi\|  \leq C\exp\!\left(Ah^{-3}\log\frac1h\right).
\]
Taking $N=0$ in \eqref{ZCMM2025b:property-2almostanalext}, the factor $h^{-k}$ is 
cancelled by the $\Oc(h^k)$ total horizontal width of the two side collars. Consequently, 
\begin{align}
 \|\Bb_{+,2}(h)\|
 &\le C\exp\!\left(Ah^{-3}\log\frac1h\right)
 \int_{h^{M_1}}^{d(h)}e^{-ts/h}\,ds\nonumber\\
 &\le C\frac ht\exp\!\left(-t h^{M_1-1}+Ah^{-3}\log\frac1h\right).
\label{ZCMM2025b:Bplus2-estimate}
\end{align}
If $t\geq h^{-L}$ and $L>M_1+2=M+7$, then $t h^{M_1-1}\geq h^{M_1-1-L}$ dominates $h^{-3}\log(1/h)$ as $h\to0$. In fact the exponential factor in \eqref{ZCMM2025b:Bplus2-estimate} decays faster than any power of $h$, and therefore $\|\Bb_{+,2}(h)\|=\Oc(h^\infty)$ uniformly for $t\geq h^{-L}$.

Combining \eqref{ZCMM2025b:EstimateOfA}, \eqref{ZCMM2025b:Bplus1-estimate}, and \eqref{ZCMM2025b:Bplus2-estimate}, we obtain
\[
 \Ab_+(h)+\Bb_+(h)=\Oc(h^\infty)
\]
uniformly for $t\ge h^{-L}$ and every fixed $L>M+7$. Multiplying 
\eqref{ZCMM2025b:res-exp-positiveEnergy-formula} on the right by $\psi(\Dl(h))$ and using 
\eqref{ZCMM2025b:eqs-propagator-formula-2} gives \eqref{ZCMM2025b:MainThm2-pos-eq1}.
\end{proof}

\subsection*{Resonance expansion for negative energy}

\begin{theorem}
 Suppose that $\Dl(h)$ satisfies Assumption~\ref{ZCMM2025b:MainAssumpV}, let 
 $\Omega_-(h)$ be an admissible contour domain  in the sense of 
 Definition~\ref{ZCMM2025b:Assumption-The-admissable-condition}, and let 
 $\psi\in C_c^\infty(\R;\R)$ satisfy
\[
 \supp \psi=[a,b]\Subset(-\infty,-mc^2).
\]
Let $\chi\in C_c^\infty(\R^3;\R)$ with $\operatorname{supp}\chi\subset\{x\in\R^3:|x|\le R_0\}$. Then, for every fixed $L>M+7$, uniformly for $t\le-h^{-L}$,
\begin{equation}
\label{ZCMM2025b:MainThm2-neg-eq}
 \chi e^{-it\Dl(h)/h}\chi\psi(\Dl(h))
 =\sum_{z\in\Res(\Dl(h))\cap\Omega_-(h)} \chi\Res_{w=z} \left(e^{-itw/h}\Rb(w;h)\right)\chi\psi(\Dl(h)) +\Oc(h^\infty)
\end{equation}
in $\Bc(L^2(\R^3;\C^4))$. Here $\Rb(w;h)$ denotes the meromorphic continuation 
of the resolvent associated with the negative-energy sheet. If the meromorphic resolvent has a pole of order greater than one at a resonance $z$, then the corresponding residue term contains polynomial factors in $t/h$ arising from the higher-order principal part of the resolvent.
\label{ZCMM2025b:MainThm2-neg}
\end{theorem}

\begin{proof}
We give the reflected argument explicitly, because the contour orientation and the roles of the physical and continued branches are different from the positive-energy case. 
Construct $\psi_h$ and $\widetilde\psi_h$ as in \eqref{ZCMM2025b:psi-h-one}--\eqref{ZCMM2025b:property-2almostanalext}, now around the interval 
$[a,b]\Subset(-\infty,-mc^2)$. The same spectral-separation argument as 
in \eqref{ZCMM2025b:spectral-cutoff-separation}--\eqref{ZCMM2025b:eqs-propagator-formula-2} gives
\begin{equation}
\label{ZCMM2025b:eqs-propagator-formula-neg}
 \chi e^{-it\Dl(h)/h}\chi\psi(\Dl(h))
 =\chi e^{-it\Dl(h)/h}\psi_h(\Dl(h))\chi\psi(\Dl(h))+\Oc(h^\infty),
\end{equation}
uniformly in $t\in\R$. Recall that $\Rb(z;h)=(z-\Dl(h))^{-1}$. For negative energies, $\Rb_-(z;h)$ denotes the branch initially defined for $\im z<0$ and meromorphically continued through the negative continuous spectrum into the upper half-plane, while $\Rb_+(z;h)$ denotes the physical upper-half-plane branch. Stone's formula is therefore still written in the form
\[
 e^{-it\Dl(h)/h}\psi_h(\Dl(h))
 =\frac1{2\pi i}\int_{\R}e^{-it\lambda/h}
 \bigl(\Rb_-(\lambda;h)-\Rb_+(\lambda;h)\bigr)\psi_h(\lambda)\,d\lambda,
\]
in the cut-off sense. Set $q_h(z):=e^{-itz/h}\widetilde\psi_h(z)$ and introduce the closed upper deformation rectangle 
\begin{equation}
\label{ZCMM2025b:CG-rectangle-neg}
 \DF_-(h):=[a-c(h)-3\eta h^k,\,b+c(h)+3\eta h^k]+i[0,d(h)].
\end{equation}
By the real-support property \eqref{ZCMM2025b:aa-real-support}, $q_h$ vanishes on the two vertical boundary segments of $\DF_-(h)$, so these sides make no contribution. The portion of $\DF_-(h)$ above the real axis consists of the central contour domain $\Omega_-(h)$ together with two side collars of width $3\eta h^k$. Since $\eta<c_v/10$, the vertical separation \eqref{ZCMM2025b:vertical-gap} implies that the side collars contain no resonances. By the choice of the fixed admissible neighbourhood in Definition~\ref{ZCMM2025b:Assumption-The-admissable-condition}, after decreasing $h_0$ if necessary, 
\[
 \DF_-(h)\Subset\Omega_-^{\mathrm{fix}}\Subset\Uc.
\]
Hence the poles of the continued branch $\Rb_-$ enclosed by $\DF_-(h)$ are exactly the resonances in $\Omega_-(h)$, and the $\bar\partial$-area integral reduces to 
\[
 \Tc_-(h):=\Tc(h)\cap\{0\le\im z\le d(h)\}
\]
by \eqref{ZCMM2025b:aa-support}. Orient the upper horizontal boundary $\im z=d(h)$ from left to right. For the physical branch $\Rb_+$, which is holomorphic in the upper half-plane, Cauchy--Green gives 
\[
 \int_{\R}q_h(\lambda)\Rb_+(\lambda;h)\,d\lambda
 =\int_{\im z=d(h)}q_h(z)\Rb_+(z;h)\,dz
 +2i\iint q_h^{\,\bar\partial}(z)\Rb_+(z;h)\,L(dz),
\]
where 
\[
 q_h^{\,\bar\partial}(z):=e^{-itz/h}\bar\partial\widetilde\psi_h(z).
\]
For the branch $\Rb_-$ meromorphically continued into the upper half-plane, excising small disks around the enclosed poles yields 
\[
\begin{aligned}
 \int_{\R}q_h(\lambda)\Rb_-(\lambda;h)\,d\lambda
 &=\int_{\im z=d(h)}q_h(z)\Rb_-(z;h)\,dz
 +2i\iint q_h^{\,\bar\partial}(z)\Rb_-(z;h)\,L(dz)\\
 &\quad+2\pi i\sum_{z\in\Res(\Dl(h))\cap\Omega_-(h)}
 \Res_{w=z} \bigl(q_h(w)\Rb_-(w;h)\bigr).
\end{aligned}
\]
Subtracting the physical identity from the continued one and dividing by $2\pi i$ gives 
\[
\begin{aligned}
 &\frac1{2\pi i}\int_{\R}q_h(\lambda)
 \bigl(\Rb_-(\lambda;h)-\Rb_+(\lambda;h)\bigr)\,d\lambda\\
 &\qquad=\frac1{2\pi i}\int_{\im z=d(h)}q_h(z)
 \bigl(\Rb_-(z;h)-\Rb_+(z;h)\bigr)\,dz\\
 &\qquad\quad+\frac1\pi\iint q_h^{\,\bar\partial}(z)
 \bigl(\Rb_-(z;h)-\Rb_+(z;h)\bigr)\,L(dz)\\
 &\qquad\quad+\sum_{z\in\Res(\Dl(h))\cap\Omega_-(h)}
 \Res_{w=z} \bigl(q_h(w)\Rb_-(w;h)\bigr).
\end{aligned}
\]
Thus, with the upper horizontal contour written from left to right, all three terms in the reflected negative-energy decomposition carry positive signs. Applying this identity to the cut-off propagator and using that $\widetilde\psi_h\equiv1$ near the enclosed poles gives
\begin{equation}
\label{ZCMM2025b:negative-resonan-expan-formula}
 \chi e^{-it\Dl(h)/h}\psi_h(\Dl(h))\chi
 =\Ab_-(h)+\Bb_-(h)+\Cb_-(h),
\end{equation}
where 
\begin{align}
 \Ab_-(h)
 &:=\frac1{2\pi i}\int_{\im z=d(h)}e^{-itz/h}
 \chi\bigl(\Rb_-(z;h)-\Rb_+(z;h)\bigr)\chi\widetilde\psi_h(z)\,dz,
 \label{ZCMM2025b:Aminus}\\
 \Bb_-(h)
 &:=\frac1\pi\iint_{\Tc_-(h)}e^{-itz/h}
 \chi\bigl(\Rb_-(z;h)-\Rb_+(z;h)\bigr)\chi
 \bar\partial\widetilde\psi_h(z)\,L(dz),
 \label{ZCMM2025b:Bminus}\\
 \Cb_-(h)
 &:=\sum_{z\in\Res(\Dl(h))\cap\Omega_-(h)}
 \chi\Res_{w=z} \left(e^{-itw/h}\Rb(w;h)\right)\chi.
 \label{ZCMM2025b:Cminus}
\end{align}
The assumption excluding real poles ensures that the upward deformation crosses no pole on the real axis.
Put $T:=-t>0$. For $z=x+is$ in the upper half-plane, 
\[
 |e^{-itz/h}|=e^{-Ts/h}.
\]
We first estimate $\Ab_-(h)$. On $\im z=d(h)$, the horizontal separation \eqref{ZCMM2025b:horizontal-gap} controls the central part of the contour, while the vertical separation \eqref{ZCMM2025b:vertical-gap} controls the short extensions in the two side collars. Since $h^k\gg h^{M+3}$ under $k\le M-3$, the entire upper contour stays a distance at least $c h^{M+3}$ from every resonance, for a fixed $c>0$. Theorem~\ref{ZCMM2025b:Thm-resol-est-away-critset} therefore applies to the continued branch $\Rb_-$ with $S(h)=c h^{M+3}$. The physical branch satisfies the self-adjoint bound $\|\chi\Rb_+(z;h)\chi\|\le C d(h)^{-1}\le Ch^{-M}$. Hence 
\begin{equation}
\label{ZCMM2025b:EstimateOfAminus}
 \|\Ab_-(h)\|
 \le C\exp\!\left(-\frac{T d(h)}h+Ah^{-3}\log\frac1h\right).
\end{equation}
Since $d(h)\ge h^M$, this is $\Oc(h^\infty)$ uniformly for $T\ge h^{-L}$ whenever $L>M+7$. 
Put $M_1:=M+5$ and split 
\[
 \Bb_-(h)=\Bb_{-,1}(h)+\Bb_{-,2}(h)
\]
according to $0\le\im z\le h^{M_1}$ and $h^{M_1}\le\im z\le d(h)$. For $z=x+is\in\Tc_-(h)$ with $0<s\le h^{M_1}$, Proposition~\ref{ZCMM2025b:Lemma-SMP}, in its negative-energy form, gives 
\[
 \|\chi\Rb_-(x+is;h)\chi\|\le Ch^{-M_1},
\]
whereas self-adjointness gives for the physical branch 
\[
 \|\chi\Rb_+(x+is;h)\chi\|\le \|\chi\|_\infty^2s^{-1}.
\]
Since $e^{-Ts/h}\le1$, the total horizontal width of the two transition strips is $\Oc(h^k)$, and \eqref{ZCMM2025b:property-2almostanalext} gives, for every $N\ge1$,
\begin{align}
 \|\Bb_{-,1}(h)\|
 &\le C_Nh^k\int_0^{h^{M_1}}
 \left(h^{-M_1}+s^{-1}\right)h^{-k}
 \left(\frac{s}{h^k}\right)^N ds\nonumber\\
 &\le C_N h^{N(M_1-k)}.
 \label{ZCMM2025b:Bminus1-estimate}
\end{align}
The condition $N\ge1$ makes the $s^{-1}$ singularity integrable. Since $M_1>k$, choosing the almost-analytic order $N$ arbitrarily large gives $\|\Bb_{-,1}(h)\|=\Oc(h^\infty)$.

On the deeper part $h^{M_1}\le s=\im z\le d(h)$, the vertical gap \eqref{ZCMM2025b:vertical-gap} and the choice $\eta<c_v/10$ leave a distance comparable to $h^k$ between the transition strips and every resonance. Theorem~\ref{ZCMM2025b:Thm-resol-est-away-critset} therefore gives for the continued branch
\[
 \|\chi\Rb_-(z;h)\chi\|
 \le C\exp\!\left(Ah^{-3}\log\frac1h\right),
\]
while the physical branch satisfies $\|\chi\Rb_+(z;h)\chi\|\le Cs^{-1}\le Ch^{-M_1}$, which is dominated by the same exponential bound for small $h$. Taking $N=0$ in \eqref{ZCMM2025b:property-2almostanalext}, the factor $h^{-k}$ is cancelled by the $\Oc(h^k)$ total horizontal width of the transition strips. Therefore,
\begin{align}
 \|\Bb_{-,2}(h)\|
 &\le C\exp\!\left(Ah^{-3}\log\frac1h\right)
 \int_{h^{M_1}}^{d(h)}e^{-Ts/h}\,ds\nonumber\\
 &\le C\frac hT\exp\!\left(-T h^{M_1-1}+Ah^{-3}\log\frac1h\right).
 \label{ZCMM2025b:Bminus2-estimate}
\end{align}
If $T\ge h^{-L}$ and $L>M_1+2=M+7$, the negative term $-T h^{M_1-1}$ dominates $Ah^{-3}\log(1/h)$ by more than any fixed multiple of $\log(1/h)$. Hence $\|\Bb_{-,2}(h)\|=\Oc(h^\infty)$ uniformly in this time range.

Combining \eqref{ZCMM2025b:EstimateOfAminus}, \eqref{ZCMM2025b:Bminus1-estimate}, and \eqref{ZCMM2025b:Bminus2-estimate} gives
\[
 \Ab_-(h)+\Bb_-(h)=\Oc(h^\infty)
\]
uniformly for $t\le-h^{-L}$ and every fixed $L>M+7$. Multiplying \eqref{ZCMM2025b:negative-resonan-expan-formula} on the right by $\psi(\Dl(h))$ and using \eqref{ZCMM2025b:eqs-propagator-formula-neg} proves \eqref{ZCMM2025b:MainThm2-neg-eq}.
\end{proof}

\begin{remark}
The opposite time directions in Theorems~\ref{ZCMM2025b:MainThm2-pos} and 
\ref{ZCMM2025b:MainThm2-neg} are dictated by the location of the poles. Indeed,
\[
 |e^{-itz/h}|=e^{t \im z/h}.
\]
Positive-energy resonances lie on the continued sheet in the lower half-plane, so their contributions decay for $t>0$. Negative-energy resonances considered here lie on the reflected continued sheet in the upper half-plane, so their contributions decay for $t<0$.
\end{remark}

\section{Resonance-free resolvent bounds and local-energy decay for mass-type potentials}
\label{zcmm2025b-sec-masstype-resexp}
Throughout this section $V:\R^3\to\R$ is a scalar, real-valued function satisfying the scalar version of Assumption~\ref{ZCMM2025b:MainAssumpV}; the matrix-valued perturbation is therefore
\[
\Vb(x)=\mbf{\beta}V(x).
\]
This mass-type structure is essential below.  We write
\[
m(x):=m+\frac{V(x)}{c^2},\qquad
p_\pm(x,\xi)=\pm\sqrt{m(x)^2c^4+c^2|\xi|^2},
\]
so that $p_+$ and $p_-$ are the two eigenvalue branches of the principal symbol, each of multiplicity two.

Let $I_+=[\lambda-\delta,\lambda+\delta]\Subset(mc^2,\infty)$ be contained in a slightly larger compact interval $I_+^{\rm out}\Subset(mc^2,\infty)$ on which the positive branch is nontrapping in the sense of Definition~\ref{ZCMM2025b:Def-non-trapping-energy} below. By \cite[Theorem~5.11]{ZCMM2025a}, after shrinking $I_+^{\rm out}$ if necessary, there is a fixed constant $C_{\rm rf}>0$ such that the positive-energy resonance-free strip has depth $C_{\rm rf}h\log(1/h)$.

Fix
\[
0<M<C_{\rm rf},\qquad \gamma(h):=h\log(1/h),
\]
and define
\begin{equation}
\Omega_M^+(h):=I_+-i[0,M\gamma(h)].
\label{eqs-resonance-domain+}
\end{equation}
Then $\Omega_M^+(h)$ is resonance-free for all sufficiently small $h$.  No assertion is made here for arbitrary $M$.

For negative energies we use the reflected interval $I_-:=-I_+\Subset(-\infty,-mc^2)$; equivalently, an independent negative interval may be treated after reflection.

The symmetry proved in \cite[Remarks~5.2 and~5.12]{ZCMM2025a} gives the reflected resonance-free strip
\begin{equation}
\Omega_M^-(h):=I_-+i[0,M\gamma(h)].
\label{eqs-resonance-domain-}
\end{equation}
Thus the positive and negative windows remain strictly separated from the thresholds $\pm mc^2$.

We next record the classical dynamics on the separated energy sheets. On either separated sheet, let
\[
\mbf{\phi}_{\pm}(t,x,\xi)=(x_\pm(t),\xi_\pm(t))
\]
denote the Hamiltonian flow generated by $p_\pm$.  The trajectories satisfy Hamilton's equations
\begin{equation}
\begin{cases}
\dot x_\pm(t)=\nabla_\xi p_\pm(x_\pm(t),\xi_\pm(t)),\\[2mm]
\dot\xi_\pm(t)=-\nabla_xp_\pm(x_\pm(t),\xi_\pm(t)),\\[2mm]
(x_\pm(0),\xi_\pm(0))=(x,\xi).
\end{cases}
\label{ZCMM2025b:Hamilton-equations-mass}
\end{equation}
The flow is used only on compact energy intervals contained in $(mc^2,\infty)$ for $p_+$, or in 
$(-\infty,-mc^2)$ for $p_-$, where the two branches are separated and smooth.  This avoids the possible 
crossing at points where $m(x)c^2=0$ and $\xi=0$.

We formulate nontrapping on an energy interval, since the estimates below are uniform in the spectral parameter.

\begin{definition}[Nontrapping on an energy interval]
\label{ZCMM2025b:Def-non-trapping-energy}
Let $\sigma\in\{+,-\}$ and let $I_\sigma$ be a compact interval contained in the corresponding separated energy region.  We say that the branch $p_\sigma$ is \emph{nontrapping on $I_\sigma$} if, for every $(x_0,\xi_0)\in p_\sigma^{-1}(I_\sigma)$ and every compact set $K\Subset T^*\R^3$, there exists $T=T(x_0,\xi_0,K)>0$ such that
\[
\mbf{\phi}_\sigma(t,x_0,\xi_0)\notin K,\qquad |t|>T.
\]
Equivalently, every trajectory on $p_\sigma^{-1}(I_\sigma)$ escapes every compact subset of phase space as $|t|\to\infty$.  
\end{definition}

In what follows the positive-energy argument is formulated for $p_+$ on $I_+^{\rm out}$; the negative-energy statement 
is obtained by reflection.

\subsection{A quantitative resolvent bound in the logarithmic strip}
\label{ZCMM2025b:FBIargument} 
We next derive a quantitative cut-off resolvent estimate in a fixed smaller logarithmic resonance-free strip.  The argument is obtained by retaining the inhomogeneous right-hand side in the FBI-transform/escape-function estimates underlying \cite[Theorem~5.11]{ZCMM2025a}.

\begin{lemma}[Polynomial cut-off resolvent bound in the resonance-free strip]
\label{ZCMM2025a: lemma-bottom-bound}
Assume that $p_+$ is nontrapping on $I_+^{\rm out}$ and let $0<M<C_{\rm rf}$.  Let $\chi\in C_0^\infty(\R^3)$ be supported in the undistorted interaction region.  Then there exist constants $N_0\geq0$, $C>0$ and $h_0>0$ such that, for $0<h<h_0$, the continuation from the upper half-plane of the cut-off resolvent,
\[
\Rb_{\chi,+}^{\rm cont}(z;h) = \chi \Rb^{\mathrm{cont}}(z;h) \chi, 
\]
is holomorphic in an open neighbourhood of the compact strip $\Omega_M^+(h)$ and satisfies
\begin{equation}
\sup_{z\in\Omega_M^+(h)}
\|\Rb_{\chi,+}^{\rm cont}(z;h)\|_{\Bc(L^2(\R^3;\C^4))}
\leq C h^{-N_0}.
\label{ZCMM2025b:quantitative-strip-resolvent}
\end{equation}
The distorted operator is understood as
\[
\Dl_\theta(h)-z:H^1(\R^3;\C^4)\longrightarrow L^2(\R^3;\C^4),
\]
with the cut-off continuation obtained from its inverse by analytic distortion.  The reflected estimate holds on $\Omega_M^-(h)$.
\end{lemma}

\begin{proof}
Choose $M_1$ with $M<M_1<C_{\rm rf}$ and choose a compact interval $J$ with 
$I_+\Subset\operatorname{int}J\Subset I_+^{\rm out}$.  After reducing $h_0$ if necessary, 
the resonance-free result of \cite[Theorem~5.11]{ZCMM2025a} applies on the larger closed 
strip $J-i[0,M_1\gamma(h)]$.  The distorted resolvent is meromorphic across the real interval 
$J$, agrees with the physical resolvent in the upper half-plane, and has no pole on this larger 
strip.  Hence it is holomorphic across $J$ and throughout a neighbourhood of the compact set 
$\Omega_M^+(h)$.

We indicate the quantitative point that is not written explicitly in \cite{ZCMM2025a}.  In the proof of 
Theorem~5.11 there, one takes
\[
\eta=C_1h\log(1/h),\qquad t_{\rm FBI}=C_2\eta,
\]
and sets $\wb=e^{-t_{\rm FBI}H/h}\Tb \ub$ on the deformed FBI manifold.  The phase-space decomposition into the characteristic set $\Sigma$ and its elliptic complement $\Sigma^c$, together with the spectral projectors 
onto the two Dirac sheets, yields the homogeneous estimate used in \cite[Section~5.3]{ZCMM2025a} 
to prove that the kernel is trivial. For the inhomogeneous equation
\[
(\Dl_\theta(h)-z)\ub=\fb,
\]
the same estimates retain the transformed right-hand side and give, uniformly for 
$z$ in any fixed smaller strip $\Omega_M^+(h)$ with $M<C_{\rm rf}$,
\begin{equation}
X\leq C\eta^{-1}\|e^{-t_{\rm FBI}H/h}\Tb \fb \|
   +C\Big(\frac{h}{\eta}+h\Big)X,
\qquad
X:=\|\langle\xi\rangle e^{-t_{\rm FBI}H/h} \Tb \ub \|_{L^2(\Lambda_{t_{\rm FBI}G})}.
\label{ZCMM2025b:inhomogeneous-FBI-estimate}
\end{equation}
A complete derivation of \eqref{ZCMM2025b:inhomogeneous-FBI-estimate}, with the forcing term retained separately on 
$\Sigma$ and $\Sigma^c$, is supplied in Appendix~\ref{ZCMM2025b:appendix-inhomogeneous-FBI}. Since $h/\eta=1/(C_1\log(1/h))\to0$, the last term is absorbed for $h$ small.  The escape function $H$ is bounded on the compact phase-space region where the deformation is nontrivial, while  $t_{\rm FBI}/h=\Oc(\log(1/h))$.  Consequently the weighted FBI maps and their inverse lose at most a fixed power of $h$: for some $K_0>0$,
\[
\|e^{-t_{\rm FBI}H/h} \Tb \fb\|\leq C h^{-K_0}\| \fb \|,\qquad \| \ub \|\leq C h^{-K_0}X.
\]
Combining these inequalities with \eqref{ZCMM2025b:inhomogeneous-FBI-estimate} gives, for a fixed $N_0$,
\[
\|\ub \|_{L^2}\leq C h^{-N_0}\| \fb \|_{L^2}.
\]
Since $(z \mbf{I} -\Dl_\theta(h))^{-1}=-(\Dl_\theta(h)-z \mbf{I} )^{-1}$, the same norm bound holds for the resolvent convention used throughout Sections~3--5. The distorted family is Fredholm of index zero in the deformation region, and \cite[Theorem~5.11]{ZCMM2025a} excludes eigenvalues there; hence the inverse exists.  Since $z \mbf{I} -\Dl_\theta(h):H^1(\R^3;\C^4)\to L^2(\R^3;\C^4)$ is a closed Fredholm operator of index zero and is injective in the resonance-free strip, it is bijective there.  Its inverse therefore maps $L^2$ into the operator domain $H^1$; for the present lemma we retain only the resulting $L^2\to L^2$ polynomial bound.

Finally, for cutoffs supported where the distortion is the identity, equality with the physical cut-off resolvent holds first in the physical half-plane and then, by uniqueness of analytic continuation, gives the continued cut-off resolvent in $\Omega_M^+(h)$.  This proves \eqref{ZCMM2025b:quantitative-strip-resolvent}.
\end{proof} 

By the $z\mapsto-z$ symmetry established for mass-type potentials in \cite[Remarks~5.2 and~5.12]{ZCMM2025a}, the preceding argument immediately yields the corresponding polynomial bound in $\Omega_M^-(h)$.

The quantitative derivation of \eqref{ZCMM2025b:inhomogeneous-FBI-estimate}, obtained by retaining the transformed forcing term in the $\Sigma$ and $\Sigma^c$ estimates of \cite[Theorem~5.11]{ZCMM2025a}, is given in Appendix~\ref{ZCMM2025b:appendix-inhomogeneous-FBI}.

\subsection{Localised propagation in a nontrapping mass-type window}
In this subsection, we retain the convention
\[
 \Rb(z;h):=(z-\Dl(h))^{-1},\qquad \im z\ne0.
\]
For $\im z>0$ in the positive-energy physical half-plane and for cutoffs $\chi$ supported in the undistorted region,
\[
 \chi\Rb(z;h)\chi=\chi(z-\Dl_\theta(h))^{-1}\chi.
\]
The right-hand side provides the meromorphic continuation of the left-hand side through $I_+$; we denote this continuation by $\Rb_{\chi,+}^{\rm cont}(z;h)$.  Thus the equality in the lower half-plane is an identity of analytic continuations, not a literal unitary-conjugation identity for non-real $\theta$.  On the physical resolvent set, self-adjointness gives the estimate
\begin{equation}
 \|\chi\Rb(z;h)\chi\|_{\Bc(L^2)}
 \le \|\chi\|_\infty^2\,|\im z|^{-1},\qquad \im z\ne0.
 \label{ZCMM2025b:physical-resolvent-bound}
\end{equation}
This is the physical-branch estimate used below.

\begin{theorem}[localised propagation in the resonance-free logarithmic strip]
\label{ZCMM2025b:MainThm1-masstype}
Assume the hypotheses of Lemma~\ref{ZCMM2025a: lemma-bottom-bound}.  Let
$\chi\in C_0^\infty(\R^3)$ and $\psi\in C_0^\infty(\R;\R)$ satisfy
\[
\supp\chi\subset\{|x|\leq R_0\},\qquad
\supp\psi\Subset I_+.
\]
Fix $0<M<C_{\rm rf}$.  Then, for every $N\in\N$, there exist constants $T_N>0$, $C_N>0$, and $h_N>0$ such that
\begin{equation}
\|\chi e^{-it\Dl(h)/h}\psi(\Dl(h))\chi\|_{\Bc(L^2(\R^3;\C^4))}
\leq C_N h^N,
\qquad t\geq T_N,\quad 0<h<h_N.
\label{ZCMM2025b:res_exp_res_free_eqs_pos}
\end{equation}
For the reflected negative-energy interval $I_-=-I_+$ and $\psi_-\in C_0^\infty(I_-)$, one has, for every $N\in\N$, constants $T_N,C_N,h_N>0$ such that
\begin{equation}
\|\chi e^{-it\Dl(h)/h}\psi_-(\Dl(h))\chi\|_{\Bc(L^2(\R^3;\C^4))}
\leq C_N h^N,
\qquad t\leq -T_N,\quad 0<h<h_N.
\label{ZCMM2025b:res_exp_res_free_eqs_neg}
\end{equation}
One may, for example, choose
\[
 T_N>\max\left\{1,\frac{N+\widehat N_0+1}{M}\right\},
 \qquad \widehat N_0:=\max\{N_0,2\},
\]
where $N_0$ is the exponent in Lemma~\ref{ZCMM2025a: lemma-bottom-bound}.
\end{theorem}

\begin{proof}
We first treat positive energies.   Let
\[
\Rb_\pm(\lambda;h):=\lim_{\varepsilon\downarrow0}\Rb(\lambda\pm i\varepsilon;h)
\]
be the physical boundary values.  With this convention Stone's formula reads
\begin{equation}
\chi e^{-it\Dl(h)/h}\psi(\Dl(h))\chi
=\frac{1}{2\pi i}\int_{\R}e^{-it\lambda/h}\psi(\lambda)
\chi\big(\Rb_-(\lambda;h)-\Rb_+(\lambda;h)\big)\chi\,d\lambda.
\label{ZCMM2025b:eqs-propagator-formula-1}
\end{equation}
Choose an almost-analytic extension $\widetilde\psi\in C_0^\infty(\C)$, supported in a fixed small complex neighbourhood of $\supp\psi\Subset I_+$, such that for every $L\in\N$
\begin{equation}
|\bar\partial\widetilde\psi(z)|\leq C_L|\im z|^L.
\label{zcmm2025b:almost-analytic-estimate}
\end{equation}
Let $\Rb_{\chi,+}^{\rm cont}(z;h)$ denote the continuation of the upper-half-plane cut-off resolvent, and let
\[
\Rb_{\chi,-}^{\rm phys}(z;h):=\chi\Rb(z;h)\chi,\qquad \im z<0.
\]
The two boundary-value terms in \eqref{ZCMM2025b:eqs-propagator-formula-1} are deformed separately.  In the lower strip the upper-half-plane branch becomes $\Rb_{\chi,+}^{\rm cont}$, while the lower-half-plane branch remains the physical resolvent $\Rb_{\chi,-}^{\rm phys}$.  Since \cite[Theorem~5.11]{ZCMM2025a} gives no pole in $\Omega_M^+(h)$, no residue term is encountered.  Because $\widetilde\psi$ is supported strictly inside the real-energy window, the vertical sides vanish.  
Orient the lower contour from left to right and use the same $\bar\partial$ convention as in Section~4.  The top real-axis jump in Stone's formula is $\Rb_{\chi,-}^{\rm phys}-\Rb_{\chi,+}^{\rm cont}$.  Cauchy--Green therefore gives the exact identity
\begin{align}
\chi e^{-it\Dl(h)/h}\psi(\Dl(h))\chi
={}&\frac{1}{2\pi i}\int_{\im z=-M\gamma(h)}
 e^{-itz/h}\widetilde\psi(z)
 \big(\Rb_{\chi,-}^{\rm phys}(z;h)-\Rb_{\chi,+}^{\rm cont}(z;h)\big)\,dz
\nonumber\\
&-\frac{1}{\pi}\iint_{-M\gamma(h)<\im z<0}
 e^{-itz/h}\bar\partial\widetilde\psi(z)
 \big(\Rb_{\chi,-}^{\rm phys}(z;h)-\Rb_{\chi,+}^{\rm cont}(z;h)\big)\,dm(z).
\label{ZCMM2025b:res_exp_resfre_pos_cauchyF}
\end{align}
We denote the bottom-contour term by $\Bb(h)$ and the area term by $\Ab(h)$.
Set $\widehat N_0:=\max\{N_0,2\}$.  On the bottom contour, Lemma~\ref{ZCMM2025a: lemma-bottom-bound} and the self-adjoint estimate \eqref{ZCMM2025b:physical-resolvent-bound} imply
\[
\|\Rb_{\chi,+}^{\rm cont}(z;h)-\Rb_{\chi,-}^{\rm phys}(z;h)\|
\leq C h^{-\widehat N_0}.
\]
Since $|e^{-itz/h}|=e^{t\im z/h}=h^{Mt}$ on $\im z=-M\gamma(h)$, we obtain
\begin{equation}
\|\Bb(h)\|\leq C h^{Mt-\widehat N_0}.
\label{ZCMM2025b:res_exp_resfre_pos_estB}
\end{equation}
Thus, given a prescribed power $N$, the bottom term is $\Oc(h^N)$ as soon as
\[
t\geq T_N>\frac{N+\widehat N_0}{M}.
\]
The time threshold necessarily depends on $N$.
For the area term write $s=-\im z\in(0,M\gamma(h))$.  Lemma~\ref{ZCMM2025a: lemma-bottom-bound}, \eqref{ZCMM2025b:physical-resolvent-bound}, and 
\eqref{zcmm2025b:almost-analytic-estimate} give
\begin{align}
\|\Ab(h)\|
&\leq C_L\int_0^{M\gamma(h)}e^{-ts/h}s^L
\big(h^{-N_0}+s^{-1}\big)\,ds \nonumber\\
&\leq C_L\left[h^{-N_0}\Big(\frac{h}{t}\Big)^{L+1}
+\Big(\frac{h}{t}\Big)^L\right],\qquad t\geq1.
\label{ZCMM2025b:res_exp_resfre_pos_estA}
\end{align}
Given $N$, choose the almost-analytic order $L>N+N_0+1$.  Then $\|\Ab(h)\|=\Oc(h^N)$ uniformly for $t\geq1$.  Combining this with the bottom estimate proves \eqref{ZCMM2025b:res_exp_res_free_eqs_pos}.

For negative energies we reflect the contour into the upper half-plane and take $t<0$.
The $z\mapsto-z$ symmetry from \cite[Remarks~5.2 and~5.12]{ZCMM2025a} gives the same polynomial continued-resolvent bound in $\Omega_M^-(h)$.  On the top contour 
$\im z=M\gamma(h)$ and $t<0$,
\[
|e^{-itz/h}|=e^{t\im z/h}=h^{M|t|}.
\]
Repeating the preceding estimates with $|t|$ in place of $t$ gives $\Oc(h^N)$ whenever 
$t\leq -T_N$.  This proves \eqref{ZCMM2025b:res_exp_res_free_eqs_neg}.
\end{proof} 

\section{Finite-time propagation and geometric cut-off resolvent bounds}
\label{ZCMM2026new:SecFiniteTime}
The polynomial estimate of Lemma~\ref{ZCMM2025a: lemma-bottom-bound} is sufficient for the local-energy decay argument of Section~\ref{zcmm2025b-sec-masstype-resexp}, but its exponent is inherited from the comparison between the exponentially weighted FBI space and ordinary $L^2$. In this section we refine the 
size of that estimate by finite-time propagation in the undistorted region.  The coarse FBI bound is still used 
later only to guarantee $h$-temperedness of the exact outgoing solution. The resulting bound is expressed 
in terms of a classical connection time. Before proving it, we give a rigorous obstruction to an order $[h\log(1/h)]^{-1}$ bound for the full distorted resolvent and then describe the finite-flight-time amplification mechanism which makes such a pure bound an overly optimistic general target for the sandwiched continued resolvent.

Throughout the section we keep the mass-type hypotheses and notation of Section~\ref{zcmm2025b-sec-masstype-resexp}.  Thus
\[
 \gamma(h)=h\log(1/h),\qquad
 p_+(x,\xi)=\sqrt{c^2|\xi|^2+(mc^2+V(x))^2},
\]
and $p_+$ is nontrapping on $I_+^{\rm out}\Subset(mc^2,\infty)$.  Choose an interval $I^\sharp$ such that 
\begin{equation}
 I_+\Subset I^\sharp\Subset I_+^{\rm out}.
 \label{ZCMM2026new:Isharp}
\end{equation}
(One could, if desired, use the interval $J$ chosen in the proof of Lemma~\ref{ZCMM2025a: lemma-bottom-bound} as $I^\sharp$). We write
\[
 \Phi_t:=\mbf{\phi}_+(t,\cdot)
\]
for the positive-sheet Hamilton flow.

\subsection{Obstructions to a pure $\gamma(h)^{-1}$ bound}
\label{ZCMM2026new:SubsecObstruction}
We first record the microlocal obstruction found by testing the actual distorted positive branch.  Let
\[
 \theta=i\nu\gamma(h),\qquad z=E-iM\gamma(h),\qquad E\in I_+,
\]
with $\nu>0$ fixed.  On a fixed bounded-frequency neighbourhood of $p_+^{-1}(I^\sharp)$, the positive eigenvalue of the distorted principal symbol has the expansion
\begin{equation}
 p_{+,\theta}(x,\xi)-z
 =p_+(x,\xi)-E+i\gamma(h)b_E(x,\xi) + \Oc(\gamma(h)^2),
 \label{ZCMM2026new:actual-branch-expansion}
\end{equation}
where
\begin{equation}
 b_E(x,\xi):=M-\nu H_{p_+}(\gb(x)\cdot\xi) 
 \label{ZCMM2026new:bE}
\end{equation}
with $\gb$ being the distortion vector field introduced in Section~\ref{ZCMM2025b:SecPrelim}. Thus $H_{p_+}(\gb\cdot\xi)$ 
measures, to first order in the imaginary distortion parameter, the 
imaginary displacement of the positive eigenvalue produced by 
complex scaling. In the undistorted region, $\gb=0$ and hence 
$b_E=M>0$.  On the other hand, for $|x|$ sufficiently large, $\gb(x)=x$ and
\begin{equation}
 H_{p_+}(x\cdot\xi)
 =\frac{c^2|\xi|^2-(mc^2+V(x))x\cdot\nabla V(x)}{p_+(x,\xi)}.
 \label{ZCMM2026new:Hpxxi}
\end{equation}
Along $p_+=E$ the right-hand side converges to
\begin{equation}
 a_\infty(E):=\frac{E^2-m^2c^4}{E}>0.
 \label{ZCMM2026new:ainfty}
\end{equation}
The asymptotic distorted positive continuum through the energy $E$ has imaginary part
\[
 -\nu a_\infty(E)\gamma(h)+\Oc(\gamma(h)^2).
\]
Consequently, in order to expose the line $\im z=-M\gamma(h)$ with a fixed margin, one necessarily requires
\begin{equation}
 \nu a_\infty(E)>M.
 \label{ZCMM2026new:exposure}
\end{equation}
Thus $b_E$ is positive in the physical region and negative on the outgoing end of the energy surface.  Every outgoing trajectory which starts in the undistorted region and tends to infinity therefore meets the transition set
\begin{equation}
 \mathcal C_{M,E}:=
 \{(x,\xi):p_+(x,\xi)=E,\ \nu H_{p_+}(\gb\cdot\xi)=M\}.
 \label{ZCMM2026new:CM}
\end{equation}

The next proposition describes the generic transverse case.  It is not used in the proof of the positive result below; its role is to show that the failure of a $C/\gamma(h)$ estimate for the full distorted resolvent is intrinsic rather than merely a defect of the FBI argument in Section~\ref{ZCMM2025b:FBIargument}. As usual, $\Psi_h^0(\R^3;\C^4)$ denotes the class of semiclassical pseudodifferential operators of order zero acting on $\C^4$-valued spinors.

\begin{proposition}[Microlocal obstruction for the full distorted resolvent]
\label{ZCMM2026new:PropFullObstruction}
Assume \eqref{ZCMM2026new:exposure}, and suppose that an outgoing positive-energy bicharacteristic meets a point $\rho_*\in\mathcal C_{M,E}$ at which
\begin{equation}
 (H_{p_+}b_E) (\rho_*)<0.
 \label{ZCMM2026new:transverse-crossing}
\end{equation}
Then there are $L^2$-normalized spinors $u_h$, microlocalised arbitrarily close to $\rho_*$, such that
\begin{equation}
 \|(\Dl_\theta(h)-z)u_h\|_{L^2}\le Ch.
 \label{ZCMM2026new:full-quasimode}
\end{equation}
In particular, whenever $\Dl_\theta(h)-z$ is invertible,
\begin{equation}
 \|(\Dl_\theta(h)-z)^{-1}\|_{L^2\to L^2}\ge c h^{-1}.
 \label{ZCMM2026new:full-lower-bound}
\end{equation}
Hence a uniform $C/\gamma(h)$ bound for the full distorted inverse cannot hold at such exposing transverse crossings.
\end{proposition}

\begin{proof}
Because the positive and negative Dirac sheets are uniformly separated on $p_+^{-1}(I^\sharp)$, the positive rank-two eigenbundle is smooth near $\rho_*$.  Microlocal block diagonalization reduces the positive block, up to $h\Psi_h^0$ terms, to a scalar principal-type operator times 
$\mbf{I}_2$. More explicitly, let $\db_{\theta,{\rm pr}}(x,\xi)$ denote the distorted principal 
symbol.  On a sufficiently small conic neighbourhood $\Uc_*$ of $\rho_*$ the two eigenvalues $p_{+,\theta}$ and $p_{-,\theta}$ satisfy $|p_{+,\theta}-p_{-,\theta}|\ge c_*>0$.  Hence the 
Riesz projectors $\Pi_{\pm,\theta}$ are smooth matrix symbols on $\Uc_*$.  Choose smooth 
local frames $e_\pm(x,\xi;\theta):\C^2\to\C^4$ and dual frames 
$\ell_\pm(x,\xi;\theta):\C^4\to\C^2$ such that 
\[
\ell_\sigma e_\tau=\delta_{\sigma\tau} \mbf{I}_2, \qquad e_\sigma\ell_\sigma=\Pi_{\sigma,\theta},
 \qquad \sigma,\tau\in\{+,-\}.
\]
For complex $\theta$ these frames need not be unitary; what matters is that the 
$4\times4$ matrix $\Eb_0=(e_+,e_-)$ is elliptic and has inverse $\Lb_0=(\ell_+;\ell_-)$.  
Pointwise one then has the exact diagonalization
\[
 \Lb_0(\mbf{d}_{\theta,{\rm pr}}-z \mbf{I}_4) \Eb_0=
 \begin{pmatrix}
 (p_{+,\theta}-z)\mbf{I}_2&0\\[1mm]
 0&(p_{-,\theta}-z) \mbf{I}_2
 \end{pmatrix}.
\]
We then quantize $\Eb_0$ and $\Lb_0$ semiclassically.  Symbolic composition gives
\[
 \Lb_0^w(\Dl_\theta-z)\Eb_0^w=
 \begin{pmatrix}
 \operatorname{Op}_h^w(p_{+,\theta}-z)&0\\
 0&\operatorname{Op}_h^w(p_{-,\theta}-z)
 \end{pmatrix}+h \Rb_0,  \qquad \Rb_0\in \Psi_h^0.
\]
If one wants an actual block diagonalization to higher order, correct $\Eb_0$ successively by symbols $\Eb_0+h\Eb_1+h^2\Eb_2+\cdots$.  At each step the off-diagonal homological equation is solved by division by $p_{+,\theta}-p_{-,\theta}$; the fixed gap $c_*>0$ makes this division harmless.  For the present proposition the first step is sufficient, since all such off-diagonal terms are $\Oc(h)=o(\gamma(h))$.  The negative block is elliptic at $z\simeq E>mc^2$ and can therefore be eliminated by a local Schur complement, again changing the positive block only by $\Oc(h)$.  

Here and below in the local normal-form coordinates we use the standard convention 
$D_s=(1/i)\partial_s$ (and similarly $D_t=(1/i)\partial_t$), so that $hD_s=(h/i)\partial_s$. A microlocally elliptic Fourier integral operator quantizing a canonical transformation which sends $p_+-E$ to the coordinate $\sigma$ gives, by Egorov's theorem \cite[Theorems~11.6 and~12.3]{ZworskiBook2012}, 
\begin{equation} 
\operatorname{Op}_h^w(p_+-E) \longmapsto h D_s+\Oc(h),
\label{ZCMM2026new:Fourier-integral-Egorov}
\end{equation} 
while the imaginary perturbation $i\gamma(h) b_E$ is transported to multiplication by $i\gamma(h) b_E(s,y,\sigma,\eta)$ modulo $\Oc(h\gamma(h))$.  
Choose canonical coordinates $(s,y;\sigma,\eta)$ near $\rho_*$ so that
\[
 p_+-E=\sigma,\qquad H_{p_+}=\partial_s
\]
on the central bicharacteristic.   Condition \eqref{ZCMM2026new:transverse-crossing} gives, after orienting $s$ in the outgoing direction,
\[
 b_E(s,0,0,0)=-\kappa s+\Oc(s^2),\qquad \kappa>0.
\]
Applying \eqref{ZCMM2026new:Fourier-integral-Egorov}, the positive block therefore has the local form
\begin{equation}
 \mbf{\Pc}_{+,h}
 =\big(hD_s-i\kappa\gamma(h)s\big) \mbf{I}_2
   +h \Rb_1(h)+\gamma(h)\Oc(s^2)+\gamma(h)^2\Rb_2(h),
 \label{ZCMM2026new:local-crossing-normal-form}
\end{equation}
with $\Rb_j(h)$ uniformly bounded microlocally.  Next, let 
\[
 \widetilde{h}:=\frac{h}{\gamma(h)}=\frac1{\log(1/h)}.
\]
The one-dimensional leading operator in \eqref{ZCMM2026new:local-crossing-normal-form} annihilates
\[
 v_h(s)=C_h\exp\!\left(-\frac{\kappa s^2}{2\widetilde h}\right).
\]
Thus $v_h$ is localised on the scale $|s|\asymp\widetilde h^{1/2}$.  On this scale
\[
 \gamma(h)s^2=\Oc(h),\qquad \gamma(h)^2=o(h).
\]
Localizing the remaining transverse variables at the ordinary $h^{1/2}$ scale changes the coefficient $b_E$ by $\Oc(h^{1/2})$, hence contributes $\gamma(h)\Oc(h^{1/2})=o(h)$.  A fixed cutoff in $s$ produces an exponentially small error in $1/\widetilde h$.   We write the normalized state in the canonical coordinates as 
\[
\wb_h(s,y)=\chi_s(s)v_h(s)a_h(y)\omega,\qquad
 a_h(y)=(\pi h)^{-1/2}e^{-|y|^2/(2h)},\qquad \omega\in\C^2,\ |\omega|=1,
\]
where $y\in\R^2$, $\chi_s=1$ near $s=0$, and the transverse coherent state $a_h$ is 
centered at $(y,\eta)=(0,0)$.  After normalization, $\| \wb_h\|=1$.  The preceding estimates give
\[
\| \mbf{\Pc}_{+,h} \wb_h \|\leq Ch.
\]
Let $\Fb_h$ be the microlocally elliptic Fourier integral operator which implements the canonical coordinate change, and let $\Eb_{+,h}$ be the $4\times2$ semiclassical pseudodifferential operator obtained by quantizing the positive-sheet frame $e_+$.  Choose microlocal 
cutoffs so that both maps are elliptic on the wavefront set of $\wb_h$, and define
\[
\widetilde{\ub}_h:=\Eb_{+,h} \Fb_h \wb_h.
\]
Microlocal ellipticity gives constants $c,C>0$, independent of $h$, such that 
$c\leq \|\widetilde{\ub}_h\| \leq C$.  The block reduction and Egorov relation proved above imply
\[
(\Dl_\theta-z) \Eb_{+,h} \Fb_h = \Eb_{+,h} \Fb_h \mbf{\Pc}_{+,h}+h \mbf{\Rc}_h 
\]
microlocally on the state, with $\mbf{\Rc}_h$ uniformly $L^2$-bounded 
(and the pieces outside the chosen microlocal neighbourhood $\Oc(h^\infty)$).  Hence
\[
\|(\Dl_\theta-z ) \widetilde{\ub}_h\| \leq C \| \mbf{\Pc}_{+,h} \wb_h \|+C h \| \wb_h\| \leq Ch.
\]
Finally we let $\ub_h=\widetilde{\ub}_h/ \|\widetilde{\ub}_h\|$.  Then $\| \ub_h\|=1$, 
its semiclassical wavefront set can be made arbitrarily close to $\rho_*$ by shrinking the 
microlocal cutoffs, and $\| (\Dl_\theta-z) \ub_h\|\leq Ch$, which is precisely \eqref{ZCMM2026new:full-quasimode}. The lower bound 
\eqref{ZCMM2026new:full-lower-bound} follows immediately.
\end{proof}

\begin{remark}
The sign change itself is forced by the exposing distortion; transversality is generic in the 
distortion strength.  Indeed, along a fixed outgoing orbit the function 
$a(s)=H_{p_+}(\gb\cdot\xi)(\Phi_s(\rho))$ joins $0$ to $a_\infty(E)$, and the crossing level is 
$M/\nu$.  If $M/\nu$ is a regular value of $a$, all crossings of this level are transverse.  Since $a$ starts below the level and tends to a value above it, at least one crossing is from below to above; at such a crossing $a'(s_*)>0$, hence $H_{p_+}b_E(\rho_*)=-\nu a'(s_*)<0$, which is exactly \eqref{ZCMM2026new:transverse-crossing}. Degenerate sign-changing crossings lead to analogous one-dimensional WKB localization, although we do 
not need that refinement here.
\end{remark}

Sandwiching by compactly supported cutoffs removes the transition-region state of Proposition~\ref{ZCMM2026new:PropFullObstruction}. This by itself does not establish a lower bound 
for the sandwiched resolvent.  Nevertheless, a second mechanism occurs entirely inside the undistorted 
region and shows why a pure $C/\gamma(h)$ bound should not be expected without additional geometric restrictions. In a positive-sheet flow box the continued outgoing equation has the model
\[
 hD_t+is,\qquad z=E-is,
\]
and its forward Green operator contains the factor $e^{s(t-r)/h}$.  A classical connection of length $T>0$ between two microlocal pieces of the interaction region therefore amplifies the continued resolvent by the factor $e^{sT/h}$.  At the bottom of the logarithmic strip, $s=M\gamma(h)$, this factor equals $h^{-MT}$.  The same behaviour is already visible for the free Dirac resolvent: if
\[
 |\xi_E|=\frac{\sqrt{E^2-m^2c^4}}{c},\qquad
 v_E=\frac{c^2|\xi_E|}{E},
\]
then the analytically continued outgoing phase over a spatial distance $r$ acquires the factor
\[
 \exp\!\left(\frac{sr}{h v_E}+\Oc\!\left(\frac{s^2r}{h}\right)\right),
\]
that is, $e^{sT/h}$ with $T=r/v_E$.  For the free operator the calculation can be made completely explicit.  Since
\[
 (\Dl_0(h)-z)(\Dl_0(h)+z) =\big(-c^2h^2\Delta+m^2c^4-z^2\big) \mbf{I}_4,
\]
one may write the free resolvent as a Dirac factor times a scalar Helmholtz resolvent. Introduce
\[
 k(z):=\frac1c\sqrt{z^2-m^2c^4},
\]
with the outgoing branch chosen by $k(E+i0)=|\xi_E|>0$.  In three dimensions the scalar outgoing kernel contains the phase factor
\[
\exp\!\left(\frac{i}{h}k(z)r\right),\qquad r=|x-y|,
\]
the remaining factor being a standard algebraic amplitude (and applying $\Dl_0+z$ changes only that amplitude).  Now, 
\[
k'(E)=\frac{E}{c\sqrt{E^2-m^2c^4}} =\frac1{v_E},\qquad v_E=\frac{c^2|\xi_E|}{E}.
\]
Therefore, for $z=E-is$ with $s>0$ small, analytic continuation of the outgoing branch gives
\[
k(E-is)=|\xi_E|-i\frac{s}{v_E}+\Oc(s^2).
\]
Substitution into the oscillatory factor yields
\[
e^{ik(E-is)r/h}  =e^{i|\xi_E|r/h}  \exp\!\left(\frac{sr}{hv_E}+\Oc\!\left(\frac{s^2r}{h}\right)\right).
\]
Thus analytic continuation below the real axis converts the usual outgoing oscillation into an exponentially growing factor.  Since the classical group velocity on the positive free sheet has magnitude $v_E$, the flight time over the distance $r$ is $T=r/v_E$, and the growth is exactly $e^{sT/h}$ to first order.  At $s=Mh\log(1/h)$ this becomes $h^{-MT}$.  This model calculation exhibits the geometric polynomial amplification that can occur for a sandwiched continued resolvent even when both cutoffs lie entirely in the undistorted region; it is used here as motivation for the upper bound below, not as a general lower-bound theorem.

This motivates the geometric estimate proved below.

\subsection{The positive-sheet connection time and uniform no-return}
\label{ZCMM2026new:SubsecNoReturnGeometry}
Let $\chi\in C_0^\infty(\R^3)$ be the spatial cutoff appearing in the localised propagation statement, and choose $\chi_1\in C_0^\infty(\R^3)$ such that
\begin{equation}
 \supp\chi\Subset\{\chi_1=1\}.
 \label{ZCMM2026new:buffered-cutoff}
\end{equation}
By the distortion-independence established in Section~\ref{ZCMM2025b:SecPrelim}, we may choose the admissible distortion field to vanish on a ball as large as needed below; this does not change the continued compactly supported resolvent or its poles. The enlarged onset radius is fixed independently of $h$.  
The FBI/escape-function argument used in Section~\ref{zcmm2025b-sec-masstype-resexp} can therefore be 
run with this admissible distortion as well (with different fixed constants), so the coarse $h$-tempered bound invoked in Lemma~\ref{ZCMM2026new:LemmaOutgoingTail} remains available for the same distortion used 
in the finite-time construction.

Define
\begin{equation}
 K_\chi:=\{(x,\xi):x\in\supp\chi_1,\ p_+(x,\xi)\in I^\sharp\}.
 \label{ZCMM2026new:Kchi}
\end{equation}
This set is compact.  We define the positive-sheet connection time by
\begin{equation}
 T_\chi:=\sup\{t\ge0:\ \exists\rho\in K_\chi\text{ with }\Phi_t(\rho)\in K_\chi\}. 
 \label{ZCMM2026new:Tchi}
\end{equation}
The notation $T_\chi$ suppresses its dependence on the fixed auxiliary choices $\chi_1$ and $I^\sharp$.  
These choices are kept fixed throughout Section~\ref{ZCMM2026new:SecFiniteTime}; thus $T_\chi$ is 
a well-defined finite geometric constant for the chosen localised energy window and buffered spatial region. Henceforth we refer to 
Figure~\ref{ZCMM2026new:Sect6-Fig-1} for the geometry of our problem and our constructions. 

\begin{figure}[htbp]
\centering
\resizebox{0.9\textwidth}{!}{%
\begin{tikzpicture}[>=Stealth, line cap=round, line join=round, font=\small]

\begin{scope}[xshift=0cm,yshift=0cm]
  \node[font=\bfseries] at (0,4.55) {(a) Finite-time geometry on the positive sheet};
  \node[font=\scriptsize,regiongray] at (0,4.17) {configuration-space projection of the phase-space picture};

  \fill[distorange!10] (0,0) circle (3.75);
  \fill[white] (0,0) circle (3.05);
  \draw[distorange,very thick] (0,0) circle (3.05);
  \node[distorange,align=center] at (0,-3.42)
    {$|x|=R_{\mathrm{esc}}$ }; 

  \fill[cutgreen!8] (0,0) circle (2.02);
  \draw[cutgreen!70!black,thick,dashed] (0,0) circle (2.02);
  \node[cutgreen!70!black] at (-2.00,1.95) {$\operatorname{supp}\chi_1$};

  \fill[cutgreen!18] (0,0) circle (1.30);
  \draw[cutgreen!80!black,thick] (0,0) circle (1.30);
  \node[cutgreen!80!black] at (-1.08,-0.92) {$\operatorname{supp}\chi$};

  \node[align=center] at (0,0.18)
    {$K_\chi$\\[-1mm]\scriptsize energy-localized phase-space set};

  \draw[decorate,decoration={brace,amplitude=5pt},regiongray]
    (-3.02,3.35) -- (3.02,3.35)
    node[midway,above=5pt,align=center,regiongray]
    {distortion chosen zero on the whole finite-time propagation region};

  \coordinate (a0) at (-0.95,0.45);
  \coordinate (a1) at (0.95,-0.35);
  \fill[trajblue] (a0) circle (2pt);
  \fill[trajblue] (a1) circle (2pt);
  \draw[trajblue,very thick,->]
    (a0) .. controls (-0.15,1.25) and (0.65,0.70) .. (a1);
  \node[trajblue,align=center] at (0.18,1.30)
    {possible connection\\[-1mm]\scriptsize $0<t\le T_\chi$};

  \coordinate (b0) at (-0.45,-0.55);
  \coordinate (b1) at (1.15,0.58);
  \coordinate (b2) at (2.65,1.50);
  \coordinate (b3) at (4.25,2.22);
  \fill[accentred] (b0) circle (2pt);
  \draw[accentred,very thick,->]
    (b0) .. controls (0.25,-0.30) and (0.55,0.90) .. (b1)
         .. controls (1.75,1.08) and (2.15,1.35) .. (b2)
         .. controls (3.15,1.72) and (3.62,1.98) .. (b3);
  \fill[accentred] (b2) circle (2.2pt);
  \node[accentred,anchor=west,align=left] at (2.78,1.28)
    {first outward crossing\\[-1mm]\scriptsize $H_{p_+}(x\!\cdot\!\xi)>0$ is one-way};
  \node[accentred,anchor=west] at (3.63,2.52) {$\Phi_t(\rho)$};

  \draw[accentred,thick,->] (3.22,0.75) arc[start angle=10,end angle=-58,radius=1];
  \node[accentred,align=center] at (3.62,-0.62)
    {\scriptsize no inward return\\[-1mm] after escape};

  \draw[trajblue,densely dotted,thick,rounded corners=4pt]
    (2.00,0.72) rectangle (2.72,1.23);
  \node[trajblue] at (2.38,0.51) {$\mathcal K_T$};

  \node[draw=regiongray!65,rounded corners=2pt,fill=white,inner sep=5pt,align=center]
    at (0,-4.32)
    {$\displaystyle T_\chi=\sup\bigl\{t\ge0:\exists\rho\in K_\chi,\ \Phi_t(\rho)\in K_\chi\bigr\}<\infty$,\\
     $\Phi_t(\mathcal K_T)\cap K_\chi=\varnothing\quad(t\ge0).$};
\end{scope}

\draw[regiongray!35] (5.05,-4.75) -- (5.05,4.72);

\begin{scope}[xshift=10.25cm,yshift=0cm]
  \node[font=\bfseries] at (0,4.55) {(b) Logarithmic strip and continued resolvent};

  \draw[->,thick] (-3.8,0) -- (3.95,0) node[right] {$\operatorname{Re}z$};
  \draw[->,thick] (0,-3.60) -- (0,3.55) node[above] {$\operatorname{Im}z$};

  \coordinate (L) at (-2.45,0);
  \coordinate (R) at (2.45,0);
  \coordinate (LB) at (-2.45,-2.35);
  \coordinate (RB) at (2.45,-2.35);
  \fill[stripblue] (L) rectangle (RB);
  \draw[trajblue,thick] (L) -- (R);
  \draw[red,thick] (LB) -- (RB);
  \draw[trajblue,dashed] (L) -- (LB);
  \draw[trajblue,dashed] (R) -- (RB);
  \node[green,above] at (0,0.10) {$I_+$};
  \node[red,below] at (0,-2.45) {$\operatorname{Im}z=-M\gamma(h)$, \scriptsize $0<M<C_{\mathrm{rf}}$};
  \node[trajblue,align=center] at (-1,-1.12)
    {resonance-free \\
    logarithmic strip\\[-1mm]};

  \coordinate (zpt) at (0.95,-1.28);
  \fill[accentred] (zpt) circle (2.4pt);
  \node[accentred,anchor=west] at (1.08,-1.22) {$z=E-is$};
  \draw[accentred,dashed] (0.95,0) -- (zpt);
  \node[accentred,anchor=east] at (0.88,-0.62) {$s$};

  \node[accentred,align=center] at (0,2.88)
    {analytic continuation of a finite-time piece\\[-1mm]
     $\displaystyle e^{sT/h}$ growth over a classical connection of time $T$};

  \draw[decorate,decoration={brace,mirror,amplitude=5pt},white]
    (-2.45,2) -- (2.45,2)
    node[midway,below=6pt,align=center,trajblue]
    {$\displaystyle s=Mh\log(1/h)\quad\Longrightarrow\quad e^{sT/h}=h^{-MT}$};

  \draw[
  decorate,
  decoration={brace,  amplitude=5pt},
  draw=green
  ]
  (-2.45,0) -- (2.45,0);
  \node[draw=regiongray!65,rounded corners=2pt,fill=white,inner sep=5pt,align=center,text width=7.0cm]
    at (0,-4.20)
    {$\displaystyle
      \bigl\|\chi R^{\mathrm{cont}}_+(E-is;h)\,\psi(D(h))\chi\bigr\|
      \ \lesssim\ 
      \frac{e^{s(T_\chi+\varepsilon)/h}-1}{s}. $};
\end{scope}

\draw[->,very thick,regiongray!70]
  (4.33,-1.95) -- (5.72,-1.95)
  node[midway,above=2pt,align=center,font=\scriptsize]
  {finite-time\\propagation};
\end{tikzpicture}
}
 \caption{Schematic geometry underlying the finite-time resolvent construction.  Panel (a) shows the configuration-space projection of a positive-sheet bicharacteristic.  The trajectory may make finite returns through the interaction region, but after its first outward crossing of $|x|=R_{\rm esc}$, the positivity $H_{p_+}(x\cdot\xi)>0$ makes the crossing one-way and excludes subsequent inward return.  The distortion may be chosen to start beyond the fixed region swept out during the finite-time propagation.  Panel (b) shows the corresponding logarithmic strip in the spectral variable.  Continuing a finite-time propagator segment of classical duration $T$ from $\im z>0$ to $z=E-is$ produces the factor $e^{sT/h}$; at $s=Mh\log(1/h)$ this equals $h^{-MT}$ and leads to the geometric resolvent bound displayed in the figure.  The picture is schematic and not to scale.}
\label{ZCMM2026new:Sect6-Fig-1}
\end{figure}
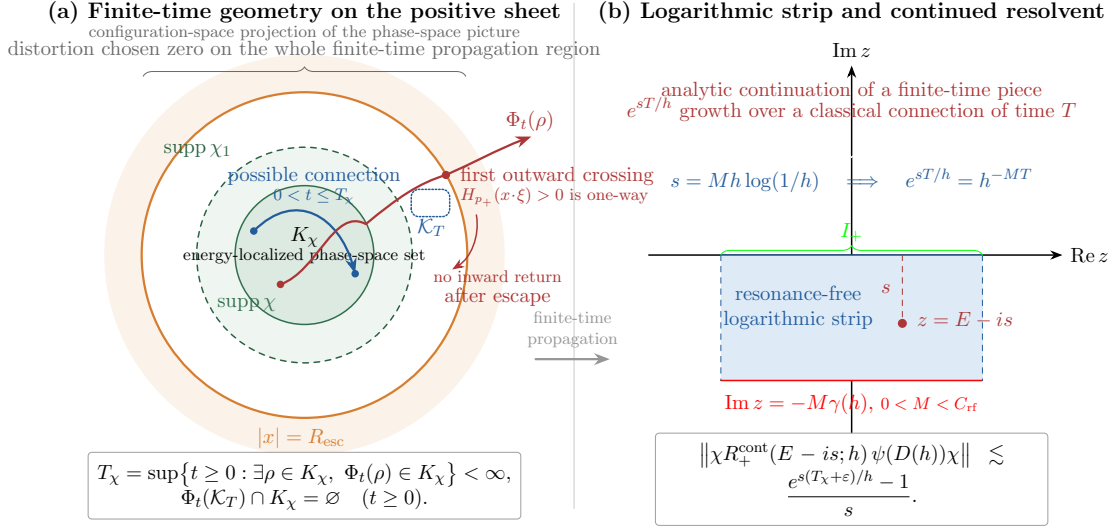


As usual, the semiclassical wavefront set is denoted $\operatorname{WF}_h(\cdot)$ and 
$\operatorname{WF}'_h(\cdot)$ denotes the wavefront relation of an operator 
\cite[Chapter~XXV]{HormanderIV1985}; for the semiclassical wavefront-set formalism see also \cite[Section~8.4]{ZworskiBook2012}. For a semiclassical pseudodifferential operator $A$ we write
$\operatorname{MS}_h(A):=\{\rho:(\rho,\rho)\in\operatorname{WF}'_h(A)\}$ for its single-phase-space microsupport.  This is the diagonal realization of the usual pseudodifferential microsupport; see also \cite[Chapter~XVIII]{HormanderIII1985}.  For a general Fourier integral operator we instead use the left and right projections $\pi_L\operatorname{WF}'_h(A)$ and $\pi_R\operatorname{WF}'_h(A)$; this distinction is needed below for the propagated operator $\Ub(T)\psi(\Dl)\chi$.

\begin{lemma}[Uniform finite connection time]
\label{ZCMM2026new:LemmaTchi}
Under the nontrapping hypothesis on $I^\sharp$, one has $T_\chi<\infty$.  Moreover, if $T>T_\chi$, then there are neighbourhoods $\Kc_0$ of
\[
\operatorname{MS}_h(\psi(\Dl(h))\chi)\cap p_+^{-1}(I^\sharp)
\]
and $\Kc_T$ of $\Phi_T(\Kc_0)$ such that
\begin{equation}
\Phi_t(\Kc_T)\cap K_\chi=\emptyset, \qquad t\ge0,
 \label{ZCMM2026new:robust-no-return}
\end{equation}
for every fixed $\psi\in C_0^\infty(I_+)$.  The neighbourhoods can be chosen uniformly for the whole energy interval $I^\sharp$.
\end{lemma}

\begin{proof}
We first establish a one-way property in the exterior.  Formula \eqref{ZCMM2026new:Hpxxi} and the long-range assumptions imply that there are $R_{\rm esc}>0$ and $c_{\rm esc}>0$ such that
\begin{equation}
 H_{p_+}(x\cdot\xi)\ge c_{\rm esc}>0,
 \qquad |x|\ge R_{\rm esc},\quad p_+(x,\xi)\in I^\sharp.
 \label{ZCMM2026new:exterior-convexity}
\end{equation}
Indeed, on $I^\sharp$ the limiting value of $c^2|\xi|^2$ at infinity is bounded away from zero uniformly, while $x\cdot\nabla V(x)=o(1)$ for $|x| \to \infty$.  Enlarge $R_{\rm esc}$ so that $\supp\chi_1\Subset B(0,R_{\rm esc})$. Along a positive-sheet trajectory,
\[
 \frac{d}{dt}|x(t)|^2  =2\frac{c^2}{p_+(x(t),\xi(t))}\,x(t)\cdot\xi(t).
\]
Once a trajectory crosses $|x|=R_{\rm esc}$ outward, one has $x\cdot\xi\ge0$ at the crossing and \eqref{ZCMM2026new:exterior-convexity} makes $x\cdot\xi$ strictly increasing thereafter.  Hence the trajectory cannot cross the same sphere inward again; see Figure~\ref{ZCMM2026new:Sect6-Fig-1}. 

Suppose now that $T_\chi=\infty$.  Then there are $\rho_j\in K_\chi$ and $t_j\to\infty$ with 
$\Phi_{t_j}(\rho_j)\in K_\chi$.  By the one-way property, the full orbit segment 
$\{\Phi_t(\rho_j):0\le t\le t_j\}$ remains in $|x|\le R_{\rm esc}$.  On the fixed energy interval 
this is a compact subset of phase space.  Set $\widetilde\rho_j=\Phi_{t_j/2}(\rho_j)$ and 
pass to a convergent subsequence.  For every fixed $S>0$, the orbit through 
$\widetilde{\rho}_j$ remains in the same compact set for $|t|\le S$ when $j$ is large.  
Passing to the limit produces a complete trapped trajectory, contradicting nontrapping.  
Thus $T_\chi<\infty$.

Finally fix $T>T_\chi$.  No point of the compact set 
$\operatorname{MS}_h(\psi(\Dl)\chi)\cap p_+^{-1}(I^\sharp)$ can return to $K_\chi$ at any time $t\ge T$.  Compactness on the finite part of the flow, followed by the one-way exterior property after escape, gives a positive phase-space separation.  Enlarging the initial and terminal sets slightly therefore gives \eqref{ZCMM2026new:robust-no-return}.  Since all estimates above are uniform on the fixed compact interval $I^\sharp$, so are the neighbourhoods.
\end{proof}

\subsection{The outgoing no-return tail}
\label{ZCMM2026new:SubsecOutgoingTail}
The next lemma is the microlocal input which converts the classical no-return statement into a resolvent estimate.  It is a standard outgoing propagation statement, but we include the argument because the spectral parameter lies at logarithmic depth and because the positive 
Dirac sheet has rank two.

\begin{lemma}[Uniform outgoing no-return tail]
\label{ZCMM2026new:LemmaOutgoingTail}
Let $T>T_\chi$, and let $\Bb_T \in  \Psi_h^0(\R^3;\C^4)$ be compactly supported in the undistorted region, microlocalised to the positive energy interval $I^\sharp$, with
\begin{equation}
\operatorname{MS}_h(\Bb_T)\Subset \Kc_T,
 \label{ZCMM2026new:BT-WF}
\end{equation}
where $\Kc_T$ is as in Lemma~\ref{ZCMM2026new:LemmaTchi}.  Then, for every 
$N>0$, there are $C_N>0$ and $h_N>0$ such that
\begin{equation}
 \|\chi (z-\Dl_\theta(h))^{-1} \Bb_T\|_{\Bc(L^2)}\leq C_Nh^N,
 \label{ZCMM2026new:no-return-tail}
\end{equation}
uniformly for
\begin{equation}
 z=E-is,\qquad E\in I_+,\qquad 0<s\le M\gamma(h),\qquad 0<h<h_N,
 \label{ZCMM2026new:tail-strip}
\end{equation}
provided $0<M<C_{\rm rf}$.  The same conclusion holds with fixed positive-sheet microlocal projectors inserted on either side.
\end{lemma}

\begin{proof}
The inhomogeneous FBI estimate in the proof of Lemma~\ref{ZCMM2025a: lemma-bottom-bound}, made explicit in Appendix~\ref{ZCMM2025b:appendix-inhomogeneous-FBI}, actually gives an $h$-tempered bound for the full distorted inverse before any compact cutoffs are inserted, in every fixed smaller resonance-free logarithmic strip. Thus, for some fixed $N_*>0$,
\begin{equation}
 \|(z-\Dl_\theta(h))^{-1}\|_{\Bc(L^2)}\leq Ch^{-N_*}
 \label{ZCMM2026new:tempered-distorted}
\end{equation}
on the strip under consideration.  Indeed, after absorbing the last term in \eqref{ZCMM2025b:inhomogeneous-FBI-estimate}, one has
\[
X\le C\eta^{-1}\|e^{-t_{\rm FBI}H/h}\Tb \fb\|  \leq C\eta^{-1}h^{-K_0}\| \fb\|.
\]
The inverse weighted FBI comparison gives $\|\ub\|\leq Ch^{-K_0}X$, hence
\[
\|\ub\|\leq C h^{-2K_0}\eta^{-1}\|\fb\|  \leq C h^{-(2K_0+1)}\|\fb\|,
\]
since $\eta=C_1h\log(1/h)$.  The resonance-free theorem gives injectivity and the distorted operator is Fredholm of index zero, so the inverse exists.  Thus one may take any fixed $N_*>2K_0+1$.  The statement of Lemma~\ref{ZCMM2025a: lemma-bottom-bound} records only the cut-off consequence because that was all Section~5 needed, but its proof provides the full 
$h$-tempered distorted inverse used here. We use the estimate \eqref{ZCMM2026new:tempered-distorted} only for temperedness, not for the size in \eqref{ZCMM2026new:no-return-tail}.

Microlocally in the undistorted region and on $I^\sharp$, the two Dirac eigenvalues are separated by a fixed gap.  The negative sheet is uniformly elliptic at positive energies.  On the positive rank-two eigenbundle, standard block diagonalization reduces the principal symbol to 
$(p_+-E)\mbf{I}_2$, with $h \Psi_h^0$ connection and polarization terms.  The outgoing 
distorted inverse therefore has the usual positive-sheet semiclassical wavefront relation
\begin{equation}
 \operatorname{WF}'_h\big((z-\Dl_\theta(h))^{-1}\big)
 \subset \Delta\cup\Lambda_+,
 \qquad
 \Lambda_+:=\{(\Phi_t(\rho),\rho):t\ge0,\ p_+(\rho)=E\},
 \label{ZCMM2026new:outgoing-relation}
\end{equation}
inside every fixed undistorted phase-space region, modulo elliptic pieces.  This is the standard real-principal-type propagation relation with the outgoing condition supplied by the complex distortion; for the underlying 
real-principal-type propagation theorem see \cite[Theorem~12.5]{ZworskiBook2012} and 
\cite[Appendix~E.5]{DyatlovZworski2019}; the outgoing orientation is the one selected by the distorted/scattering resolvent. The rank-two multiplicity changes the transport amplitudes but not the canonical relation.

By \eqref{ZCMM2026new:robust-no-return} and \eqref{ZCMM2026new:BT-WF}, no forward positive-sheet bicharacteristic issued from $\operatorname{MS}_h(\Bb_T)$ meets the microlocal support of $\chi$.  
The two supports are also disjoint at time zero.  Hence
\[
\big(\operatorname{MS}_h(M_\chi)\times\operatorname{MS}_h(\Bb_T)\big)
 \cap(\Delta\cup\Lambda_+)=\emptyset, 
\]
where $M_\chi$ denotes multiplication by $\chi$; both $M_\chi$ and $\Bb_T$ are pseudodifferential operators, 
so the single-phase-space notation $\operatorname{MS}_h$ introduced above applies to them.
Moreover, the off-relation part of the outgoing parametrix is $\Oc(h^K)$ for arbitrary $K$. It remains only to justify uniformity when $s/h$ grows like $\log(1/h)$.  The term $-is$ does not alter the real principal-type canonical relation.  Over a fixed classical propagation time it changes transport amplitudes by at most
\[
e^{Cs/h}\le e^{CM\gamma(h)/h}=h^{-CM},
\]
a fixed polynomial factor.  The coarse estimate \eqref{ZCMM2026new:tempered-distorted} contributes another fixed polynomial loss.  Since the microlocal construction can be carried out to arbitrarily high order, choose $K$ larger than all these fixed losses plus the prescribed $N$.  The resulting remainder is $\Oc(h^N)$ uniformly on \eqref{ZCMM2026new:tail-strip}. 
This proves \eqref{ZCMM2026new:no-return-tail}.
\end{proof}

\begin{remark}
The use of Lemma~\ref{ZCMM2025a: lemma-bottom-bound} in the preceding proof is not circular.  Its polynomial estimate is used only to guarantee that the exact outgoing solution is $h$-tempered.  The sharper size below comes from the finite-time propagator; no FBI norm comparison enters that estimate.
\end{remark}

\subsection{A geometric continued cut-off resolvent estimate}
\label{ZCMM2026new:SubsecGeometricResolvent}
We now derive the main estimate of the section.  Let $\psi\in C_0^\infty(I_+)$ be fixed.  In the upper half-plane define
\[
 \Rb(z;h)=(z-\Dl(h))^{-1},\qquad \Ub(t)=e^{-it\Dl(h)/h}.
\]
For a fixed $T>0$ set
\begin{equation}
 \Gb_T(z):=-\frac{i}{h}\int_0^T \Ub(t)e^{itz/h}\,dt.
 \label{ZCMM2026new:GT}
\end{equation}
The operator $\Gb_T(z)$ is entire in $z$ and, for $\im z>0$,
\begin{equation}
 \Rb(z;h)=\Gb_T(z)+e^{iTz/h}\Rb(z;h)\Ub(T).
 \label{ZCMM2026new:finite-time-splitting}
\end{equation}
Indeed, this is just the Laplace representation of the physical outgoing resolvent split at time $T$.
Introduce
\begin{equation}
 \Rb^{\rm cont}_{\chi,\psi,+}(z;h) :=\chi(z-\Dl_\theta(h))^{-1}\psi(\Dl(h))\chi
 \label{ZCMM2026new:spectrally-localised-cont-res}
\end{equation}
the continuation of $\chi\Rb(z;h)\psi(\Dl(h))\chi$ from the upper half-plane, with the distortion chosen to be the identity on all fixed spatial supports occurring in the finite-time construction.  
There is a small technical point in calling \eqref{ZCMM2026new:spectrally-localised-cont-res} the continuation 
of $\chi\Rb(z;h)\psi(\Dl(h))\chi$: the factor $\psi(\Dl(h))\chi$ is pseudolocal but is not literally compactly supported in $x$, whereas the standard distortion identity is exact for spatial cutoffs supported in the undistorted region.  To make the statement precise, fix $\chi_2\in C_0^\infty(\R^3)$, supported in the undistorted region and equal to one on a neighbourhood of $\supp\chi$.  Then semiclassical pseudolocality gives
\[
 (1-\chi_2)\psi(\Dl(h))\chi=\Oc(h^\infty)_{L^2\to L^2}.
\]
The full $h$-tempered distorted-resolvent estimate \eqref{ZCMM2026new:tempered-distorted} therefore implies, uniformly in the logarithmic strip,
\[
 \chi(z-\Dl_\theta(h))^{-1}\psi(\Dl(h))\chi
 =\chi(z-\Dl_\theta(h))^{-1}\chi_2\psi(\Dl(h))\chi+\Oc(h^\infty).
\]
The term with $\chi_2$ is the genuine two-sided compactly supported analytic continuation of the corresponding physical resolvent expression.  Thus \eqref{ZCMM2026new:spectrally-localised-cont-res} is unambiguous modulo $\Oc(h^\infty)$, which is all that is needed for the estimates and the subsequent Cauchy--Green argument.  The same $\chi_2$ is used in the proof below. The admissible distortion is chosen to be the identity on all fixed spatial supports occurring in the finite-time construction; the resulting compactly sandwiched continuation is independent of that admissible choice by uniqueness of meromorphic continuation.

\begin{theorem}[Geometric cut-off resolvent bound]
\label{ZCMM2026new:TheoremGeometricResolvent}
Assume that $p_+$ is nontrapping on $I_+^{\rm out}$ and let $0<M<C_{\rm rf}$.  Let $\chi,\chi_1$ satisfy \eqref{ZCMM2026new:buffered-cutoff}, let $T_\chi$ be defined by \eqref{ZCMM2026new:Tchi}, and let $\psi\in C_0^\infty(I_+)$.  Then for every $\varepsilon>0$ there are $C_\varepsilon>0$ and $h_\varepsilon>0$ such that, for
\[
 E\in I_+,  \qquad 0<s\le M\gamma(h), \qquad 0<h<h_\varepsilon,
\]
one has
\begin{equation}
 \| \Rb^{\rm cont}_{\chi,\psi,+}(E-is;h)\|_{\Bc(L^2)}  \leq C_\varepsilon
 \frac{\exp\!\big(s(T_\chi+\varepsilon)/h\big)-1}{s}.
 \label{ZCMM2026new:geometric-resolvent-bound}
\end{equation}
The quotient on the right is interpreted continuously at $s=0$, with value 
$(T_\chi+\varepsilon)/h$.  In particular,
\begin{equation}
 \|\Rb^{\rm cont}_{\chi,\psi,+}(E-iM\gamma(h);h)\|_{\Bc(L^{2})} \leq \frac{C_\varepsilon}{M\gamma(h)} h^{-M(T_\chi+\varepsilon)}.
 \label{ZCMM2026new:bottom-geometric-resolvent}
\end{equation}
\end{theorem}

\begin{proof}
Fix $ T=T_\chi+\varepsilon$.  By semiclassical pseudolocality, if $\chi_2\in C_0^\infty(\R^3)$ equals one near $\supp\chi$, then
\begin{equation}
 (1-\chi_2)\psi(\Dl(h))\chi=\Oc(h^\infty)_{L^2\to L^2}.
 \label{ZCMM2026new:functional-locality}
\end{equation}
Finite propagation speed for the Dirac evolution, or equivalently fixed-time Egorov theory \cite[Theorem 11.1]{ZworskiBook2012}, then shows that $\Ub(T)\psi(\Dl)\chi$ is microlocally contained in a fixed compact phase-space set.  Indeed, semiclassical functional calculus gives $\psi(\Dl(h))\in\Psi_h^0$ microlocally on the separated positive-energy window, with microsupport contained in the set where one of the eigenvalues of the principal symbol lies in $\supp\psi$.  Since 
$\supp\psi\Subset I_+\Subset(mc^2,\infty)$, the negative eigenvalue is uniformly separated 
from $\supp\psi$, so only the positive sheet occurs.  Pseudolocality and the fact that 
$\chi_2=1$ near $\supp\chi$ give \eqref{ZCMM2026new:functional-locality}; 
consequently, modulo $\Oc(h^\infty)$,
\[
\operatorname{MS}_h\big(\psi(\Dl)\chi\big) \Subset\{(x,\xi):x\in\supp\chi_2,\ p_+(x,\xi)\in\supp\psi\}.
\]
The set on the right is compact: $x$ ranges in a compact set, and the energy condition 
$p_+\in\supp\psi$ gives a uniform bound on $|\xi|$.  For fixed $T$, the propagator 
$\Ub(T)=e^{-iT\Dl/h}$ is microlocally a Fourier integral operator associated, on this separated sheet, with the Hamilton flow $\Phi_T$; equivalently, fixed-time Egorov theory gives
\[
\operatorname{MS}_h\big(\Ub(T)\psi(\Dl)\chi\big) \subset \Phi_T\!\left(\operatorname{MS}_h(\psi(\Dl)\chi)\right).
\]
Note that, in the present matrix-valued setting,  one first performs the positive/negative sheet block reduction; 
on the positive rank-two block the principal symbol is $p_+\mbf{I}_2$, so the scalar Hamilton flow in 
Theorem~11.1 \cite{ZworskiBook2012} governs the microsupport, while the lower-order $2\times2$ 
terms only change the transport amplitudes. More precisely, because $\Ub(T)\psi(\Dl)\chi$ is a Fourier integral operator rather than a pseudodifferential operator, one should use the left projection of its wavefront relation:
\[
\pi_L\operatorname{WF}'_h\!\big(\Ub(T)\psi(\Dl)\chi\big) \subset \Phi_T\!\left(\operatorname{MS}_h(\psi(\Dl)\chi)\right).
\]
This is the fixed-time Egorov statement in the form needed here; see \cite[Theorem~11.1]{ZworskiBook2012}.
The image of a compact set under the fixed-time smooth Hamilton flow is compact.  
At the level of configuration space the same fact is consistent with finite propagation 
speed for the Dirac equation: an exactly compactly supported piece can move by 
at most $c|T|$ in time $T$.  Thus the propagated operator is microlocally contained in a fixed compact phase-space set, independent of $h$.  Choose a compact neighbourhood 
$\Kc_T$ of that propagated set which still satisfies the robust no-return property, and 
choose a matrix symbol $\mbf{b}_T\in \Ssf(1;M_4(\C))$ with $\mbf{b}_T=\mbf{I}_4$ on a 
neighbourhood of the propagated microsupport and $\supp \mbf{b}_T\Subset \Kc_T$.  With $\Bb_T=\operatorname{Op}_h^w(\bb_T)$, pseudolocality gives
\begin{equation}
\Bb_T \Ub(T)\psi(\Dl(h))\chi =\Ub(T)\psi(\Dl(h))\chi+\Oc(h^\infty), \qquad \operatorname{MS}_h(\Bb_T)\Subset\Kc_T.
 \label{ZCMM2026new:BT-Egorov}
\end{equation} 
Since the spatial projection of $\Kc_T$ is compact, the onset radius of the admissible distortion can 
then be chosen beyond it, so that $\Bb_T$ lies entirely where $\Dl_\theta=\Dl$. Strictly speaking, a pseudodifferential operator with compact $x$-support is only microlocally spatially supported because 
its Schwartz kernel has rapidly decaying off-diagonal tails.  Inserting an additional spatial cutoff equal 
to one on the $x$-projection of $\operatorname{MS}_h(\Bb_T)$ changes $\Bb_T$ by $\Oc(h^\infty)$; 
after the polynomial temperedness estimate \eqref{ZCMM2026new:tempered-distorted}, this remains 
$\Oc(h^\infty)$.  Thus the standard compactly supported distortion identity applies to the $\Bb_T$ 
term modulo the harmless remainder already present below.

Multiplying \eqref{ZCMM2026new:finite-time-splitting} by the cutoffs and using \eqref{ZCMM2026new:BT-Egorov}, then continuing from the upper half-plane to 
$z=E-is$, gives
\begin{align}
 \Rb^{\rm cont}_{\chi,\psi,+}(z;h)
 ={}&\chi\Gb_T(z)\psi(\Dl(h))\chi \nonumber\\
 &+e^{iTz/h}\chi(z-\Dl_\theta(h))^{-1}\Bb_T \Ub(T)\psi(\Dl(h))\chi
 +\Oc(h^\infty).
 \label{ZCMM2026new:continued-finite-time-splitting}
\end{align}
By Lemma~\ref{ZCMM2026new:LemmaOutgoingTail}, the second line is $\Oc(h^L)$ for every prescribed $L$, uniformly throughout the strip.  Indeed, the scalar factor satisfies
\begin{equation}
 |e^{iT(E-is)/h}|=e^{sT/h}
 \le h^{-MT},
 \label{ZCMM2026new:finite-time-polynomial-factor}
\end{equation}
and the no-return estimate is available to arbitrary order.

It remains to estimate the first line.  By unitarity of $\Ub(t)$ on the physical space,
\begin{align}
 \|\chi\Gb_T(E-is)\psi(\Dl(h))\chi\|
 &\le \frac{C}{h}\int_0^T e^{st/h}\,dt =C\frac{e^{sT/h}-1}{s}.
 \label{ZCMM2026new:GT-bound}
\end{align}
The $\Oc(h^\infty)$ remainder is absorbed into the right-hand side for small $h$, which proves \eqref{ZCMM2026new:geometric-resolvent-bound}.  Setting $s=M\gamma(h)$ and using $e^{M\gamma(h)T/h}=h^{-MT}$ gives \eqref{ZCMM2026new:bottom-geometric-resolvent}.
\end{proof}

\begin{remark}[Geometric meaning and near-sharpness]
\label{ZCMM2026new:RemarkSharpness}
The factor $e^{sT_\chi/h}$ is the price of analytically continuing the outgoing finite-time propagator below the real axis.  The flow-box and free Dirac calculations above show that this dependence is geometrically natural, although we do not prove a matching lower bound for every choice of sandwiched cutoffs. If two microlocal pieces of 
$\{\chi=1\}$ are joined by a clean positive-energy bicharacteristic of travel time $T_0$, the flow-box model gives a contribution of size $s^{-1}e^{sT_0/h}$, up to the standard semiclassical amplitude.  Thus \eqref{ZCMM2026new:geometric-resolvent-bound} has the expected exponential dependence on classical flight time.  At $s\to 0^{+}$ it reduces to the familiar nontrapping $h^{-1}$ scale, while at logarithmic depth it becomes a polynomial in $h^{-1}$ with an explicit geometric exponent.
\end{remark}

\subsection{Consequences for the Cauchy--Green deformation}
\label{ZCMM2026new:SubsecCGConsequence}
We finally show how Theorem~\ref{ZCMM2026new:TheoremGeometricResolvent} sharpens the local-energy decay statement of Section~5.  The physical lower-half-plane branch satisfies
\begin{equation}
 \|\chi\Rb(E-is;h)\psi(\Dl(h))\chi\|_{\Bc(L^{2})} \leq Cs^{-1},
 \qquad s>0,
 \label{ZCMM2026new:physical-spectral-bound}
\end{equation}
by self-adjointness.  Hence Theorem~\ref{ZCMM2026new:TheoremGeometricResolvent} implies, for every $\varepsilon>0$,
\begin{equation}
 \|\Rb^{\rm phys}_{\chi,\psi,-}(E-is;h)
      -\Rb^{\rm cont}_{\chi,\psi,+}(E-is;h)\|_{\Bc(L^{2})} \leq C_\varepsilon\frac{e^{s(T_\chi+\varepsilon)/h}}{s}.
 \label{ZCMM2026new:geometric-jump}
\end{equation}
The notation means that the fixed spectral cutoff $\psi(\Dl)$ is inserted on the right. 
More precisely, on the real axis one has, as an operator-valued spectral distribution,
\[
 \chi\big(\Rb_-(\lambda)-\Rb_+(\lambda)\big)\psi(\Dl)\chi =\psi(\lambda)\,\chi\big(\Rb_-(\lambda)-\Rb_+(\lambda)\big)\chi.
\]
Thus the operator cutoff already supplies the scalar spectral localization in Stone's formula.  In the 
Cauchy--Green deformation below we therefore use a separate auxiliary scalar cutoff 
$\varphi\in C_0^\infty(I_+)$ with $\varphi=1$ on a neighbourhood of $\supp\psi$, rather than 
multiplying by a second copy of $\psi(\lambda)$.

\begin{corollary}[Geometric time threshold for local-energy decay]
\label{ZCMM2026new:CorGeometricDecay}
Under the hypotheses of Theorem~\ref{ZCMM2026new:TheoremGeometricResolvent}, let $N\in\N$ and choose $0<M<C_{\rm rf}$.  If
\begin{equation}
 T_N>T_\chi+\frac{N+1}{M},
 \label{ZCMM2026new:geometric-time-threshold}
\end{equation}
then, without any further enlargement of $T_N$, 
\begin{equation}
 \|\chi e^{-it\Dl(h)/h}\psi(\Dl(h))\chi\|_{\Bc(L^{2})}  \leq C_Nh^N, \qquad t\ge T_N,
 \label{ZCMM2026new:geometric-local-decay}
\end{equation}
for all sufficiently small $h$.  By the $z\mapsto-z$ symmetry of the mass-type problem, the reflected negative-energy statement holds for $t\le-T_N$.
\end{corollary}

\begin{proof}
Use the same Cauchy--Green deformation as in \eqref{ZCMM2025b:res_exp_resfre_pos_cauchyF}, 
but write the boundary jump with the fixed operator cutoff $\psi(\Dl)$ on the right.  Choose the auxiliary $\varphi\in C_0^\infty(I_+)$ described above and an almost-analytic extension $\widetilde\varphi$ with 
$|\bar\partial\widetilde\varphi(z)|\le C_L|\im z|^L$ for every $L$.  Multiplication by $\varphi(\lambda)$ does not change the real-axis jump because that jump, after insertion of $\psi(\Dl)$, is spectrally supported in $\supp\psi$.  The compact support of $\widetilde\varphi$ also removes the vertical contour terms.  This auxiliary cutoff is essential for bookkeeping: using an almost-analytic extension of $\psi$ itself in addition to the operator $\psi(\Dl)$ would give $\psi(\lambda)^2$ on the real axis.  Put $s=-\im z$.

On the bottom contour $s=M\gamma(h)$, \eqref{ZCMM2026new:geometric-jump} and
$ |e^{-itz/h}|=e^{-ts/h}=h^{Mt}$ give, for every fixed $\varepsilon>0$,
\begin{equation}
 \|\Bb(h)\| \leq C_\varepsilon h^{Mt} \frac{h^{-M(T_\chi+\varepsilon)}}{\gamma(h)} 
 = C_\varepsilon \gamma(h)^{-1}  h^{M(t-T_\chi-\varepsilon)}.
 \label{ZCMM2026new:geometric-bottom-term}
\end{equation}
Since $\gamma(h)^{-1}=h^{-1}/\log(1/h)$, this is $\Oc(h^N)$ whenever
\[
 t>T_\chi+\varepsilon+\frac{N+1}{M}.
\]
For the area term, use $|\bar\partial\widetilde\varphi(E-is)| \leq C_L s^L$ together with \eqref{ZCMM2026new:geometric-jump}.  If $t>T_\chi+\varepsilon$, then
\begin{eqnarray}
 \|\Ab(h)\| & \leq & C_L\int_0^{M \gamma(h)} e^{-ts/h} s^L \frac{e^{s(T_\chi+\varepsilon)/h}}{s}\,ds  \leq  C_L \int_0^\infty e^{-(t-T_\chi-\varepsilon)s/h} s^{L-1} \,ds  \nonumber \\ 
 & = & C_L \Gamma(L) \left(\frac{h}{t-T_\chi-\varepsilon}\right)^L.
 \label{ZCMM2026new:geometric-area-term}
\end{eqnarray}
Choosing $L>N$ makes the area term $\Oc(h^N)$.  Because the inequality in \eqref{ZCMM2026new:geometric-time-threshold} is strict, choose $\varepsilon>0$ so small that $T_N>T_\chi+\varepsilon+(N+1)/M$.  Then both estimates hold uniformly for every $t\ge T_N$, proving \eqref{ZCMM2026new:geometric-local-decay} with the stated $T_N$ itself. The negative-energy assertion is obtained by reflection exactly as in Section~5.

\end{proof}

\begin{remark}
Theorem~\ref{ZCMM2026new:TheoremGeometricResolvent} replaces the unspecified FBI exponent in Lemma~\ref{ZCMM2025a: lemma-bottom-bound} by the classical connection time $T_\chi$.  For a fixed desired power $N$, choosing $M$ as close as allowed to $C_{\rm rf}$ gives the heuristic optimized threshold
\[
 T_N\gtrsim T_\chi+\frac{N+1}{C_{\rm rf}},
\]
in the sense that every fixed $M<C_{\rm rf}$ is admissible and $M$ may be chosen arbitrarily close to 
$C_{\rm rf}$ before $h\to0$.  No long-time Egorov theorem or second-microlocal calculus is used: the only unitary propagation occurs for the fixed time $T_\chi+\varepsilon$, and the infinite future is handled by the outgoing no-return relation of Lemma~\ref{ZCMM2026new:LemmaOutgoingTail}.
\end{remark}

\noindent
\textit{Acknowledgement}. The authors would like to thank the China Scholarship Council (CSC) 
for the award of a PhD studentship to the first author (CSC grant no.\ 202106740004). The second 
author also acknowledges financial support from the Dr Perry James Browne Sussex Research 
Centre on Mathematics and its Applications.



\appendix

\section{Derivation of the inhomogeneous FBI estimate}
\label{ZCMM2025b:appendix-inhomogeneous-FBI}
The purpose of this appendix is to justify the estimate \eqref{ZCMM2025b:inhomogeneous-FBI-estimate} used in the proof of Lemma~\ref{ZCMM2025a: lemma-bottom-bound}.  We retain a non-zero right-hand side throughout the same $\Sigma/\Sigma^c$ decomposition used in the proof of \cite[Theorem~5.11]{ZCMM2025a}.  As established inputs we use the FBI reduction \cite[(5.9)]{ZCMM2025a}, the remainder bound \cite[(5.10)]{ZCMM2025a}, and the pointwise estimates of \cite[Lemmas~5.8--5.10]{ZCMM2025a}.  In particular, the $\Oc(h)$ estimate for the FBI remainder in \cite[(5.10)]{ZCMM2025a} is taken as an input; the stationary-phase construction of that remainder is not repeated.

Let
\[
 \eta=C_1h\log(1/h),\qquad t_{\rm FBI}=C_2\eta,\qquad  \theta=i\eta,
\]
with the constants chosen as in \cite[Section~5.3]{ZCMM2025a}.  For $z\in\Omega_M^+(h)$, $0<M<C_{\rm rf}$, consider
\begin{equation}
 (\Dl_\theta(h)-z)\ub=\fb.
 \label{ZCMM2025b:appendix-inhomogeneous-equation}
\end{equation}
Initially we may take $\ub\in C_0^\infty(\R^3;\C^4)$ and then extend the result by density in the graph norm of $\Dl_\theta(h)$.  Put
\[
 \wb=e^{-t_{\rm FBI}H/h}\Tb\ub,\qquad
 \gb=e^{-t_{\rm FBI}H/h}\Tb\fb,
 \qquad
 X:=\|\langle\xi\rangle\wb\|_{L^2(\Lambda_{t_{\rm FBI}G})}.
\]
The FBI identity \cite[(5.9)]{ZCMM2025a}, applied to
\eqref{ZCMM2025b:appendix-inhomogeneous-equation}, gives the equality in
$L^2(\Lambda_{t_{\rm FBI}G})$
\begin{equation}
 (\widetilde d_\theta-zI_4)\wb+\Rb_\theta(t_{\rm FBI},h)\wb=\gb.
 \label{ZCMM2025b:appendix-FBI-forced-equation}
\end{equation}
Moreover, \cite[(5.10)]{ZCMM2025a} gives, uniformly for $z$ in the fixed
smaller strip $\Omega_M^+(h)$,
\begin{equation}
 \| \Rb_\theta(t_{\rm FBI},h)\wb\|_{L^2(\Lambda_{t_{\rm FBI}G})}
 \le ChX.
 \label{ZCMM2025b:appendix-Rtheta-bound}
\end{equation}

We now follow the phase-space decomposition used in
\cite[Section~5.3]{ZCMM2025a}.  Let
\[
 \Sigma:=\left\{\alpha\in\Lambda_{t_{\rm FBI}G} \, : \,   \left| \re\big(\widetilde p_{+,\theta}(\alpha)-z\big)\right|
 \le \frac{\langle\xi\rangle}{C_3}\right\}.
\]
Choose an open set $\mathcal X\Subset\Lambda_{t_{\rm FBI}G}$ containing
$\Sigma$ on which the two deformed eigenvalues remain uniformly separated,
and choose a scalar cutoff $\varrho\in C_0^\infty(\mathcal X)$ with
$\varrho=1$ on a neighbourhood of $\Sigma$.  Write
\[
 \wb_\Sigma:=\varrho\wb,\qquad
 \wb_{\Sigma^c}:=(1-\varrho)\wb.
\]
On $\mathcal X$ let $\Pi_{\pm,\theta}$ be the pointwise spectral projectors
of \cite[Lemma~5.7]{ZCMM2025a} and put
\[
 \wb_+:=\Pi_{+,\theta}\wb_\Sigma,
 \qquad
 \wb_-:=\Pi_{-,\theta}\wb_\Sigma.
\]
Thus $\wb_\Sigma=\wb_++\wb_-$.

\subsection*{The elliptic region $\Sigma^c$}
By \cite[Lemma~5.8]{ZCMM2025a},
\[
 \|(\widetilde{\db}_\theta-z \mbf{I}_4)^{-1}\|  \leq C\langle\xi\rangle^{-1}  \qquad\hbox{on }\Sigma^c.
\]
Since $1-\varrho$ is scalar, it commutes with the matrix-valued function
$\widetilde{\db}_\theta-z \mbf{I}_4$.  Hence \eqref{ZCMM2025b:appendix-FBI-forced-equation} gives
\[
 (\widetilde{\db}_\theta-z \mbf{I}_4)\wb_{\Sigma^c}  =(1-\varrho) \gb-(1-\varrho) \Rb_\theta(t_{\rm FBI},h)\wb.
\]
Consequently,
\begin{equation}
 \|\langle\xi\rangle\wb_{\Sigma^c}\| \leq C\big(\|\gb\|+\| \Rb_\theta(t_{\rm FBI},h)\wb\|\big) \leq C\|\gb\|+ChX.
 \label{ZCMM2025b:appendix-Sigmac-estimate}
\end{equation}

\subsection*{The characteristic region $\Sigma$: the positive sheet}
Multiplying \eqref{ZCMM2025b:appendix-FBI-forced-equation} by $\varrho$
and then applying $\Pi_{+,\theta}$ gives, using the exact spectral identities
of \cite[Lemma~5.7]{ZCMM2025a},
\begin{equation}
 (\widetilde p_{+,\theta}-z)\wb_+ =\Pi_{+,\theta}\varrho\gb  -\Pi_{+,\theta}\varrho \Rb_\theta(t_{\rm FBI},h)\wb.
 \label{ZCMM2025b:appendix-plus-equation}
\end{equation}
The positivity estimate used in \cite[(5.15)]{ZCMM2025a} is uniform on any
fixed smaller strip $\Omega_M^+(h)$, $M<C_{\rm rf}$: after reducing $h_0$ if
necessary, there exists $c_*>0$ such that
\begin{equation}
 \left| \im \left\langle  (\widetilde{p}_{+,\theta}-z)\wb_+,\wb_+  \right\rangle\right|  \geq c_*\eta\|\wb_+\|^2.
 \label{ZCMM2025b:appendix-plus-coercivity}
\end{equation}
Taking the imaginary part of the $L^2$ pairing of \eqref{ZCMM2025b:appendix-plus-equation} with $\wb_+$, 
using the boundedness of $\Pi_{+,\theta}$ and $\varrho$, and then using
\eqref{ZCMM2025b:appendix-Rtheta-bound}, we obtain
\begin{equation*}
 \eta\|\wb_+\| \leq C\big(\|\gb\|+\| \Rb_\theta(t_{\rm FBI},h)\wb\|\big) \leq C\|\gb\|+ChX,
\end{equation*}
that is,
\begin{equation}
 \|\wb_+\|  \leq C\eta^{-1} \|\gb\|+C\frac{h}{\eta}X.
 \label{ZCMM2025b:appendix-plus-estimate}
\end{equation}
This is the source of the factor $h/\eta$ in
\eqref{ZCMM2025b:inhomogeneous-FBI-estimate}.

\subsection*{The characteristic region $\Sigma$: the negative sheet}
Similarly, applying $\Pi_{-,\theta}$ gives
\[
 (\widetilde p_{-,\theta}-z)\wb_-  =\Pi_{-,\theta}\varrho\gb  -\Pi_{-,\theta}\varrho \Rb_\theta(t_{\rm FBI},h)\wb.
\]
By \cite[(5.16)]{ZCMM2025a}, there is a constant $c_1>0$, independent of
$h$ and $z\in\Omega_M^+(h)$, such that
$|\widetilde p_{-,\theta}-z|\ge c_1$ on $\supp\varrho$.  Therefore
\begin{equation}
 \|\wb_-\|
 \le C\|\gb\|+ChX.
 \label{ZCMM2025b:appendix-minus-estimate}
\end{equation}

On $\Sigma$ the variable $|\xi|$ is bounded, exactly as in the proof of
\cite[Theorem~5.11]{ZCMM2025a}; hence $\langle\xi\rangle\le C_\Sigma$ on
$\supp\varrho$.  Combining
\eqref{ZCMM2025b:appendix-plus-estimate} and
\eqref{ZCMM2025b:appendix-minus-estimate}, and using $\eta\leq1$ for small
$h$, gives
\begin{equation}
 \|\langle\xi\rangle\wb_\Sigma\| \leq C_\Sigma(\|\wb_+\|+\|\wb_-\|) \leq C\eta^{-1}\|\gb\|   +C\left(\frac{h}{\eta}+h\right)X.
 \label{ZCMM2025b:appendix-Sigma-estimate}
\end{equation}
Together with \eqref{ZCMM2025b:appendix-Sigmac-estimate}, this yields
\begin{equation}
 X\le C\eta^{-1}\|\gb\|
   +C\left(\frac{h}{\eta}+h\right)X,
 \label{ZCMM2025b:appendix-final-FBI-estimate}
\end{equation}
which is precisely \eqref{ZCMM2025b:inhomogeneous-FBI-estimate}.

We emphasize the origin of the two error coefficients.  In the FBI-side
formulation of \cite{ZCMM2025a}, $\varrho$ is a scalar pointwise cutoff and
$\Pi_{\pm,\theta}$ are pointwise matrix projectors.  Thus $\varrho$ commutes
exactly with $\widetilde{\db}_\theta$ and the projectors commute with
$\widetilde{\db}_\theta$ by their spectral identities; no additional commutator
term is created in the preceding argument.  The $\Oc(h)$ error is already contained in $\Rb_\theta(t_{\rm FBI},h)$ through
\eqref{ZCMM2025b:appendix-Rtheta-bound}.  Dividing this $\Oc(h)X$ term by the
positive-sheet coercivity scale $\eta$ produces $(h/\eta)X$ in \eqref{ZCMM2025b:appendix-plus-estimate}; 
the negative-sheet and elliptic estimates retain it as $hX$.  If one instead implements the localization or
the spectral projectors as semiclassical pseudodifferential operators before the FBI reduction, the corresponding commutators are 
$\Oc(h)$ and are absorbed into the same harmless $ChX$ contribution.

Finally,
\[
 \frac{h}{\eta}=\frac{1}{C_1\log(1/h)}\longrightarrow0,
\]
so the last term in \eqref{ZCMM2025b:appendix-final-FBI-estimate} is
absorbed for sufficiently small $h$.  Let
\[
 H_0:=\sup_{\Lambda_{t_{\rm FBI}G}}|H|<\infty
\]
on the deformation region used in the FBI argument.  Since
$t_{\rm FBI}/h=C_1C_2\log(1/h)$ and the FBI transform is an isometry before
inserting the weight, there is $K_0=C_1C_2H_0$ such that
\[
 \|\gb\|\le Ch^{-K_0}\|\fb\|,
 \qquad
 \|\ub\|\le Ch^{-K_0}\|\wb\|\le Ch^{-K_0}X.
\]
Hence, after absorption,
\[
 \|\ub\|
 \le C h^{-2K_0}\eta^{-1}\|\fb\|
 \le C h^{-(2K_0+1)}\|\fb\|.
\]
Thus any integer $N_0>2K_0+1$ is admissible in
\eqref{ZCMM2025b:quantitative-strip-resolvent}.  The estimate extends from
$C_0^\infty$ to the domain $H^1(\R^3;\C^4)$ by graph-norm density.  By
\cite[Theorem~5.11]{ZCMM2025a}, $\Dl_\theta(h)-z$ is injective throughout
the resonance-free strip; since it is Fredholm of index zero, it is
bijective there.  This completes the quantitative step used in
Lemma~\ref{ZCMM2025a: lemma-bottom-bound}.

\end{document}